\documentclass[12pt,letterpaper]{scrartcl}
\pdfoutput=1
\usepackage{preamble}

\title{Modular functors from non-semisimple 3d TFTs}

\author{
	Aaron Hofer\footnote{\normalsize{\texttt{\href{mailto:aaron.hofer@uni-hamburg.de}{aaron.hofer@uni-hamburg.de}}}} \qquad
	Ingo Runkel\footnote{\normalsize{\texttt{\href{mailto:ingo.runkel@uni-hamburg.de}{ingo.runkel@uni-hamburg.de}}}}\\[0.5cm]
	 \\  
	\normalsize\slshape Fachbereich Mathematik, Universit\"{a}t Hamburg,\\
	\normalsize\slshape Bundesstra{\ss}e 55, 20146 Hamburg, Germany
	}

\date{}

\begin{document}
\maketitle
\begin{abstract}
\noindent
Given a not necessarily semisimple modular tensor category $\calc$, we use the corresponding 3d\,TFT defined in \autocite{DGGPR19} to explicitly describe a modular functor as a symmetric monoidal 2-functor from a 2-category of oriented bordisms to a 2-category of finite linear categories. 
This recovers a result by Lyubashenko \autocite{Lyu1996ribbon} obtained via generators and relations.

Pulling back the modular functor for $\calc$ to a $2$-category of bordisms with orientation reversing involution cancels the gluing anomaly, and further pulling back to the original bordism category along a doubling functor leads to the modular functor for the Drinfeld centre $\calZ(\calc)$. 
\end{abstract}

\thispagestyle{empty}

\newpage

\tableofcontents

\newpage
\section{Introduction}
A modular functor is, colloquially speaking, a systematic assignment of mapping class group representations to surfaces which is compatible with gluing along boundaries.
Modular functors originate from the study of $2$-dimensional conformal field theories as solution spaces to certain differential equations \autocite{Segal1988cftmodfunc,MS1989cft,TUY1989CFT,AU2007ModFunc,FBZ2001VOA}.
In this context, the surfaces are equipped with a complex structure, and such modular functors are also called \emph{complex-analytic}.
A second natural source of modular functors are 3-dimensional topological field theories (TFTs), where the vector spaces assigned to surfaces carry mapping class group actions. In this case, the surfaces are just smooth manifolds, and such modular functors are called \emph{topological}. The relation between the two notions is discussed for example in \autocite{BK2001modular,Loo2010WZW}.

\medskip

In this paper, we study topological modular functors obtained from 3-dimensional surgery TFTs. The original class of examples are the surgery TFTs constructed by Reshetikhin and Turaev \autocite{RT1991, Tu16}. Their construction takes a modular fusion category as an input. These categories are finitely semisimple, monoidal, have duals, and a non-degenerate braiding. A generalisation of the resulting modular functor to modular tensor categories, which are still finite but need no longer be semisimple, was given in \autocite{Lyu1996ribbon} by an approach based on generators and relations. The surgery TFT construction itself was generalised to not necessarily semisimple modular tensor categories in \autocite{DGGPR19}, building in particular on \autocite{CGP12,DGP2017}. These TFTs are non-compact in the sense that they are no longer defined on all bordisms \cite{lurie2009classification,Haioun23nonssrel}. However, they include all bordisms necessary to obtain a modular functor, and it was verified in \autocite{DGGPR20} that the resulting mapping class group actions agree with those given in \autocite{Lyu1996ribbon}.
The main goal of this paper is to explicitly construct the modular functor in terms of the non-semisimple $3$d TFT of \autocite{DGGPR19}, including in particular the gluing morphisms. Related results in a slightly different setting were obtained in \autocite{DeRenzi2021nonssext2}, see \Cref{rem:mainthm}.

\medskip

A construction of Lyubashenko's modular functor for Drinfeld centres based on string-nets is given in \autocite{MSWY23drinmodfunc}, see also \autocite{KST2023framedstring} for a framed version. There it is also shown how to naturally extend the modular functor to an open-closed bordism category.

Our main motivation for a TFT-based construction comes from a possible future application to a correspondence between non-semisimple 2d\,CFT and 3d\,TFT in the spirit of \autocite{FFFS2002,FRSI,FFRS2007defects}.

\subsection*{Statement of the main results}

The variant of a modular functor that we will work with is defined to be a symmetric monoidal $2$-functor. The source category is the bordism $(2,1)$-category $\bord_{2+\epsilon,2,1}^{\mathrm{\chi}}$ consisting of closed oriented $1$-dimensional manifolds, $2$-dimensional oriented bordisms, and isotopy classes of orientation preserving diffeomorphisms, together with extra decorations necessary to compensate for a gluing anomaly indicated by the superscript $\chi$. 
The target 2-category is $\cat{P}\mathrm{rof}_{\mathbbm{k}}^{\coend \mathrm{ex}}$ consisting of finite linear categories, left exact profunctors, and natural transformations, see \Cref{sec:modfunc} for details. This corresponds to the definition used in \autocite{FSY2023RCFTstring1}, related notions were studied in  \autocite{Tillmann1998modfunc} and \autocite[Ch.\,6]{KL2001nsstqft}.

Our first main result is
\begin{thm}[{\Cref{thm:chiralmf}}]
    For every modular tensor category $\calc$, the 3d TFT
    \begin{align}
    \widehat{\rmV}_{\calc} \colon \widehat{\bord}_{3,2}^{\chi}(\calc) \to \vect
    \end{align}
    of \autocite{DGGPR19} induces a symmetric monoidal $2$-functor
    \begin{align}
    \mathrm{Bl}_\calc^{\chi}\colon\bord_{2+\epsilon,2,1}^{\mathrm{\chi}} \to \cat{P}\mathrm{rof}_{\mathbbm{k}}^{\coend \mathrm{ex}}.
    \end{align}
\end{thm}
Here $\widehat{\bord}_{3,2}^{\chi}(\calc)$ is a subcategory of the bordism category $\bord_{3,2}^{\chi}(\calc)$ with $\calc$ labelled ribbon graphs and extra decorations necessary to compensate a gluing anomaly, see \Cref{ssec:adbordcat}.
We note that in our setting actually $\cat{P}\mathrm{rof}_{\mathbbm{k}}^{\coend \mathrm{ex}} \simeq \coend \mathrm{ex}$ (see \Cref{ssec:chiral-bord-prof}), so we could have equally well defined modular functors with $\coend \mathrm{ex}$ as target. However, the present formulation is better adapted to the way the TFT is used to define $\mathrm{Bl}_\calc^{\chi}$ on $1$- and $2$-morphisms.

\medskip

The main technical contribution in this paper is the realisation of the morphisms needed for the gluing of surfaces by evaluating $\widehat{\rmV}_{\calc}$ on certain $3$-dimensional bordisms: 
\begin{prop}[{\Cref{prop:gluingfunc}}]
    Let $\Sigma$ be a
    surface with at least one incoming and one
    outgoing boundary component, and let $\Sigma_{\mathrm{gl}}$ be the surface obtained from gluing these boundaries. Then there is a natural isomorphism
    \begin{align}\label{eq:gluingiso}
        \mathrm{Bl}^{\chi}_{\calc}(\Sigma_{\mathrm{gl}}) \cong \oint^{X\in \calc }\! \mathrm{Bl}^{\chi}_{\calc}(\Sigma)(X,X).
    \end{align}
   induced by a $3$-dimensional bordism.
\end{prop}
To be more precise, the isomorphism (\ref{eq:gluingiso}) is induced by
the dinatural family 
\begin{align}
\mathrm{Bl}^{\chi}_{\calc}(\Sigma)(X,X) \to \mathrm{Bl}^{\chi}_{\calc}(\Sigma_{\mathrm{gl}})
\end{align}
which is obtained by evaluating the TFT $\widehat{\rmV}_{\calc}$ on a family of bordisms $M_X$ from $\underline{\Sigma}$ to $\underline{\Sigma_{\mathrm{gl}}}$ in $\widehat{\bord}_{3,2}^{\chi}(\calc)$, where $\underline{\Sigma}$ and $\underline{\Sigma_{\mathrm{gl}}}$ are obtained from $\Sigma$ and $\Sigma_{\mathrm{gl}}$ by replacing boundary components with marked disks, respectively, see \Cref{ssec:modfunc3mor} for details. Locally $M_X$ can be visualised as
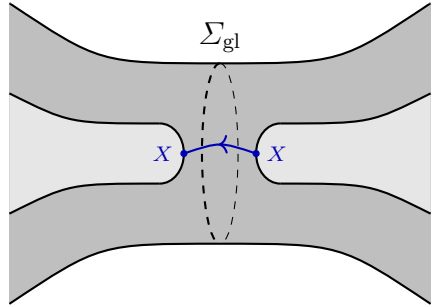
\begin{align*}
       \begin{tikzpicture}[scale=0.8]
       \begin{scope}
       \fill[lightgray] (7.5,-10) rectangle (14.5,-5);
       \fill[white] (7.4,-4.98) -- (11,-6.02) -- (14.6,-4.98) -- cycle;
       \fill[white] (11,-6) .. controls (9,-6) .. (7.5,-5)
        (11,-6) .. controls (13,-6) .. (14.5,-5)
        (14.5,-5) -- (7.5,-5);
       \draw[thick] (11,-6) .. controls (9,-6) .. (7.5,-5)
        (11,-6) .. controls (13,-6) .. (14.5,-5);
       \fill[white] (7.4,-10.02) -- (11,-8.98) -- (14.6,-10.02) -- cycle;
       \fill[white] (11,-9) .. controls (9,-9) .. (7.5,-10)
        (11,-9) .. controls (13,-9) .. (14.5,-10)
        (14.5,-5) -- (7.5,-5);
       \draw[thick] (11,-9) .. controls (9,-9) .. (7.5,-10)
        (11,-9) .. controls (13,-9) .. (14.5,-10);
       \fill[lightgray!40] (10.02,-7.01) -- (10.02,-7.99) -- (7.5,-8.48) -- (7.5,-6.52) --cycle;
       \filldraw[fill=lightgray!40,thick] (10,-7) arc (90:-90:0.4 and 0.5);
       \filldraw[fill=lightgray,thick]
        (10,-7) .. controls (8.5,-7) .. (7.5,-6.5)
        (10,-8) .. controls (8.5,-8) .. (7.5,-8.5);
       \fill[lightgray!40] (11.98,-7.01) -- (11.98,-7.99) -- (14.5,-8.48) -- (14.5,-6.52) --cycle;
        \filldraw[fill=lightgray!40,thick] (12,-7) arc (90:270:0.4 and 0.5);
       \filldraw[fill=lightgray,thick]
        (12,-7) .. controls (13.5,-7) .. (14.5,-6.5) 
        (12,-8) .. controls (13.5,-8) .. (14.5,-8.5);
        \filldraw[string-blue] (10.4,-7.5) circle (0.05)
        (11.6,-7.5) circle (0.05);
        \begin{scope}[very thick,decoration={
                      markings,
                      mark=at position 0.52 with {\arrow{>}}}
                     ] 
            \draw[string-blue,thick,postaction={decorate}] (11.6,-7.5) .. controls (11,-7.3) .. (10.4,-7.5);
        \end{scope}
        \draw[dashed,thick] (11,-6) arc (90:270:0.3 and 1.5);
        \draw[dashed] (11,-6) arc (90:-90:0.3 and 1.5);
       \node at (10.4,-7.5) [string-blue,left] {\scriptsize $X$};
       \node at (11.6,-7.5) [string-blue,right] {\scriptsize $X$};
       \node at (11,-6) [above] {$\Sigma_{\mathrm{gl}}$};
       \end{scope}
   \end{tikzpicture} 
\end{align*}
where the dashed line indicates where the gluing was performed. 

We refer to the modular functor defined on the 2-category $\bord_{2+\epsilon,2,1}^{\mathrm{\chi}}$ which includes the anomaly compensating decorations as \emph{chiral}. Another interesting source 2-category is $\bord_{2+\epsilon,2,1}^{\circlearrowleft}$, which consists of oriented bordisms equipped with an orientation reversing involution, but which does not carry the extra decorations for anomaly cancellation. It contains open-closed bordisms as a full subcategory, see Section\,\ref{sec:anomfreemod}.

There is a symmetric monoidal 2-functor $U\colon \bord^{\circlearrowleft}_{2+\epsilon,2,1} \to \bord^{\chi}_{2+\epsilon,2,1}$ which assigns decorations coming from the involution in such a way that pullback along $U$ produces an ``anomaly-free'' modular functor \Cref{ssec:bordinv}. In the converse direction, there is a symmetric monoidal 2-functor $\widehat{(-)} \colon \bord_{2+\epsilon,2,1}^{\mathrm{\chi}} \to \bord^{\circlearrowleft}_{2+\epsilon,2,1}$ obtained by sending any oriented manifold $M$ to its orientation double $\widehat{M} := M \disjun -M$ with the natural orientation reversing involution.
Applying these constructions to the chiral modular functor $\mathrm{Bl}_\calc^{\chi}$ obtained in the above theorem gives a geometric relation between $\calc$ and its Drinfeld centre $\calZ(\calc) \simeq \calc \boxtimes \overline{\calc}$:

\begin{prop}[{\Cref{prop:condi}}]
    Let $\calc$ be a modular tensor category. There exists a symmetric monoidal $2$-natural isomorphism filling the following diagram of symmetric monoidal $2$-functors
\begin{align}
\begin{tikzcd}[ampersand replacement=\&]
	{\bord_{2+\epsilon,2,1}^{\mathrm{\chi}}} \& {\bord^{\circlearrowleft}_{2+\epsilon,2,1}} \& {\bord_{2+\epsilon,2,1}^{\mathrm{\chi}}} \\
	\\
	\&\& {\cat{P}\mathrm{rof}_{\mathbbm{k}}^{\coend \mathrm{ex}}}
	\arrow[""{name=0, anchor=center, inner sep=0}, "{\mathrm{Bl}_{\calZ(\calc)}^{\chi}}"', from=1-1, to=3-3]
	\arrow["{\widehat{(-)}}", from=1-1, to=1-2]
	\arrow["U", from=1-2, to=1-3]
	\arrow["{\mathrm{Bl}_\calc^{\chi}}", from=1-3, to=3-3]
	\arrow["\cong"', shorten >=8pt, Rightarrow, from=1-3, to=0]
\end{tikzcd}
\quad .
\end{align}
\end{prop}
Finally, by factoring the orientation doubling $2$-functor 
$\widehat{(-)}$ 
over the $2$-category of open-closed oriented bordisms we realise $\mathrm{Bl}_{\calZ{(\calc)}}^{\chi}$ as the closed sector of a full modular functor. To prove the above Proposition we study the behaviour of the TFT $ \widehat{\rmV}_{\calc}$ under orientation reversal as well as the Deligne tensor product, see \Cref{ssec:tftorideligne} for detailed statements.
\medskip

The rest of this paper is organised as follows. In \Cref{sec:algprelim} we recall the algebraic ingredients used throughout this paper, including finite ribbon categories as well as some results on coends in functor categories. In Section \ref{sec:3dTFT} we briefly review the construction as well as some properties of the $3$d TFTs of \autocite{DGGPR19}. In \Cref{sec:modfunc} we explicitly construct a chiral modular functor from these $3$d TFTs
and prove our main result \Cref{thm:chiralmf}. Finally, in \Cref{sec:anomfreemod} we extend the chiral modular functors to a bordism $2$-category with orientation reversing involution and explain the relation to the Drinfeld centre and give a topological proof of its monadicity (\Cref{prop:freeforgetadj}). 

\subsection*{Acknowledgements}
We thank Marco De Renzi, Nils Carqueville, Hannes Knötzele, Lukas Woike, and Yang Yang for useful discussions as well as helpful comments on an earlier draft.
AH and IR are partially supported by the Deutsche Forschungsgemeinschaft (DFG, German Research Foundation) under Germany`s Excellence Strategy - EXC 2121 ``Quantum Universe" - 390833306, and by the Collaborative Research Centre CRC 1624 ``Higher structures, moduli spaces and integrability'' - 506632645. 

\subsection*{Conventions}
Throughout this article, $\calc$ will be a finite ribbon category.
Moreover, we will appeal to the standard coherence results and assume $\calc$ to be strictly monoidal and strictly pivotal. We will always work over an algebraically closed field $\mathbbm{k}$ of characteristic zero. Unless otherwise noted, functors between abelian (and linear) categories will always be assumed to be additive (and linear).

Every manifold we will consider will be oriented, and every diffeomorphism will be orientation preserving unless explicitly stated otherwise. For any manifold $M$ we will denote the manifold with reversed orientation by $-M$. The interval $[0,1]$ will be denoted by $I$ and the unit circle by $S^1$. Finally, by a closed manifold we mean a compact manifold without boundary.

\section{Algebraic preliminaries}\label{sec:algprelim}
In this section, we collect definitions and results related to ribbon categories and coends that will be needed for the construction. For a detailed exposition, we refer to \autocite{BK2001modular,EGNO,TV,Loregian2021coendcalc,FS16coendsCFT}.

\subsection{Finite ribbon categories}\label{ssec:ribboncats}

A linear category is called \emph{finite} if it is equivalent, as a linear category, to the category $A$-$\Mod$
of finite dimensional modules of some finite dimensional algebra $A$. In particular, a finite linear category is abelian, every object has a projective cover, and its Hom sets are finite dimensional vector spaces. The complete intrinsic definition of finite linear categories can be found in \autocite[Sec.\,1.8]{EGNO}. 
By a \emph{finite tensor category} we mean a finite linear category which is in addition a rigid monoidal category such that the monoidal product $\otimes$ is bilinear and the monoidal unit $\unit$ is simple. The monoidal product is exact in both arguments by rigidity \autocite[Prop.\,4.2.1]{EGNO}.

A \emph{finite ribbon category} $\calc$ is a finite tensor category which is also ribbon. We will denote a choice of a set of representatives of isomorphism classes of simple objects by $\Irr \subset \calc$. 
We will employ the following conventions for structure morphisms in $\calc$. Every object $X$ in $\calc$ has a two-sided dual $X^{*}$, with duality morphisms denoted by:
\begin{equation}
   \begin{aligned}
    \mathrm{ev}_X &\colon X^{*} \otimes X \to \unit, \qquad \mathrm{coev}_X \colon \unit \to X \otimes X^{*}, 
    \\
    \widetilde{\mathrm{ev}_X} &\colon X \otimes X^{*} \to \unit, \qquad \widetilde{\mathrm{coev}_X} \colon \unit \to X^{*} \otimes X. 
\end{aligned} 
\end{equation}
The components of the braiding and twist isomorphisms will be denoted by
\begin{align}
    \beta_{X,Y} \colon X \otimes Y \to Y \otimes X, \qquad \theta_X \colon X \to X.
\end{align}
The twist $\theta$ satisfies
\begin{align}
    \theta_{X\otimes Y} = \beta_{Y,X} \circ \beta_{X,Y} \circ (\theta_X \otimes \theta_Y), \quad \mathrm{and} \quad 
    (\theta_X)^* = \theta_{X^*}
\end{align}
for all $X,Y \in \calc$. A direct computation shows that the twist $\theta$ endows $\calc$ with a pivotal structure, which is even spherical \autocite[Sec.\,8.10]{EGNO}. 
In diagrammatic notation these structural morphisms will be represented as
{\allowdisplaybreaks\begin{align}
        \lev_X &= 
    \begin{tikzpicture}[baseline]
        \draw[string-blue,thick] (0.5,0) arc (0:180:0.5cm);
        \draw[string-blue,thick,->] (0.5,0) arc (0:92:0.5cm);
        \node at (-0.5,0) [string-blue,thick,below] {\scriptsize $X^*$};
        \node at (0.5,0) [string-blue,thick,below] {\scriptsize $X$};
    \end{tikzpicture} 
    & \lcoev_X &= 
    \begin{tikzpicture}[baseline]
        \draw[string-blue,thick] (0.5,0) arc (0:-180:0.5cm);
        \draw[string-blue,thick,->] (0.5,0) arc (0:-92:0.5cm);
        \node at (0.5,0) [string-blue,thick,above] {\scriptsize $X^*$};
        \node at (-0.5,0) [string-blue,thick,above] {\scriptsize $X$};
    \end{tikzpicture} 
    & \beta_{X,Y} &= 
    \begin{tikzpicture}[baseline]
        \draw[string-blue, thick] (-0.5,-0.5) node[string-blue,below] {\scriptsize $X$} -- (0.5,0.5) node[string-blue,above] {\scriptsize $X$};
        \draw[string-blue, thick] (0.5,-0.5) node[string-blue,below] {\scriptsize $Y$} -- (0.15,-0.15);
        \draw[string-blue, thick] (-0.15,0.15) -- (-0.5,0.5) node[string-blue,above] {\scriptsize $Y$};
\end{tikzpicture}
        \nonumber \\
        \rev_X &= 
        \begin{tikzpicture}[baseline]
            \draw[string-blue,thick] (0.5,0) arc (0:180:0.5cm);
            \draw[string-blue,thick,->] (-0.5,0) arc (180:87:0.5cm);
            \node at (-0.5,0) [string-blue,thick,below] {\scriptsize $X$};
            \node at (0.5,0) [string-blue,thick,below] {\scriptsize $X^{*}$};
        \end{tikzpicture}
        & \rcoev_X &= 
       \begin{tikzpicture}[baseline]
            \draw[string-blue,thick] (0.5,0) arc (0:-180:0.5cm);
            \draw[string-blue,thick,->] (-0.5,0) arc (-180:-87:0.5cm);
            \node at (0.5,0) [string-blue,thick,above] {\scriptsize $X$};
            \node at (-0.5,0) [string-blue,thick,above] {\scriptsize $X^{*}$};
        \end{tikzpicture} 
        & \theta_{X} &= 
        \begin{tikzpicture}[baseline]
        \draw[string-blue, thick] (0,-0.2) -- (0,0.-0.5) node[string-blue,below] {\scriptsize $X$} ;
         \draw[string-blue, thick] (0,0.2) --(0,0.5) node[string-blue,above] {\scriptsize $X$};
        \draw[string-blue, thick] (0,-0.2) arc (180:90:0.15 and 0.3);
        \draw[string-blue, thick] (0,0.2) arc (180:210:0.15 and 0.3);
        \draw[string-blue,thick] (0.15,0.1) arc (90:-90:0.1 and 0.1);
        \draw[string-blue, thick] (0.15,-0.1) arc (-90:-120:0.15 and 0.3);
\end{tikzpicture}
\end{align}}
Note that we read such diagrams from the bottom to the top.
We will use the same conventions as \autocite{TV} and denote with $\overline{\calc}$ the same underlying pivotal
monoidal category, but equipped with the inverse braiding and twist, i.e.\
\begin{equation}
    \begin{aligned}
        \overline{\beta}_{X,Y} := \beta_{Y,X}^{-1} \colon X \otimes Y \to Y\otimes X \qquad \qquad \overline{\theta}_{X} := \theta_{X}^{-1} \colon X \to X
    \end{aligned}
\end{equation}
see \autocite[Sec.\,1.2.2 \& Ex.\ 3.1.7]{TV}. We will call $\overline{\calc}$ the \emph{mirrored} category of $\calc$.

We will now discuss a specific Hopf algebra in $\calc$ which will be a fundamental ingredient throughout the rest of this article. 
First, recall that $\calc$ admits an inner Hom functor $\calc^{\mathrm{op}} \times \calc \to \calc$ which sends $(X,Y) \in \calc^{\mathrm{op}} \times \calc$ to $X^{*} \otimes Y$. Due to our finiteness assumptions the coend of this functor exists 
and can be explicitly described as a cokernel, see \autocite[Ch.\,5]{KL2001nsstqft} for details. We will denote this coend as.
\begin{align}
    \coend := \int^{X \in \calc}\! X^{*} \otimes X
\end{align}
with universal dinatural transformation
\begin{align}
    \iota_X \colon X^* \otimes X \to \coend.
\end{align}
The object $\coend$ naturally carries the structure of a Hopf algebra with Hopf pairing in the braided monoidal category $\calc$. For Hopf algebras in braided monoidal categories we will follow the convention of \autocite[Ch.\,6]{TV}. We will denote its structure morphisms as follows:
\begin{equation}\label{eq:coendhopf}
\begin{aligned}
    (\mathrm{Product}) \; &\mu \colon \coend \otimes \coend \to \coend, \hfill &&&\mathrm{(Unit)} \;  &\eta \colon \unit \to \coend,\\
    (\mathrm{Coproduct}) \; &\Delta \colon \coend \to \coend\otimes \coend, \hfill &&&\mathrm{(Counit)}\;  &\epsilon \colon \coend \to \unit, \\
    (\mathrm{Antipode}) \; &S \colon \coend \to \coend, \hfill &&&\mathrm{(Pairing)} \;  &\omega \colon \coend\otimes \coend \to \unit. 
\end{aligned}
\end{equation}
These morphisms are induced from the universal property of the coend as shown in \Cref{fig:coend-Hopf-def}.
If we do not assume $\calc$ to be strictly pivotal the canonical isomorphisms $(X\otimes Y)^* \cong Y^* \otimes X^*$ for the product, $\unit^* \otimes \unit \cong \unit$ for the unit, and $X \cong X^{**}$ for the antipode are needed, see \autocite[Sec.\,6.4\,\&\,6.5]{TV} for details. 
Exchanging the over and under braidings in the definition of $\omega$ gives a pairing $\overline{\omega} \colon \coend \otimes \coend \to \unit $ which satisfies 
\begin{equation}\label{eq:hopfpairingmirror}
    \omega \circ (S\otimes \id_\coend)= \overline{\omega} = \omega \circ (\id_\coend \otimes S)
\end{equation}
see \autocite[Eq.\,5.2.8]{KL2001nsstqft}.

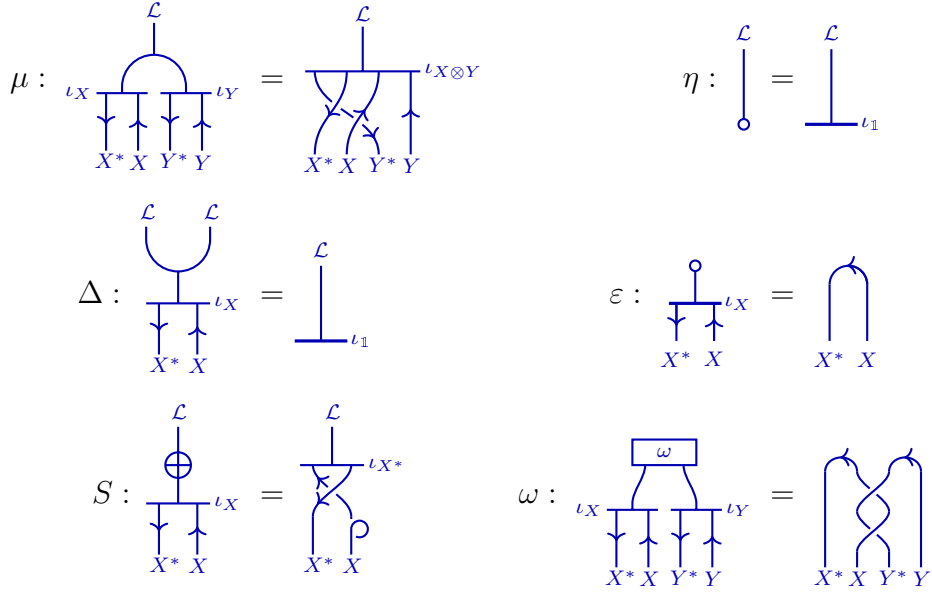
\begin{figure}[tb]
\begin{equation*}
    \begin{aligned}
      \mu : \begin{tikzpicture}[scale=0.85,baseline=0cm]
        \draw[string-blue,thick]
        (0.5,0) -- (0.5,-0.1)
        (-0.5,0) -- (-0.5,-0.1)
        (0.5,0) arc (0:180:0.5 and 0.5)
        (0,0.5) -- (0,1)
        (0.1,-0.1) -- (0.9,-0.1)
        (-0.1,-0.1) -- (-0.9,-0.1);
       \begin{scope}[decoration={
                      markings,
                      mark=at position 0.55 with {\arrow{>}}}
                     ] 
            \draw[string-blue,thick,postaction={decorate}] (-0.75,-0.1) -- (-0.75,-1);
            \draw[string-blue,thick,postaction={decorate}] (-0.25,-1) -- (-0.25,-0.1);
            \draw[string-blue,thick,postaction={decorate}] (0.75,-1) -- (0.75,-0.1);
            \draw[string-blue,thick,postaction={decorate}] (0.25,-0.1) -- (0.25,-1);
        \end{scope}
       \node at (-0.8,-0.1) [string-blue,left] {\scriptsize $\iota_X$};
       \node at (0.8,-0.1) [string-blue,right] {\scriptsize $\iota_Y$};
       \node at (-0.68,-0.875) [string-blue,below] {\scriptsize $X^*$};
       \node at (-0.23,-0.9) [string-blue,below] {\scriptsize $X$};
        \node at (0.32,-0.875) [string-blue,below] {\scriptsize $Y^*$};
       \node at (0.77,-0.9) [string-blue,below] {\scriptsize $Y$};
       \node at (0,0.915) [string-blue,above] {\scriptsize $\coend$};
   \end{tikzpicture}  
   &= \begin{tikzpicture}[scale=0.85,baseline=0.05cm]
        \draw[string-blue,thick]
        (0,0.3) -- (0,1)
        (0.9,0.3) -- (-0.9,0.3);
        \begin{scope}[decoration={
                      markings,
                      mark=at position 0.85 with {\arrow{>}}}
                     ] 
        \draw[string-blue,thick,postaction={decorate}] (-0.75,0.3) .. controls (-0.75,-0.3) and (0.25,-0.4) .. (0.25,-1);
        \end{scope}
       \begin{scope}[decoration={
                      markings,
                      mark=at position 0.55 with {\arrow{>}}}
                     ] 
            \draw[string-blue,thick,postaction={decorate}] (0.75,-1) -- (0.75,0.3);
            \fill[white] (-0.23,-0.55) -- (-0.13,-0.4) -- (0.05,-0.45) -- (-0.05,-0.6) -- cycle;
            \begin{scope}[xshift=-0.3cm,yshift=0.3cm]
                \fill[white] (-0.23,-0.55) -- (-0.13,-0.4) -- (0.05,-0.45) -- (-0.05,-0.6) -- cycle;
            \end{scope}
            \draw[string-blue,thick,postaction={decorate}] (-0.25,0.3) .. controls (-0.25,-0.3) and (-0.75,-0.4) .. (-0.75,-1);
            \draw[string-blue,thick,postaction={decorate}] (-0.25,-1) .. controls (-0.25,-0.4) and (0.25,-0.3) .. (0.25,0.3);
        \end{scope}
       \node at (0.8,0.3) [string-blue,right] {\scriptsize $\iota_{X\otimes Y}$};
       \node at (-0.68,-0.875) [string-blue,below] {\scriptsize $X^*$};
       \node at (-0.23,-0.9) [string-blue,below] {\scriptsize $X$};
        \node at (0.32,-0.875) [string-blue,below] {\scriptsize $Y^*$};
       \node at (0.77,-0.9) [string-blue,below] {\scriptsize $Y$};
       \node at (0,0.915) [string-blue,above] {\scriptsize $\coend$};
   \end{tikzpicture}   &
   \hfill \eta : \begin{tikzpicture}[baseline]
        \draw[string-blue,thick] (0.25,-0.5) -- (0.25,0.5);
        \filldraw[draw=string-blue, thick,fill=white] (0.25,-0.5) circle (0.07);
        \node at (0.2475,0.475) [string-blue,above] {\scriptsize $\coend$};
   \end{tikzpicture}  
   &= \begin{tikzpicture}[baseline]
        \draw[string-blue,thick] (0.25,-0.5) -- (0.25,0.5);
        \draw[string-blue,very thick] (-0.1,-0.5) -- (0.6,-0.5);
        \node at (0.2475,0.475) [string-blue,above] {\scriptsize $\coend$};
        \node at (0.5,-0.5) [string-blue,right] {\scriptsize $\iota_{\unit}$};
   \end{tikzpicture}   \\
   \Delta : \begin{tikzpicture}[scale=0.85,baseline=0cm]
        \draw[string-blue,thick]
        (-1.6,1.2) -- (-1.6,1)
        (-2.6,1.2) -- (-2.6,1)
        (-1.6,1) arc (0:-180:0.5 and 0.5)
        (-2.1,0.5) -- (-2.1,0)
        (-1.6,0) -- (-2.6,0);
       \begin{scope}[decoration={
                      markings,
                      mark=at position 0.55 with {\arrow{>}}}
                     ] 
            \draw[string-blue,thick,postaction={decorate}] (-1.8,-0.8) -- (-1.8,0);
            \draw[string-blue,thick,postaction={decorate}] (-2.4,0) -- (-2.4,-0.8);
        \end{scope}
       \node at (-1.7,0) [string-blue,right] {\scriptsize $\iota_X$};
       \node at (-1.78,-0.7) [string-blue,below] {\scriptsize $X$};
       \node at (-2.3,-0.675) [string-blue,below] {\scriptsize $X^*$};
       \node at (-1.575,1.15) [string-blue,above] {\scriptsize $\coend$};
       \node at (-2.575,1.15) [string-blue,above] {\scriptsize $\coend$};
   \end{tikzpicture}   
   &= \begin{tikzpicture}[baseline]
        \draw[string-blue,thick] (0.25,-0.5) -- (0.25,0.5);
        \draw[string-blue,very thick] (-0.1,-0.5) -- (0.6,-0.5);
        \node at (0.2475,0.475) [string-blue,above] {\scriptsize $\coend$};
        \node at (0.5,-0.5) [string-blue,right] {\scriptsize $\iota_{\unit}$};
   \end{tikzpicture}   
   & \hfill \epsilon : \begin{tikzpicture}[baseline]
        \draw[string-blue,thick] (0.25,0) -- (0.25,0.5);
       \draw[string-blue,thick] (0,-0.5) -- (0,0);
       \draw[string-blue,thick,->] (0,0) -- (0,-0.3);
       \draw[string-blue,thick] (0.5,-0.5) -- (0.5,0);
       \draw[string-blue,thick,->] (0.5,-0.5) -- (0.5,-0.2);
        \draw[string-blue,very thick] (-0.1,0) -- (0.6,0);
        \filldraw[draw=string-blue, thick,fill=white] (0.25,0.5) circle (0.07);
        \node at (0.5,0) [string-blue,right] {\scriptsize $\iota_{X}$};
        \node at (0.5,-0.5) [string-blue,below] {\scriptsize $X$};
        \node at (0,-0.5) [string-blue,below] {\scriptsize $X^*$};
   \end{tikzpicture}   
   &= \begin{tikzpicture}[baseline]
        \draw[string-blue,thick] (0.5,0.25) arc (0:180:0.25cm);
        \draw[string-blue,thick,->] (0.5,0.25) arc (0:95:0.25cm);
        \draw[string-blue,thick] (0.5,0.25) -- (0.5,-0.5);
        \draw[string-blue,thick] (0,0.25) -- (0,-0.5);
        \node at (0,-0.5) [string-blue,thick,below] {\scriptsize $X^*$};
        \node at (0.5,-0.5) [string-blue,thick,below] {\scriptsize $X$};
    \end{tikzpicture}
   \\
   S :\begin{tikzpicture}[scale=0.85,baseline=0cm]
        \draw[string-blue,thick]
        (-2.1,0.6) circle (0.2)
        (-2.1,1.2) -- (-2.1,0)
        (-1.9,0.6) -- (-2.3,0.6)
        (-1.6,0) -- (-2.6,0);
       \begin{scope}[decoration={
                      markings,
                      mark=at position 0.55 with {\arrow{>}}}
                     ] 
            \draw[string-blue,thick,postaction={decorate}] (-1.8,-0.8) -- (-1.8,0);
            \draw[string-blue,thick,postaction={decorate}] (-2.4,0) -- (-2.4,-0.8);
        \end{scope}
       \node at (-1.7,0) [string-blue,right] {\scriptsize $\iota_X$};
       \node at (-1.78,-0.7) [string-blue,below] {\scriptsize $X$};
       \node at (-2.3,-0.675) [string-blue,below] {\scriptsize $X^*$};
       \node at (-2.085,1.15) [string-blue,above] {\scriptsize $\coend$};
   \end{tikzpicture} 
   &= \begin{tikzpicture}[scale=0.85,baseline=0cm]
        \draw[string-blue,thick]
        (-2.1,1.2) -- (-2.1,0.6)
        (-1.6,0.6) -- (-2.6,0.6);
       \begin{scope}[decoration={
                      markings,
                      mark=at position 0.33 with {\arrow{<}}}
                     ]
            \draw[string-blue,thick] (-1.54,-0.4) arc (0:-115:0.13 and 0.13);
            \draw[string-blue,thick] (-1.8,-0.8) -- (-1.8,-0.2);
            \fill[white] (-1.85,-0.4) rectangle (-1.75,-0.25);
            \draw[string-blue,thick] (-1.8,-0.4) arc (180:0:0.13 and 0.13);
            \draw[string-blue,thick] (-2.4,-0.2) -- (-2.4,-0.8);
            \draw[string-blue,thick,postaction={decorate}] (-2.4,0.6) .. controls (-2.4,0.3) and (-1.8,0.1).. (-1.8,-0.2);
            \begin{scope}[xshift=-2cm,yshift=0.7cm]
                \fill[white] (-0.25,-0.55) -- (-0.13,-0.4) -- (0.05,-0.45) -- (-0.07,-0.6) -- cycle;
            \end{scope} 
            \draw[string-blue,thick,postaction={decorate}] (-2.4,-0.2) .. controls (-2.4,0.1) and (-1.8,0.3) .. (-1.8,0.6);
        \end{scope}
       \node at (-1.7,0.6) [string-blue,right] {\scriptsize $\iota_{X^*}$};
       \node at (-1.78,-0.7) [string-blue,below] {\scriptsize $X$};
       \node at (-2.3,-0.675) [string-blue,below] {\scriptsize $X^*$};
       \node at (-2.085,1.15) [string-blue,above] {\scriptsize $\coend$};
   \end{tikzpicture}  
   & \hfill \omega : \begin{tikzpicture}[scale=0.85,baseline=0cm]
        \filldraw[fill=white,draw=string-blue,thick] (-0.5,0.6) rectangle (0.5,1);
        \draw[string-blue,thick]
        (0.5,0) -- (0.5,-0.1)
        (-0.5,0) -- (-0.5,-0.1)
        (0.1,-0.1) -- (0.9,-0.1)
        (-0.1,-0.1) -- (-0.9,-0.1)
        (0.5,0) .. controls (0.5,0.2) and (0.3,0.4) .. (0.3,0.6)
        (-0.5,0) .. controls (-0.5,0.2) and (-0.3,0.4) .. (-0.3,0.6);
       \begin{scope}[decoration={
                      markings,
                      mark=at position 0.55 with {\arrow{>}}}
                     ] 
            \draw[string-blue,thick,postaction={decorate}] (-0.75,-0.1) -- (-0.75,-1);
            \draw[string-blue,thick,postaction={decorate}] (-0.25,-1) -- (-0.25,-0.1);
            \draw[string-blue,thick,postaction={decorate}] (0.75,-1) -- (0.75,-0.1);
            \draw[string-blue,thick,postaction={decorate}] (0.25,-0.1) -- (0.25,-1);
        \end{scope}
       \node at (-0.8,-0.1) [string-blue,left] {\scriptsize $\iota_X$};
       \node at (0.8,-0.1) [string-blue,right] {\scriptsize $\iota_Y$};
       \node at (-0.68,-0.875) [string-blue,below] {\scriptsize $X^*$};
       \node at (-0.23,-0.9) [string-blue,below] {\scriptsize $X$};
        \node at (0.32,-0.875) [string-blue,below] {\scriptsize $Y^*$};
       \node at (0.77,-0.9) [string-blue,below] {\scriptsize $Y$};
       \node at (0,0.8) [string-blue] {\scriptsize $\omega$};
   \end{tikzpicture} 
   &= \begin{tikzpicture}[scale=0.85,baseline=0cm]
       \begin{scope}[decoration={
                      markings,
                      mark=at position 0.52 with {\arrow{>}}}
                     ] 
            \draw[string-blue,thick] (-0.75,0.5) -- (-0.75,-1);
            \draw[string-blue,thick] (-0.25,-1) -- (-0.25,-0.75);
            \draw[string-blue,thick] (0.75,-1) -- (0.75,0.5);
            \draw[string-blue,thick] (0.25,-0.75) -- (0.25,-1);
            \draw[string-blue,thick,postaction={decorate}] (0.75,0.5) arc (0:180:0.25 and 0.22);
            \draw[string-blue,thick,postaction={decorate}] (-0.25,0.5) arc (0:180:0.25 and 0.22);
            \draw[string-blue,thick] (0.25,-0.75) .. controls (0.25,-0.55) and (-0.25,-0.325) .. (-0.25,-0.125);
            \begin{scope}[xshift=0.1cm,yshift=0.05cm]
                \fill[white] (-0.2,-0.5) -- (-0.1,-0.4) -- (0,-0.5) -- (-0.1,-0.6) -- cycle;
            \end{scope}
            \draw[string-blue,thick] (-0.25,-0.75) .. controls (-0.25,-0.55) and (0.25,-0.325) .. (0.25,-0.125);
            \draw[string-blue,thick] (-0.25,0.5) .. controls (-0.25,0.3) and (0.25,0.075) .. (0.25,-0.125);
            \begin{scope}[xshift=0.1cm,yshift=0.685cm]
                \fill[white] (-0.2,-0.5) -- (-0.1,-0.4) -- (0,-0.5) -- (-0.1,-0.6) -- cycle;
            \end{scope}
            \draw[string-blue,thick] (0.25,0.5) .. controls (0.25,0.3) and (-0.25,0.075) .. (-0.25,-0.125);
        \end{scope}
       \node at (-0.68,-0.875) [string-blue,below] {\scriptsize $X^*$};
       \node at (-0.23,-0.9) [string-blue,below] {\scriptsize $X$};
        \node at (0.32,-0.875) [string-blue,below] {\scriptsize $Y^*$};
       \node at (0.77,-0.9) [string-blue,below] {\scriptsize $Y$};
   \end{tikzpicture} 
    \end{aligned}
\end{equation*}
\caption{The dinatural transformations defining the Hopf algebra structure morphisms and the Hopf pairing on $\coend$.}\label{fig:coend-Hopf-def}
\end{figure}

\begin{definition}\label{def:mtc}
    A finite ribbon category $\calc$ is called \emph{modular} if the canonical Hopf pairing $\omega$ of the coend $\coend$ is non-degenerate.
\end{definition}
There are other equivalent definitions of modularity, see \autocite[Thm.\,1.1]{SHIMIZU2019nondeg}. For the rest of this and the next subsection we will always assume $\calc$ to be modular. 
It can be shown that $\calc$ is unimodular and that the Hopf algebra $\coend$ admits 
a unique-up-to-scalar two-sided integral $\Lambda$ and cointegral $\lambda$, 
see \autocite[Sec.\,2]{DGGPR19} and references therein for more details. 

We will normalise the integral and cointegral in terms of the \textsl{modularity parameter} $\zeta \in \mathbbm{k}^{\times}$ as
\begin{align}
    \lambda \circ \Lambda = \id_{\unit} 
    \quad , \quad
    \omega \circ (\id_{\coend} \otimes \Lambda) = \zeta \lambda \ .
\end{align}
That the second condition is possible is in fact equivalent to non-degeneracy of $\omega$.
The following statement is analogous to the case of classical Hopf algebras \autocite[Cor.\,4.2.13]{KL2001nsstqft}. 
\begin{proposition}
    The cointegral $\lambda$ induces a non-degenerate pairing 
    $\kappa := \lambda \circ \mu \circ (S \otimes \id_{\coend}) \colon \coend \otimes \coend \to \unit$ 
    which equips $\coend$ with the structure of a Frobenius algebra in $\calc$. The copairing is given by $\Delta \circ \Lambda$.  
\end{proposition}
The pairing $\kappa$ is called the \emph{Radford pairing}.  
Since the Hopf pairing is also non-degenerate, the composition 
\begin{align}
\calS := (\omega \otimes \id_{\coend}) \circ (\id_{\coend} \otimes (\Delta\circ \Lambda)) \colon \coend \to \coend
\label{eq:def-S-endo-of-L}
\end{align}
is invertible. We can define another endomorphism of $\coend$ by the universal property via $\cat{T} \circ \iota_X = \iota_X \circ (\id_{X^*} \otimes \theta_X)$. 
This morphism is invertible with inverse $\cat{T}^{-1} \circ \iota_X = \iota_X \circ (\id_{X^*} \otimes \theta_X^{-1})$. 
Moreover the constants $\Delta^{\pm}$ defined via 
\begin{align}\label{eq:defDeltapm}
    \epsilon \circ \cat{T}^{\pm1} \circ \Lambda = \Delta^{\pm} \, \id_{\unit}
\end{align} are non-zero and satisfy $\zeta = \Delta^+ \Delta^-$, see \autocite[Prop.\,2.6 \& Cor.\,4.6]{DGGPR19}. Using this and the normalisation of $\Lambda$ and $\lambda$ a direct computation shows that 
\begin{align} \label{eq:S^2}
    \calS^2 = \zeta S^{-1}
\end{align}
where $S^{-1}$ is the inverse of the antipode of $\coend$. 

\subsection{Tensor ideals and modified traces}
Let us denote with $\Proj(\calc)$ the full subcategory of projective objects of $\calc$. This forms a \emph{tensor ideal} in $\calc$, i.e.\ it is closed under retracts, and absorbent with respect to monoidal products with arbitrary objects of $\calc$ \autocite[Prop.\,4.2.12]{EGNO}.
In fact, $\Proj(\calc)$ is the smallest non-zero tensor ideal of $\calc$ \autocite[Lem.\,4.4.1]{GKP10}.

Recall that the (right) partial trace of an endomorphism $f \in \End_\calc (X \otimes Y)$ is defined as the endomorphism 
\begin{equation}
\rptr(f) = (\id_X \otimes \rev_Y) \circ (f \otimes \id_{Y^*}) \circ (\id_X \otimes \lcoev_Y) \in \End_\calc (X),
\end{equation}
or graphically
\begin{align}
\rptr(f) =\begin{tikzpicture}[xscale=-1,baseline=0.35cm]
        \filldraw[fill=white,draw=string-blue,thick] (-0.8,0.1) rectangle (0.9,0.7);
        \draw[string-blue,thick](-0.4,0.7) -- (-0.4,0.8)
        (-0.4,0.05) -- (-0.4,0.1)
        (-1.6,0.05) -- (-1.6,0.8);
        \begin{scope}[decoration={
                      markings,
                      mark=at position 0.55 with {\arrow{>}}}
                     ] 
            \draw[string-blue,thick,postaction={decorate}] (0.5,-0.6) -- (0.5,0.1);
            \draw[string-blue,thick,postaction={decorate}] (0.5,0.7) -- (0.5,1.4);
            \draw[string-blue,thick,postaction={decorate}] (-0.4,0.8) arc (0:180: 0.6 and 0.3);
            \draw[string-blue,thick,postaction={decorate}] (-1.6,0.05) arc (-180:0: 0.6 and 0.3);
        \end{scope}
       \node at (-0.3,0.9) [string-blue] {\scriptsize $W$};
       \node at (0.55,1.3) [string-blue,right] {\scriptsize $X$};
       \node at (0.55,-0.5) [string-blue,right] {\scriptsize $X$};
       \node at (0.05,0.4) [string-blue] {$f$};
   \end{tikzpicture} 
\end{align}

A \emph{modified trace} $\trI$ on $\Proj(\calc)$ is a family of linear maps
\begin{equation}
 \{ \trI_P : \End_\calc(P) \to \kk \}_{P \in \Proj(\calc)}
\end{equation}
satisfying the following conditions:
\begin{itemize}
	\item[1)] {\em Cyclicity}: For all $P,Q \in \Proj(\calc)$ and $f \colon P \to Q$, $g \colon Q \to P$ we have 
	 \begin{equation}
	     \trI_Q(f \circ g) = \trI_P(g \circ f).
	 \end{equation}
\item[2)] {\em Right partial trace}: For all $P \in \Proj(\calc)$, $X \in \calc$ and
	 $h \in \End_\calc(P \otimes X)$, 
	 \begin{equation}
	     \trI_{P \otimes X}(h) = \trI_P(\rptr(h));
	 \end{equation}
\end{itemize}
In general, there is also a notion of left partial trace and a modified trace would need to satisfy a condition analogous to 2) above for the left partial trace. However, 
since $\calc$ is ribbon in our setting, it suffices to only consider one of them \autocite[Thm\,3.3.1]{GKP10}. 

It was shown in \autocite[Thm\,5.5\,\&\,Cor.\,5.6]{GKP18} that there is unique-up-to-scalar non-zero
modified trace on $\Proj(\calc)$, and this trace induces a non-degenerate pairing on the Hom spaces. 

\subsection{Coends in functor categories}
As a final algebraic ingredient, we will now recall some general results on coends, in particular in functor categories. See \autocite{Loregian2021coendcalc,FS16coendsCFT} for a more detailed exposition and proofs. In this section $\calB$, $\calc$, and $\calD$ will denote finite ribbon categories, albeit these conditions can be drastically relaxed in general. Moreover, all functors we will consider will be assumed to be linear unless stated otherwise. 

First recall the ``delta distribution" property of the Hom functor, which is a reformulation of the Yoneda lemma.

\begin{lemma}[{\autocite[Prop.\,4]{FS16coendsCFT}}]
Let  $F \colon \calB \to \vect$ be a linear functor. For any object $Y \in \calB$ the coend of the functor 
\begin{equation}
\Hom_{\calB}(-,Y) \otimes_{\mathbbm{k}}F(-) \colon \calB^\op \times \calB \to \vect
\end{equation}
can be realised as the vector space $F(Y)$ with the family of linear maps
 \begin{equation}
     \begin{aligned}
         i_X \colon \Hom_{\calB}(X,Y) \otimes_{\mathbbm{k}}F(X) &\to F(Y)\\
         (f,x) &\mapsto F(f)(x) 
     \end{aligned}
 \end{equation}
for $X \in \calB$. Moreover, the isomorphism $F(Y) \cong \int^{X\in \calB} \Hom_{\calB}(X,Y) \otimes_{\mathbbm{k}}F(X)$ is natural in $Y$.
\end{lemma}

In practice this allows one to explicitly compute many coends as can be seen from the following useful corollary.       
\begin{corollary}\label{cor:gluingalg1}
For any $U,V,U',V' \in \calB$ the coend of the functor
\begin{align}
    \Hom_{\calB}(U \otimes (-), V) \otimes_{\mathbbm{k}} \Hom_{\calB}(U', (-) \otimes V') \colon \calB^{\op} \times \calB \to \vect
\end{align}
exists and can be realised by $(\Hom_{\calB}(U \otimes U', V \otimes V'), i)$ with the dinatural transformation 
\begin{equation}
\begin{aligned}
i_X \colon \Hom_{\calB}(U \otimes X, V) \otimes_{\mathbbm{k}} \Hom_{\calB}(U', X \otimes V') &\to \Hom_{\calB}(U \otimes U', V \otimes V')
\\ 
\begin{tikzpicture}[scale=0.9,baseline=0.3cm]
        \filldraw[fill=white,draw=string-blue,thick] (-1,0.1) rectangle (0.7,0.7);
        \begin{scope}[decoration={
                      markings,
                      mark=at position 0.52 with {\arrow{>}}}
                     ] 
            \draw[string-blue,thick,postaction={decorate}] (0.3,-1) -- (0.3,0.1);
            \draw[string-blue,thick,postaction={decorate}] (-0.6,-1) -- (-0.6,0.1);
            \draw[string-blue,thick,postaction={decorate}] (-0.15,0.7) -- (-0.15,1.8);
        \end{scope}
       \node at (-0.55,-0.5) [string-blue,right] {\scriptsize $U$};
       \node at (0.35,-0.5) [string-blue,right] {\scriptsize $X$};
       \node at (0,1.2) [string-blue,right] {\scriptsize $V$};
       \node at (-0.15,0.4) [string-blue] {$f$};
   \end{tikzpicture} \quad
   \otimes \quad
\begin{tikzpicture}[scale=0.9,baseline=0.3cm]
        \filldraw[fill=white,draw=string-blue,thick] (-1,0.1) rectangle (0.7,0.7);
        \begin{scope}[decoration={
                      markings,
                      mark=at position 0.52 with {\arrow{>}}}
                     ] 
            \draw[string-blue,thick,postaction={decorate}] (-0.15,-1) -- (-0.15,0.1);
            \draw[string-blue,thick,postaction={decorate}] (0.3,0.7) -- (0.3,1.8);
            \draw[string-blue,thick,postaction={decorate}] (-0.6,0.7) -- (-0.6,1.8);
        \end{scope}
       \node at (0,-0.5) [string-blue,right] {\scriptsize $U'$};
       \node at (-0.55,1.2) [string-blue,right] {\scriptsize $X$};
       \node at (0.35,1.2) [string-blue,right] {\scriptsize $V'$};
       \node at (-0.15,0.4) [string-blue] {$g$};
   \end{tikzpicture} \quad \quad
   &  \mapsto \quad
    \begin{tikzpicture}[scale=0.9,baseline=0.3cm]
        \filldraw[fill=white,draw=string-blue,thick] (-1,-0.45) rectangle (0.7,0.15)
        (-1.9,0.65) rectangle (-0.2,1.25);
        \draw[string-blue,thick](-0.6,0.15) -- (-0.6,0.65);
        \begin{scope}[decoration={
                      markings,
                      mark=at position 0.52 with {\arrow{>}}}
                     ] 
            \draw[string-blue,thick,postaction={decorate}] (-0.15,-1) -- (-0.15,-0.45);
            \draw[string-blue,thick,postaction={decorate}] (0.3,0.15) -- (0.3,1.8);
            \draw[string-blue,thick,postaction={decorate}] (-1.5,-1) -- (-1.5,0.65);
            \draw[string-blue,thick,postaction={decorate}] (-1.05,1.25) -- (-1.05,1.8);
        \end{scope}
        \node at (-1.5,-0.55) [string-blue,right] {\scriptsize $U$};
       \node at (-1.05,1.5) [string-blue,right] {\scriptsize $V$};
       \node at (-1.05,0.95) [string-blue] {$f$};
       \node at (0,-0.75) [string-blue,right] {\scriptsize $U'$};
       \node at (0.35,1.1) [string-blue,right] {\scriptsize $V'$};
       \node at (-0.25,-0.2) [string-blue] {$g$};
   \end{tikzpicture} 
\end{aligned}
\end{equation}
Moreover, the isomorphism between $\Hom_{\calB}(U \otimes U', V \otimes V')$ and the coend over $X$ of $\Hom_{\calB}(U \otimes X, V) \otimes_{\mathbbm{k}} \Hom_{\calB}(U', X \otimes V')$ is natural in $U,V,U',$ and $V'$.
\end{corollary}
More generally let $G \colon \calB \times 
\calc^{\mathrm{op}}\times\calc \to \calD$ be a functor. We can fix an object $Y \in \calB$, often called a \emph{parameter} in this context, to obtain a functor 
\begin{equation}
G_Y := G(Y,-,-) \colon \calc^{\mathrm{op}}\times\calc \to \calD.
\end{equation} 
Due to our finiteness assumptions the coend of $G_Y$ exists and is an object
  \begin{equation}
  e_{Y} := \int^{X\in\calc}\! G_{Y}(X,X) \in \calD.
  \end{equation}
The fact that $G$ is functorial in $Y$ as well implies that the assignment
  \begin{equation}
  Y \,\longmapsto\, e_Y
  \end{equation}
defines a functor from $\calB$ to $\calD$, which we will denote by $\int^{X\in\calc}\! G(-,X,X).$
On the other hand we can reinterpret $G$ as a functor
  \begin{equation} 
  \widetilde G \colon
  \calc^{\mathrm{op}}\times\calc \to \Fun (\calB , \calD).
  \end{equation}
If the coend of $\tilde G$ exists as an object in the functor category
$\Fun(\calB , \calD)$, then we denote it by
  \begin{equation} 
  \Big( \int^{X\in\calc}\! \widetilde G(X,X) \Big) (-) \colon \calB \to \calD.
  \end{equation}
The parameter theorem for coends explains how these two constructions are related:

\begin{theorem}[{\autocite[Sec.\,IX.7.]{MacLane1978cats}}]
The functor 
\begin{equation}
  \int^{X\in\calc}\! G(-,X,X) \colon \calB \to \calD
\end{equation}
has a natural structure of a coend for the functor
  \begin{equation}
  \tilde G \colon \calc^{\mathrm{op}}\times\calc \to 
  \Fun(\calB , \calD),
  \end{equation}
provided that all coends $\int^{X\in\calc}\! G(Y,X,X)$ exist, i.e. 
\begin{equation}
\int^{X\in\calc}\! G(-,X,X) \cong \Big( \int^{X\in\calc}\! \widetilde G(X,X) \Big) (-) \end{equation}
as objects in $\Fun(\calB,\calD)$.
\end{theorem}

We now turn to results that are specific to the setting of finite tensor categories, in particular instead of general functors we will from now on only consider left exact functors. We denote the category of left exact functors from $\calc$ to $\calD$ by $\coend \mathrm{ex}(\calc,\calD)$, and analogously with $\calR \mathrm{ex}(\calc,\calD)$ the category of right exact functors. 
 The following idea is due to \autocite{Lyu1996ribbon}.
\begin{definition}
    Let $\calB, \calc,$ and $\calD$ be finite tensor categories and let 
    \begin{equation}
    G \colon \calB \times\calc^{\op}\times \calc \to \calD
    \end{equation}
    be a functor left exact in each argument.
    The \emph{left exact coend} $\oint^{X\in \calc}\! G(-,X,X)$ 
    is the coend of $G$ over $\calc$ which is universal with respect to functors in the category $\coend\mathrm{ex}(\calB, \calD)$.
\end{definition}
The notation $\oint$ indicates that this is in general not the same thing as the standard coend in the full functor category $\Fun(\calB, \calD)$. 
Left exact coends work well with the Hom functor in the following sense.
\begin{proposition}[{\autocite[Prop.\,9]{FS16coendsCFT}}]\label{prop:lexcoendhom}
    The coend (over the first and third argument) of the functor
\begin{align}
    \Hom_{\calc}\big((-)\otimes (-), (-)\otimes (-)\big) \colon \calc^{\op} \times \calc^{\op} \times \calc \times \calc \to \vect
\end{align}
exists in the functor category $\coend\mathrm{ex}(\calc^{\op}\times \calc, \vect)$ and is given by $\big(\Hom_{\calc}(\coend\otimes (-), (-)),i \big)$
with the family, natural in $U$ and $V$, of dinatural transformations 
\begin{align*}
 i_X^{U,V} \colon \Hom_{\calc}(X\otimes U, X\otimes V) &\to \Hom_{\calc}(\coend\otimes U, V)
\\ 
\begin{tikzpicture}[scale=0.9,baseline=0.3cm]
        \filldraw[fill=white,draw=string-blue,thick] (-1,0.1) rectangle (0.7,0.7);
        \begin{scope}[decoration={
                      markings,
                      mark=at position 0.52 with {\arrow{>}}}
                     ] 
            \draw[string-blue,thick,postaction={decorate}] (0.3,-1) -- (0.3,0.1);
            \draw[string-blue,thick,postaction={decorate}] (-0.6,-1) -- (-0.6,0.1);
            \draw[string-blue,thick,postaction={decorate}] (0.3,0.7) -- (0.3,1.8);
            \draw[string-blue,thick,postaction={decorate}] (-0.6,0.7) -- (-0.6,1.8);
        \end{scope}
       \node at (-0.55,-0.5) [string-blue,right] {\scriptsize $X$};
       \node at (0.35,-0.5) [string-blue,right] {\scriptsize $U$};
       \node at (-0.55,1.2) [string-blue,right] {\scriptsize $X$};
       \node at (0.35,1.2) [string-blue,right] {\scriptsize $V$};
       \node at (-0.15,0.4) [string-blue] {$f$};
   \end{tikzpicture} \quad
   &  \mapsto \quad
    \begin{tikzpicture}[scale=0.9,baseline=0.3cm]
        \filldraw[fill=white,draw=string-blue,thick] (-1,-0.1) rectangle (0.7,0.5)
        (-2.2,1.3) rectangle (-0.8,1.7);
        \draw[string-blue,thick](-0.6,0.5) -- (-0.6,0.8)
        (-0.6,-0.25) -- (-0.6,-0.1)
        (-1.4,-0.25) -- (-1.4,0.8)
        (-1,1.3) -- (-1,0.8)
        (-2,1.3) -- (-2,0.3);
        \begin{scope}[decoration={
                      markings,
                      mark=at position 0.52 with {\arrow{>}}}
                     ] 
            \draw[string-blue,thick,postaction={decorate}] (0.3,-1) -- (0.3,-0.1);
            \draw[string-blue,thick,postaction={decorate}] (0.3,0.5) -- (0.3,1.8);
            \draw[string-blue,thick] (-1.5,0.8) -- (-0.5,0.8);
            \draw[string-blue,thick,postaction={decorate}] (-1.4,-0.25) arc (-180:0: 0.4 and 0.3);
            \draw[string-blue,thick,postaction={decorate}] (-2,-1) -- (-2,1.3);
        \end{scope}
       \node at (-0.4,-0.4) [string-blue] {\scriptsize $X$};
       \node at (-0.3,0.7) [string-blue] {\scriptsize $\iota_X$};
       \node at (-1.5,1.5) [string-blue] {\scriptsize $\kappa$};
       \node at (0.35,-0.5) [string-blue,right] {\scriptsize $U$};
       \node at (0.35,1.2) [string-blue,right] {\scriptsize $V$};
       \node at (-0.15,0.2) [string-blue] {$f$};
       \node at (-1.8,-0.5) [string-blue] {\scriptsize $\coend$};
   \end{tikzpicture} 
\end{align*}
where $\kappa$ is the Radford pairing on the coend $\coend$ from above.
\end{proposition}
Using the rigidity of $\calc$ we can also express this as
\begin{equation}
\oint^{X\in \calc}\! \Hom_{\calc}(X^*\otimes X\otimes -,-) \cong \Hom_{\calc}\left(\left(\int^{X\in \calc} \! X^*\otimes X\right)\otimes -,-\right).
\end{equation}
Or in other words: the left exact coend ``commutes" with the Hom functor.
It should be noted here that from a formal point of view, there is no reason to consider left exact instead of right exact functors. Indeed, 
for finite tensor categories the resulting functor categories are actually equivalent \autocite[Thm.\,3.2]{FSS19eilenbergwatts}.
However, for our purposes it turns out that left exact functors are more natural (cf.\ \Cref{sec:modfunc}).
To understand the notion of left exact coends better we recall a result from homological algebra (which however will not be used in the following).
\begin{lemma}\label{lem:rexproj}
    Let $\calA$ be an abelian category with enough injectives and let $\calA'$ be an abelian category. The restriction functor $\mathcal{L}\mathrm{ex}(\calA,\calA') \to \Fun(\mathrm{Inj}(\calA),\calA')$ is an equivalence.
\end{lemma}
\begin{proof}
We claim that $F \mapsto R_0F$ is an inverse to the restriction functor with $R_0F\colon \calA \to \calA'$ the zeroth right derived functor of $F \colon \mathrm{Inj}(\calA)\to \calA'$. Recall that this is defined on objects $X \in \calA$ as $R_0F(X) = \mathrm{ker}\big(F(I_0) \to F(I_1)\big)$ with $0 \to X \to I_0 \to I_1 \to \cdots$ an injective resolution of $X$. Note that $R_0F$ is well defined up to isomorphism since every choice of injective resolution leads to isomorphisms in homology. Using standard techniques in homological algebra it is straightforward to verify that $R_0F$ is indeed left exact and that for a left exact functor $G$ we have $R_0G \cong G$. 
\end{proof}
Analogously it can be shown that $\mathcal{R}\mathrm{ex}(\calA,\calA') \simeq \Fun(\mathrm{Proj}(\calA),\calA')$. Combining the equivalences $\coend\mathrm{ex}(\calA,\calA') \simeq \Fun(\mathrm{Inj}(\calA),\calA')$ and $\mathcal{R}\mathrm{ex}(\calA,\calA') \simeq \Fun(\mathrm{Proj}(\calA),\calA')$ with the fact that in a finite tensor category, every projective object is also injective and vice versa
\autocite[Prop.\,6.1.3]{EGNO} we immediately get an equivalence $\mathcal{R}\mathrm{ex}(\calc,\calD) \simeq \mathcal{L}\mathrm{ex}(\calc,\calD)$. Moreover, we also have $\mathcal{L}\mathrm{ex}(\calc,\calD) \simeq \Fun(\Proj(\calc),\calD)$. With this we can reinterpret the definition of the left exact coend as the requirement of having a projective object present. 

\begin{remark}
The equivalence $\mathcal{R}\mathrm{ex}(\calc,\calD) \simeq \mathcal{L}\mathrm{ex}(\calc,\calD)$ constructed above
is different from the one given in \autocite[Thm.\,3.2]{FSS19eilenbergwatts} mentioned before. This can be seen by considering where for example the identity functor is sent under both equivalences. In the above
equivalence the identity gets sent to itself while in the one by Fuchs et.\ al.\ it gets sent to the Nakayama functor \autocite[Sec.\,3.5]{FSS19eilenbergwatts}. 
\end{remark}

\section{Non-semisimple 3d TFT}\label{sec:3dTFT}
In this section we will briefly review the non-semisimple TFTs of \autocite{DGGPR19}. First we will recall bichrome graphs as well as the corresponding $3$-manifold invariants. Then we will define a $3$-dimensional bordism category and introduce two subcategories of so-called admissible bordisms and study their relation.  We apply the universal TFT construction to obtain two non-compact TFTs and discuss an algebraic model for the state spaces of these TFTs, as well as their behaviour under orientation reversal and Deligne products. Our exposition mostly follows \autocite[Sec.\,2]{DGGPR20}, except that we include results on the relation of the TFTs for the two dual sets of admissible bordisms, and on their behaviour under orientation reversal and Deligne products.

\subsection{3-manifold invariants}\label{ssec:3mf-inv}
Bichrome graphs are a generalisation of the ribbon graphs of  \autocite[Sec.\! I.2]{Tu16} where two kinds of edges are present, \emph{red} edges without labels, and \emph{blue} edges labelled by objects of $\calc$. Coupons can be, according to the edges intersecting them, either \emph{bichrome} and unlabelled, or \emph{blue} and labelled as usual by morphisms of $\calc$. Furthermore, there are only two possible configurations of bichrome coupons allowed, from which we will only need the following one (which we just draw as a horizontal line):
\begin{align*}
    \begin{tikzpicture}
        \draw[string-blue,thick] (0.25,0) -- (0.25,0.6);
       \draw[string-red,thick] (0,-1) -- (0,0);
       \draw[string-red,thick,->] (0,0) -- (0,-0.5);
       \draw[string-red,thick] (0.5,-1) -- (0.5,0);
       \draw[string-red,thick,->] (0.5,-1) -- (0.5,-0.4);
       \draw[string-red,thick] (-0.2,0) -- (0.7,0);
       \node at (0.25,0.55) [string-blue,above] {$\coend$};
   \end{tikzpicture} 
\end{align*}
See \autocite[Sec.\,3.2]{DGGPR19} for the other one, which involves an end instead of the coend,
and \autocite[Rem.\,3.6]{DGGPR19} on why it suffices to only consider one. Red coupons are generally forbidden. 

There is a category $\calB_{\calc}$ of $\calc$-coloured bichrome graphs. Objects $(\underline{V},\underline{\epsilon})$ are finite sequences $((V_1,\epsilon_1),\dots,(V_n,\epsilon_n))$ where every $V_k$ is an object of $\calc$ and $\epsilon_k \in \{+,-\}$. 
From every object $(\underline{V},\underline{\epsilon})$ we get a set of $\calc$-labelled framed, oriented blue points located at the fixed points $0,1,2,3,\dots$ on the real axis in $\R^2$. 
A morphism $T \colon (\underline{V},\underline{\epsilon}) \to (\underline{W},\underline{\nu})$ is an isotopy class of bichrome graphs in $\R^2 \times I$ between the corresponding standard sets of framed, oriented blue points such that the framings, orientations, and labels match. The subcategory $\calR_\calc$ consisting of all objects but only blue graphs is the familiar category of $\calc$-coloured ribbon graphs of \autocite{Tu16}. In \autocite[Sec.\,3.1]{DGGPR19} it is shown how to extend the Reshetikhin-Turaev functor $F_\calc \colon \calR_\calc \to \calc$ to all bichrome graphs. For this construction a non-zero integral $\Lambda$ of $\coend$ needs to be chosen, and for this reason the corresponding functor will be denoted by $F_{\Lambda} \colon \calB_\calc \to \calc$. We will not recall the construction in full detail, instead we will illustrate how to evaluate the functor $F_\intL : \calB_\calc \to \calc$ in an example in \Cref{fig:bicex}.

\begin{figure}[tb]
\centering
        \begin{subfigure}{0.3\textwidth}
            \begin{tikzpicture}
        \begin{scope}[decoration={
                      markings,
                      mark=at position 0.52 with {\arrow{>}}}
                     ] 
            \draw[string-red,thick,postaction={decorate}] (-1,0.75) arc (180:360: 0.5 and 0.5);
            \draw[string-red,thick,postaction={decorate}] (0,0.1) arc (0:180: 0.5 and 0.5);
            \draw[string-red,thick,postaction={decorate}] (-1,0.1) arc (180:360: 0.5 and 0.5);
            \fill[white] (-1,0.36) rectangle (-0.7,0.48)
            (-1.6,-0.58) rectangle (-1.45,-0.45);
            \draw[string-red,thick] (-0.9,1.4) -- (-0.9,1.3)
            (-0.1,1.4) -- (-0.1,1.3)
            (-0.9,1.3) .. controls (-0.9,1.1) and (-1,0.95) .. (-1,0.75)
            (-0.1,1.3) .. controls (-0.1,1.1) and (0,0.95) .. (0,0.75);
            \draw[string-red,thick] (-1.1,1.4) -- (0.1,1.4);
            \draw[string-blue,thick,postaction={decorate}] (-0.5,1.4) -- (-0.5,2.4);
            \fill[white] (-0.3,0.36) rectangle (0,0.48);
            \draw[string-red,thick] (-0.5,-0.4) arc (270:90: 0.5 and 0.5);
            \draw[string-red,thick] (-0.5,0.25) arc (270:360: 0.5 and 0.5);
        \end{scope}
        \node at (-0.45,1.9) [string-blue,right] {\scriptsize $\coend$};
   \end{tikzpicture} 
            \caption{Morphism in $\calB_\calc$ from $\varnothing$ to $(\coend,+)$}\label{fig:bicex1}
        \end{subfigure}
    \hfill
    \begin{subfigure}{0.3\textwidth}
        \begin{tikzpicture}[scale=0.85,baseline=0cm]
        \begin{scope}[decoration={
                      markings,
                      mark=at position 0.52 with {\arrow{>}}}
                     ] 
            \draw[string-red,thick,postaction={decorate}] (-2.5,0.7) arc (0:180: 0.35 and 0.35);  
            \draw[string-red,thick,postaction={decorate}] (-1.9,0.7) arc (0:-180: 0.3 and 0.35);  
            \fill[white](-3.13,1.1) rectangle (-3.0,0.85);
            \draw[string-red,thick,postaction={decorate}] (-3,0.8) arc (0:180: 0.4 and 0.3);
            \draw[string-red,thick] (-1.2,0.7) -- (-1.2,0.9)
            (-1.9,0.7) -- (-1.9,0.9)
            (-3,0.8) -- (-3,-0.5)
            (-3.8,0.8) -- (-3.8,-0.5)
            (-1.2,0.5) -- (-1.2,0.7)
            (-3.2,0.5) -- (-3.2,0.7)
            (-2.55,-0.2) -- (-2.55,-0.5)
            (-1.85,-0.2) -- (-1.85,-0.5);
            \fill[white] (-3.1,0.3) rectangle (-2.9,0.15);
            \draw[string-red,thick] (-3.2,0.5) .. controls (-3.2,0.2) and (-2.55,0.1) .. (-2.55,-0.2);
            \draw[string-red,thick,postaction={decorate}] (-1.2,0.5) .. controls (-1.2,0.2) and (-1.85,0.1) .. (-1.85,-0.2);
            \draw[string-red,thick] (-1,0.9) -- (-2.1,0.9);
            \draw[string-blue,thick,postaction={decorate}] (-1.55,0.9) -- (-1.55,1.9);
        \end{scope}
        \node at (-1.5,1.4) [string-blue,right] {\scriptsize $\coend$};
   \end{tikzpicture} 
        \caption{Cut presentation of the morphism}\label{fig:bicex2}
    \end{subfigure}
    \hfill
    \begin{subfigure}{0.3\textwidth}
        \begin{tikzpicture}[scale=0.85,baseline=0cm]
        \begin{scope}[decoration={
                      markings,
                      mark=at position 0.52 with {\arrow{>}}}
                     ] 
            \draw[string-blue,thick,postaction={decorate}] (-2.5,0.7) arc (0:180: 0.35 and 0.35);  
            \draw[string-blue,thick,postaction={decorate}] (-1.9,0.7) arc (0:-180: 0.3 and 0.35);  
            \fill[white](-3.13,1.1) rectangle (-3.0,0.85);
            \draw[string-blue,thick,postaction={decorate}] (-3,0.8) arc (0:180: 0.4 and 0.3);
            \draw[string-blue,thick] (-1.2,0.7) -- (-1.2,0.9)
            (-1.9,0.7) -- (-1.9,0.9)
            (-3,0.8) -- (-3,-0.5)
            (-3.8,0.8) -- (-3.8,-0.5)
            (-1.2,0.5) -- (-1.2,0.7)
            (-3.2,0.5) -- (-3.2,0.7)
            (-2.55,-0.2) -- (-2.55,-0.5)
            (-1.85,-0.2) -- (-1.85,-0.5);
            \fill[white] (-3.1,0.3) rectangle (-2.9,0.15);
            \draw[string-blue,thick] (-3.2,0.5) .. controls (-3.2,0.2) and (-2.55,0.1) .. (-2.55,-0.2);
            \draw[string-blue,thick,postaction={decorate}] (-1.85,-0.2) .. controls (-1.85,0.1) and (-1.2,0.2) .. (-1.2,0.5);
            \draw[string-blue,thick] (-1,0.9) -- (-2.1,0.9);
            \draw[string-blue,thick,postaction={decorate}] (-1.55,0.9) -- (-1.55,1.9);
        \end{scope}
        \node at (-1.5,1.4) [string-blue,right] {\scriptsize $\coend$};
        \node at (-3,-0.45) [string-blue,below] {\scriptsize $X_2$};
        \node at (-3.75,-0.41) [string-blue,below] {\scriptsize $X_2^*$};
        \node at (-1.8,-0.45) [string-blue,below] {\scriptsize $X_1$};
        \node at (-2.45,-0.41) [string-blue,below] {\scriptsize $X_1^*$};
   \end{tikzpicture} 
        \caption{$\calc$-coloured ribbon graph for the morphism}\label{fig:bicex3}
    \end{subfigure}
    \caption{Schematic algorithm for the evaluation of the functor $F_\intL : \calB_\calc \to \calc$ on a morphism.}\label{fig:bicex}
\end{figure}

Starting from the bichrome graph in Figure\,\ref{fig:bicex}(a), we form the \emph{cut presentation} in Figure\,\ref{fig:bicex}(b) where each red loop is cut once (bichrome coupons are considered part of the red graph). 
Now we choose objects $X_1, \dots, X_n$ and label the red edges correspondingly by the $X_k$ and $X_k^*$, e.g.\ in Figure\,\ref{fig:bicex}(b) we have $n=2$. The bichrome coupons meeting these edges will be labelled by the universal dinatural morphism $\iota_{X_k}$. At this point, all edges and coupons are labelled and thus blue, see Figure\,\ref{fig:bicex}(c). 
We have now obtained a $\calc$-coloured ribbon graph, i.e.\ a morphism in $\calR_\calc$. To this we can apply the Reshetikhin-Turaev functor $F_\calc$ to obtain a morphism in $\calc$. By construction, the resulting family of morphisms in $\calc$ is dinatural in the labels $X_1,\dots,X_n$. Thus, by the universal property of $\coend$ we get a morphism out of $\coend^{\otimes n}$, possibly tensored with the objects coming from blue boundary vertices. As the final step we precompose our morphism with the $n$-fold tensor power of the integral $\intL$, again tensored with the identity on the objects coming from blue boundary vertices. In our example we obtain 
\begin{equation}
    F_\intL( T ) =  (\overline{\omega} \otimes \id_{\coend}) \circ (\Lambda \otimes (\Delta\circ \Lambda)) = \mathcal{S} \circ S \circ \Lambda = \mathcal{S} \circ \Lambda 
\end{equation}
where $\overline{\omega} = \omega \circ (S \otimes \id_\coend)$ from \eqref{eq:hopfpairingmirror}
and $\mathcal{S}$ is the modular $S$-morphism from \eqref{eq:def-S-endo-of-L}. 

It can be shown that this construction gives a functor $F_\intL$ called the \emph{Lyubashenko-Reshetikhin-Turaev functor}, we refer to \autocite[Sec.\,3.1]{DGGPR19} for more details on the definition and well-definedness of $F_\intL$. Moreover, by construction we have a commuting diagram of functors
\begin{equation}
\begin{tikzcd}[ampersand replacement=\&]
	{\calR_\calc} \& \calc \\
	{\calB_\calc}
	\arrow["{F_\calc}", from=1-1, to=1-2]
	\arrow[from=1-1, to=2-1]
	\arrow["{F_\Lambda}"', from=2-1, to=1-2]
\end{tikzcd}
\end{equation}
where the functor $\calR_\calc \to \calB_\calc$ acts as the identity on objects and as the inclusion of purely blue ribbon graphs into bichrome graphs at the level of morphisms.

We can now use this functor to define invariants of closed bichrome graphs, i.e.\ endomorphisms of $\varnothing$ in $\calB_\calc$, and of $3$-manifolds. For this we need so-called \emph{admissible bichrome graphs}, which are bichrome graphs with at least one blue edge labelled by a projective object of $\calc$. For an admissible closed bichrome graph $T$ we define its \emph{cutting presentation} as a bichrome graph $T_V$ featuring a single incoming boundary vertex and a single outgoing one, both positive and labelled by $V \in \Proj(\calc)$, and whose trace closure is $T$. For any $T$ a closed admissible bichrome graph and $T_V$ a cutting presentation of $T$, the scalar
\begin{equation}
 F'_\intL(T) := \rmt_V(F_\intL(T_V))
\end{equation} 
is a topological invariant of $T$, so in particular it is independent of the choice of cutting presentation $T_V$, see \autocite[Thm.\,3.3]{DGGPR19}. 

We can now use $F'_\intL$ to define invariants of \emph{decorated $3$-manifolds}, i.e.\ pairs $(M,T)$, where $M$ is a connected closed $3$-manifold, and where $T \subset M$ is a closed bichrome graph. We call a pair $(M,T)$ \emph{admissible} if $T$ is.

Let now $(M,T)$ be an admissible decorated $3$-manifold, and let $L$ be a surgery presentation of $M$ given by a red framed oriented link in $S^3$ with $\ell$ components and signature $\sigma$. Assume further that the bichrome graph $T$ is contained in the exterior of the surgery link $L$, so that we can think of them as simultaneously embedded in $S^3$. Moreover let $\calD$ be a choice of square root of the modularity parameter $\zeta$, and 
\begin{equation}\label{eq:defdelta}
    \delta = \frac{\calD}{\Delta_+} = \frac{\Delta_-}{\calD}.
\end{equation}
Modular tensor categories with $\delta = 1$ are sometimes called \emph{anomaly free}.
The scalar
\begin{equation}\label{eq:L'-via-F'}
 \rmL'_\calc(M,T) := \calD^{-1-\ell} \delta^{-\sigma} F'_\intL(L \cup T)
\end{equation}
is a topological invariant of the pair $(M,T)$, see \autocite[Thm.\,3.8]{DGGPR19}, and called the \emph{renormalised Lyubashenko invariant} of the admissible decorated $3$-manifold $(M,T)$. 

\begin{remark}
The construction of \autocite{DGGPR19} also allows for other tensor ideals of $\calc$ as long as they permit a modified trace, see \autocite{BGR2023linkinvs} for examples of this.
In particular by choosing $\calc$ as tensor ideal with the categorical trace one recovers the original Lyubashenko invariants. 
The construction of renormalised $3$-manifold invariants via modified traces was introduced in \autocite{CGP12}. 
\end{remark}

Let us now discuss how $\rmL'_\calc$ behaves under orientation reversal.
To do so, we need to be more specific on the orientation data contained in bichrome graph. Namely, a bichrome graph $T$ (in a 3-manifold $M$ or in $\R^2\times I$) is a ribbon graph with blue and red components, and each of these consists of ribbons and coupons. A ribbon is an embedded annulus or rectangle with a one-dimensional core which carries a $1$-orientation and gives the direction of the ribbon, as well as a $2$-orientation on the ribbon surface. In pictures like the one in \Cref{fig:bicex}(a), we draw the core and its orientation, while the ribbon surface is not shown and is assumed to be parallel to the paper-plane $\R \times \{0\} \times I$ and carries its $2$-orientation. A coupon is an embedded rectangle with a $2$-orientation and a choice of bottom edge and top edge to which ribbons attach, and which determine the source and target object of the labelling morphism (for a blue coupon). 

For a bichrome graph $T$ in $\R^2\times I$ we define the \emph{mirrored} bichrome graph $\overline{T}$ as the image of $T$ under reflection along the plane $\R \times \{0\} \times I$ (the paper-plane in our drawings). 
In particular, if a ribbon lies in this plane (or is parallel to the plane), the reflection does not affect its $1$- or $2$-orientation. 
Note that we can obtain a diagram of $\overline{T}$, i.e.\ a projection of the embedding $\overline{T} \subset \R^2\times I$ to the plane $\R \times \{0\} \times I$, from a diagram of $T$ by exchanging all overcrossings with undercrossings. 

For a decorated $3$-manifold $(M,T)$ we denote by $(-M,T)$ the decorated $3$-manifold obtained by reversing the orientation of $M$ and keeping the $1$- and $2$-orientations of $T$ as they are. For example, if $M$ is the $3$-sphere written as $\R^3 \cup \{\infty\}$ and $T$ a closed bichrome graph in $\R^3$, then $(-M,T)$ is orientation-preserving diffeomorphic to $(M,\overline T)$ via the above reflection.

\begin{lemma}\label{lem:invori}
For an admissible decorated $3$-manifold $(M,T)$ we have
    \begin{equation}
        \rmL'_\calc(-M,T) = \rmL'_{\overline{\calc}}(M,T),
    \end{equation}
    with $\overline{\calc}$ the mirrored category of $\calc$.
\end{lemma}
\begin{proof}
    First recall that if $L$ is a surgery presentation of $M$, then $\overline{L}$ is a surgery presentation of $-M$ \autocite[Ch.\,3.4]{Sav12topof3man}. Moreover, for an admissible pair $(-M,T)$ a surgery representation is given by $(S^3,\overline{L}\cup \overline{T})$. In particular the signature of $\overline{L}$ is minus the signature of $L$. Let us denote with $F'_{\overline{\intL}}$ the Lyubashenko-Reshetikhin-Turaev functor for $\overline{\calc}$. On bichrome graphs $T$ with no red components we immediately get 
    \begin{equation}
        F'_{\overline{\intL}}(\overline{T}) = F'_{\intL}(T)
    \end{equation}
    by the same argument as in the semisimple setting \autocite[Cor.\,II.2.8.4]{Tu16}. 

    Next note that since $\calc$ and $\overline\calc$ are equal as pivotal categories, the coend $\coend$ is given by the same object and dinatural transformation in either case. 
    This remains true for the coalgebra structure and the unit morphism in \eqref{eq:coendhopf} as their definition in \Cref{fig:coend-Hopf-def} does not involve the ribbon structure. 
    The product $\overline\mu$ and antipode $\overline{S}$ are in general different when computed in $\overline\calc$, and we denote by $\overline{\coend}$ the corresponding Hopf algebra in $\overline{\calc}$. We stress that $\overline{\coend}$ is in general not a Hopf algebra in $\calc$ as the braiding enters the Hopf algebra axioms. 
    It can however be directly checked from the definition that $\overline{\mu} = \mu \circ \beta^{-1}_{\coend,\coend} : \coend \otimes \coend \to \coend$ as morphisms in $\calc$.
    Using this it follows that $\Lambda$ is also a two-sided integral for $\overline{\coend}$, and we set $\overline\Lambda = \Lambda$.
    Then in particular $\Delta_{\pm}^{\overline{\calc}} = \Delta_{\mp}^{\calc}$. Using this the claim follows.
\end{proof}

Next we turn to $3$-manifold invariants for Deligne products of modular tensor categories. Recall that the \emph{Deligne product} of two finite linear categories $\calA$ and $\calB$ is a finite linear category $\calA \boxtimes \calB$ together with a functor $\boxtimes \colon \calA \times \calB \to \calA \boxtimes \calB$ which is right exact in each argument and satisfies the following universal property: Let $\calD$ be another finite linear category and let $F \colon \calA \times \calB \to \calD$ be a functor right exact in each argument, then there is a unique (up to unique natural isomorphism)
right exact functor $\widehat{F}  \colon \calA \boxtimes \calB \to \calD$ together with an equivalence
$F \cong \widehat{F} \circ \boxtimes$. See \autocite[Sec.\,1.11]{EGNO} for the existence of the Deligne product as well as more properties. For $\calc$ and $\calD$ modular tensor categories, also $\calc\boxtimes \calD$ is a modular tensor category.

Let $(M,T)$ be an admissible $\calc \boxtimes \calD$ decorated $3$-manifold such that all labels of $T_{\calc \boxtimes \calD}$ are of the form $X\boxtimes Y$ for some $X\in \calc$ and $Y\in \calD$, and all coupons are pure tensors under the isomorphism
\begin{align}
        \Hom_{\calc \boxtimes \calD} (-\boxtimes -, -\boxtimes -)\cong \Hom_\calc(-,-) \otimes_\kk \Hom_{\calD}(-,-) 
\end{align}
coming from the universal property of the Deligne product \autocite[Prop.\,1.11.2]{EGNO}. Now define $T_\calc$ (resp.\ $T_\calD$) to be the $\calc$ (resp.\ $\calD$) admissible bichrome graph with same underlying topological graph as $T$, but with only the $\calc$ (resp.\ $\calD$) labels.
\begin{lemma}\label{lem:invdeligne}
    Let $(M,T)$, $(M,T_\calc)$, and $(M,T_\calD)$ be as above. Then
    \begin{equation}
        \rmL'_{\calc\boxtimes\calD}(M,T) = \rmL'_{\calc}(M,T_\calc) \cdot \rmL'_{\calD}(M,T_\calD).
    \end{equation}
\end{lemma}
\begin{proof}
    First recall the isomorphism $\coend_{\calc \boxtimes \calD} \cong \coend_{\calc}\boxtimes \coend_{\calD}$ from \autocite[Cor.\! 3.12]{FSS19eilenbergwatts}. This is actually an isomorphism of Hopf algebras because the Hopf algebra structures come from the universal property of the relevant coends, thus $\Lambda_{\calc \boxtimes \calD} = \Lambda_{\calc}\boxtimes \Lambda_{\calD}$. Therefore the invariants of purely red graphs are the same. For blue graphs note that the unique-up-to-scalar modified trace on $\calc\boxtimes \calD$ is canonically determined by the modified traces on $\calc$ and $\calD$. To see this combine the relation between trivialisations of the Nakayama functor and modified traces on a finite tensor category \autocite[Thm.\! 3.6]{SW21tracefieldtheory} with the behaviour of the Nakayama functor under the Deligne tensor product \autocite[Prop.\,3.20]{FSS19eilenbergwatts}.
\end{proof}

\subsection{Admissible bordism categories}\label{ssec:adbordcat}
 Before coming to the definition of the relevant $3$-dimensional bordism categories, we need to adapt a few definitions from above to a more general setting. A \emph{blue set $P$ inside a closed surface $\Sigma$} is a finite, discrete set of blue points of $\Sigma$,
 each endowed with a $0$-orientation $\pm$, a non-zero tangent vector, and a label given by an object of $\calc$. 
 A \emph{bichrome graph $T$ inside a 3-dimensional bordism $M \colon \Sigma \to \Sigma'$} is a bichrome graph embedded inside $M$ such that its boundary vertices are given by blue sets inside the boundary $\partial M$. Moreover the boundary identification $\partial M \cong -\Sigma \sqcup\Sigma'$ needs to be compatible with the blue sets. 
 With this terminology in place, we can define the symmetric monoidal category ${\bord}_{3,2}^{\chi}(\calc)$ of $3$-dimensional bordisms with $\calc$-coloured bichrome graphs:
\begin{itemize}\setlength{\leftskip}{-1.3em}

    \item 
An \emph{object $\underline{\Sigma}$ of ${\bord}_{3,2}^{\chi}(\calc)$} is a triple $(\Sigma, P, \lambda)$ where: 
\begin{enumerate}\setlength{\leftskip}{-1.3em}
 \item $\Sigma$ is a closed surface;
 \item $P \subset \Sigma$ is a blue set;
 \item $\lambda \subset H_1(\Sigma;\R)$ is a Lagrangian subspace with respect to the intersection pairing.\footnote{This is needed to precisely formulate the gluing anomaly, see \autocite[Ch.\,IV]{Tu16}. 
 See also \autocite[Sec.\,4]{Haioun23nonssrel} for an interpretation of the origin of $\lambda$ in the context of relative TFTs.}
\end{enumerate}

\item
A \emph{morphism $\underline{M} \colon \underline{\Sigma} \to \underline{\Sigma}'$} is an equivalence class of triples $(M,T,n)$ where:
\begin{enumerate}\setlength{\leftskip}{-1.3em}
 \item $M$ is a 3-dimensional bordism from $\Sigma$ to $\Sigma'$;
 \item $T \subset M$ is a bichrome graph from $P$ to $P'$;
 \item $n \in \Z$ is an integer called the \emph{signature defect} of $M$.
\end{enumerate}
Two triples $(M,T,n)$ and $(M',T',n')$ are equivalent if $n = n'$ and if there exists an isomorphism of bordisms $f : M \rightarrow M'$ satisfying $f(T) = T'$. 

\item
The \emph{identity morphism $\id_{\underline{\Sigma}} : \underline{\Sigma} \rightarrow \underline{\Sigma}$ associated with an object $\underline{\Sigma} = (\Sigma,P,\lambda)$ of ${\bord}_{3,2}^{\chi}(\calc)$} is the equivalence class of the triple
\begin{equation}
 (\Sigma \times I,P \times I,0).
\end{equation}

\item
The \emph{composition $\underline{M}' \circ \underline{M} : \underline{\Sigma} \to \underline{\Sigma}''$ of morphisms $\underline{M} \colon \underline{\Sigma} \rightarrow \underline{\Sigma}'$ and $\underline{M}' \colon \underline{\Sigma}' \rightarrow \underline{\Sigma}''$} in 
${\bord}_{3,2}^{\chi}(\calc)$ is the equivalence class of the triple
\begin{equation}\label{eq:anomaly-morphism-composition}
 \big(\, M \cup_{\Sigma'} M' \,,\, T \cup_{P'} T' \,,\, n + n' - \mu(M_*(\lambda),\lambda',M'^* (\lambda'')) \,\big),
\end{equation}
where $M_*(\lambda)$ and $M'^*(\lambda'')$ are certain Lagrangian subspaces of $H_1(\Sigma';\R)$ and $\mu$ denotes the Maslov index, see \autocite[Sec.\,IV.3-4]{Tu16} for details.

\item
The \emph{monoidal product $\underline{\Sigma} \disjun \underline{\Sigma}'$ of objects $\underline{\Sigma}$, $\underline{\Sigma}'$} in ${\bord}_{3,2}^{\chi}(\calc)$ is the triple 
\begin{equation}
 (\Sigma \sqcup \Sigma',P \sqcup P',\lambda \oplus \lambda').
\end{equation}
The \emph{unit of ${\bord}_{3,2}^{\chi}(\calc)$} is the object whose surface is the empty set, and it will be denoted $\varnothing$. 
The \emph{monoidal product $\underline{M} \disjun \underline{M}' : \underline{\Sigma} \disjun \underline{\Sigma}' \rightarrow \underline{\Sigma}'' \disjun \underline{\Sigma}'''$ of morphisms $\underline{M} : \underline{\Sigma} \rightarrow \underline{\Sigma}''$, $\underline{M}' : \underline{\Sigma}' \rightarrow \underline{\Sigma}'''$} in ${\bord}_{3,2}^{\chi}(\calc)$ is the equivalence class of the triple
\begin{equation}
 (M \sqcup M',T \sqcup T',n+n').
\end{equation}
\end{itemize}
It is straightforward to see that ${\bord}_{3,2}^{\chi}(\calc)$ can be equipped with a symmetric braiding, and moreover with a pivotal structure. For the latter,
the dual object of $\underline{\Sigma}$ in ${\bord}_{3,2}^{\chi}(\calc)$ is given by $\underline{\Sigma}^* = (-\Sigma,-P,\lambda)$ where $-P$ is the same set of blue points with 
$0$-orientation and tangent vector
reversed. The $3$-manifold underlying the dual of a morphism $[(M,T,n)]$ is again $M$ and not $-M$, but with in-going and out-going boundary component exchanged, see \autocite[Sec.\! 10.1.4]{TV} for more details on the pivotal structure. As usual we will say a morphism $[(M,T,n)]$ has a certain topological property if the $3$-manifold underlying $M$ has this property, e.g.\ we say $[(M,T,n)]$ is connected if and only if $M$ is connected. The superscript $\chi$ indicates that we are working with an extension of the standard bordism category ${\bord}_{3,2}(\calc)$ (which does not include Lagrangian subspaces and signature defects).

The TFTs of \autocite{DGGPR19} are defined on a strictly smaller subcategory which is no longer rigid. There are two choices for this subcategory:

\begin{definition}
 The \emph{admissible bordism (sub)categories} $\widehat{\bord}_{3,2}^{\chi}(\calc)$ and $\widecheck{\bord}_{3,2}^{\chi}(\calc)$ are the symmetric monoidal subcategories of ${\bord}_{3,2}^{\chi}(\calc)$ with the same objects but featuring only morphisms $[(M,T,n)]$ which satisfy one of the following \emph{admissibility conditions}:
\begin{enumerate}
    \item $\widehat{\bord}_{3,2}^{\chi}(\calc)$: Every connected component of $M$ disjoint from the \emph{outgoing} boundary contains an admissible bichrome subgraph of $T$, i.e.\ at least one edge of $T$ in that component is labelled with a projective object. 
    \item $\widecheck{\bord}_{3,2}^{\chi}(\calc)$: Every connected component of $M$ disjoint from the \emph{incoming} boundary contains an admissible bichrome subgraph of $T$.
\end{enumerate}	
\end{definition}
Note that there are morphisms in ${\bord}_{3,2}^{\chi}(\calc)$ which are neither in $\widehat{\bord}_{3,2}^{\chi}(\calc)$ nor $\widecheck{\bord}_{3,2}^{\chi}(\calc)$, e.g.\ $(M,\emptyset,0)$ for any closed $3$-manifold $M$. As mentioned above, neither $\widehat{\bord}_{3,2}^{\chi}(\calc)$ nor $\widecheck{\bord}_{3,2}^{\chi}(\calc)$ are rigid,
in fact the dualisable objects are exactly the surfaces with at least one projectively labelled marked point. The following lemma is clear:
\begin{lemma}
The duality functor $(-)^* \colon \bord_{3,2}^{\chi}(\calc) \to \bord_{3,2}^{\chi}(\calc)^{\mathrm{op}}$ induces an equivalence of symmetric monoidal categories $\widehat{\bord}_{3,2}^{\chi}(\calc) \simeq \widecheck{\bord}_{3,2}^{\chi}(\calc)^{\mathrm{op}}$.
\end{lemma}

\subsection{TFT construction}\label{sec:tqftcon}
We extend the renormalised Lyubashenko invariant to closed morphisms $\underline{M} = [(M,T,n)]$, i.e.\ endomorphisms of the monoidal unit, 
of $\widehat{\bord}_{3,2}^{\chi}(\calc)$ by setting
\begin{equation}
 \rmL'_\calc(\underline{M}) := \delta^n \rmL'_\calc(M,T)
\end{equation}
for $M$ connected and by setting
\begin{equation}
 \rmL'_\calc(\underline{M}_1 \disjun \ldots \disjun \underline{M}_k) := \prod_{i=1}^k \rmL'_\calc(\underline{M}_i)
\end{equation}
for $\underline{M}$ a finite disjoint union of closed connected morphisms $\underline{M}_1, \ldots, \underline{M}_k$.

The universal TFT construction of \autocite{BHMV95univ} allows one to extend $\rmL'_\calc$ to a functor from bordisms to vector spaces as follows.
For $\underline{\Sigma} \in \widehat{\bord}_{3,2}^{\chi}(\calc)$ define $ \widehat\calV(\underline{\Sigma})$ to be the free vector space generated by the set of morphisms $\underline{M}_{\underline{\Sigma}} : \varnothing \rightarrow \underline{\Sigma}$ of $\widehat{\bord}_{3,2}^{\chi}(\calc)$, and $ \widehat\calV'(\underline{\Sigma})$ the free vector space generated by the set of morphisms $\underline{M}'_{\underline{\Sigma}} : \underline{\Sigma} \rightarrow \varnothing$ of $\widehat{\bord}_{3,2}^{\chi}(\calc)$. Next, consider the bilinear form
\begin{equation}\label{eq:curly-V-pairing}
 \begin{aligned}
     \langle \cdot , \cdot \rangle_{\underline{\Sigma}} \colon  \widehat\calV'(\underline{\Sigma}) \times  \widehat\calV(\underline{\Sigma})  &\rightarrow  \mathbbm{k} \\
   (\underline{M}'_{\underline{\Sigma}},\underline{M}_{\underline{\Sigma}})  &\mapsto  \rmL'_\calc(\underline{M}'_{\underline{\Sigma}} \circ \underline{M}_{\underline{\Sigma}}).
 \end{aligned}
\end{equation}
Let $\widehat\rmV_{\calc}(\underline{\Sigma})$ be the quotient vector space of $ \widehat\calV(\underline{\Sigma})$ with respect to the right radical of the bilinear form $\langle \cdot , \cdot \rangle_{\underline{\Sigma}}$, and similarly let $\widehat\rmV'_{\calc}(\underline{\Sigma})$ be the quotient vector space of the $ \widehat\calV'(\underline{\Sigma})$ with respect to the left radical of the bilinear form $\langle \cdot , \cdot \rangle_{\underline{\Sigma}}$. 
We will abuse notation by denoting both projections as $[ {} \cdot {} ]$, i.e.\ $[ {} \cdot {} ] : \calV(\underline{\Sigma}) \to  \widehat\rmV_\calc(\underline{\Sigma})$ and $[ {} \cdot {} ] : \calV'(\underline{\Sigma}) \to  \widehat\rmV'_\calc(\underline{\Sigma})$,
    and by also denoting the pairing induced from \eqref{eq:curly-V-pairing} by
\begin{equation}\label{eq:straight-V-pairing}
 \langle \cdot , \cdot \rangle_{\underline{\Sigma}} \colon   \widehat\rmV'_{\calc}(\underline{\Sigma}) \otimes  \widehat\rmV_{\calc}(\underline{\Sigma}) \rightarrow \mathbbm{k} .
\end{equation}
This pairing is non-degenerate by construction.

For $\underline{M} \colon \underline{\Sigma} \rightarrow \underline{\Sigma}'$ a morphism of $\widehat{\bord}_{3,2}^{\chi}(\calc)$ let $ \widehat\rmV_{\calc}(\underline{M})$ be the linear map defined by
\begin{equation}
 \begin{aligned}
   \widehat\rmV_{\calc}(\underline{M}) \colon  \widehat\rmV_{\calc}(\underline{\Sigma}) & \rightarrow   \widehat\rmV_{\calc}(\underline{\Sigma}') \\
   [\underline{M}_{\underline{\Sigma}}] & \mapsto  [\underline{M} \circ \underline{M}_{\underline{\Sigma}}],
 \end{aligned}
\end{equation}
and similarly let $\widehat\rmV'_{\calc}(\underline{M})$ be the linear map defined by
\begin{equation}
 \begin{aligned}
   \widehat\rmV'_{\calc}(\underline{M}) \colon   \widehat\rmV'_{\calc}(\underline{\Sigma}') & \rightarrow   \widehat\rmV'_{\calc}(\underline{\Sigma}) \\
    [\underline{M}'_{\underline{\Sigma}'}] & \mapsto  [\underline{M}'_{\underline{\Sigma}'} \circ \underline{M}].
 \end{aligned}
\end{equation}
The construction clearly defines functors
\begin{equation}
 \widehat{\rmV}_{\calc} \colon \widehat{\bord}_{3,2}^{\chi}(\calc) \to \Vect_{\mathbbm{k}}, \qquad \widehat{\rmV}'_{\calc} \colon \widehat{\bord}_{3,2}^{\chi}(\calc)^{\op}  \to \Vect_{\mathbbm{k}},
\end{equation}
moreover from the equivalence $\widehat{\bord}_{3,2}^{\chi}(\calc)^{\op} \simeq \widecheck{\bord}_{3,2}^{\chi}(\calc)$
we further get
\begin{equation}
 \widecheck{\rmV}_{\calc} \colon \widecheck{\bord}_{3,2}^{\chi}(\calc) \to \Vect_{\mathbbm{k}}, \qquad \widecheck{\rmV}'_{\calc} \colon \widecheck{\bord}_{3,2}^{\chi}(\calc)^{\op}  \to \Vect_{\mathbbm{k}},
\end{equation}
where $\widecheck{\rmV}_{\calc} = \widehat{\rmV}'_{\calc} \circ (-)^*$ and $\widecheck{\rmV}'_{\calc} = \widehat{\rmV}_{\calc} \circ (-)^*$ and $(-)^*$ is the duality functor.
Note that we can restrict our attention to the covariant functors $\widehat{\rmV}_{\calc} \colon \widehat{\bord}_{3,2}^{\chi}(\calc) \to \Vect_{\mathbbm{k}}$ and $\widecheck{\rmV}_{\calc} \colon \widecheck{\bord}_{3,2}^{\chi}(\calc) \to \Vect_{\mathbbm{k}}$ because the contravariant ones can be recovered from them by precomposition with $(-)^*$. 

In fact, both functors define TFTs, with the non-trivial step being the proof of monoidality:
\begin{theorem}[{\autocite[Thm.\,4.12]{DGGPR19}}]
Let $\calc$ be a modular tensor category. The functors
\begin{equation}
\widehat{\rmV}_{\calc} \colon \widehat{\bord}_{3,2}^{\chi}(\calc)  \to \Vect_{\mathbbm{k}}
  \quad\text{and}\quad \widecheck{\rmV}_{\calc} \colon \widecheck{\bord}_{3,2}^{\chi}(\calc) \to \Vect_{\mathbbm{k}}
\end{equation}
are symmetric monoidal.
\end{theorem}
\begin{remark}\label{rem:statespacesspanned}
    \begin{enumerate}[(1)]
        \item
        By the universal construction for any $\underline{\Sigma} \in {\bord}_{3,2}^{\chi}(\calc)$ the state space $\widehat{\rmV}_\calc(\underline{\Sigma})$ and its linear dual $\widehat{\rmV}_\calc(\underline{\Sigma})^*$ are spanned by the vectors obtained from evaluating $\widehat{\rmV}_\calc$ on all bordisms $\varnothing \to \underline{\Sigma}$ and $\underline{\Sigma} \to \varnothing$, respectively. 
        \item 
        For any $\underline{\Sigma} \in {\bord}_{3,2}^{\chi}(\calc)$ we have 
        \begin{align}\label{eq:dualstatespace}
            \widecheck{\rmV}_{\calc}(\underline{\Sigma})^* \cong \widecheck{\rmV}'_{\calc}(\underline{\Sigma}) = \widehat{\rmV}_{\calc}(\underline{\Sigma}^*)
        \end{align}
        where $\underline{\Sigma}^*$ is the dual object of $\underline{\Sigma}$ in ${\bord}_{3,2}^{\chi}(\calc)$, and where the first isomorphism is induced by the non-degenerate pairing \eqref{eq:straight-V-pairing} (but for $\widecheck{\rmV}_{\calc}$). In this sense we get a pair of TFTs dual to each other. 
    \end{enumerate}
\end{remark}

\subsection{Algebraic state spaces}\label{sec:algstatespace}
We will now very briefly recall an algebraic model for the state spaces of $\widehat{\rmV}_\calc$ associated with connected objects of ${\bord}_{3,2}^{\chi}(\calc)$.

For this let $g,q,p\in\Z_{\geq0}$, we consider a standard closed connected surface $\Sigma_g$ of genus $g$, a handlebody $H_g$ with $\partial H_g = \Sigma_g$, and the Lagrangian subspace $\lambda_g \subset H_1(\Sigma_g)$ given by the kernel of the inclusion of $\Sigma_g$ into $H_g$. 
For a $(p+q)$-tuple of objects $(\underline{X};\underline{Y}) \equiv (X_1,\ldots,X_p,Y_1,\ldots,Y_q) \in \calc^{\times (p+q)}$, we denote with 
\begin{equation}\label{eq:standard-surf_gV}
\underline{\Sigma}_g^{(\underline{X};\underline{Y})} = (\Sigma_g,P_{(\underline{X};\underline{Y})},\lambda_g)
\end{equation}
the object of $\bord_{3,2}^{\chi}(\calc)$ with $p$ negatively oriented marked points decorated using the $X_i$ and $q$ positively oriented marked points decorated using the $Y_j$. For $f \in \Hom_\calc(\coend^{\otimes g} \otimes X_1 \otimes \ldots \otimes X_p ,Y_1 \otimes \ldots \otimes Y_q)$ we consider the admissible bordism $\left[\left(H_g,T_f,0\right)\right] \colon \varnothing \to \underline{\Sigma}_g^{(\underline{X};\underline{Y})}$ 
shown in \Cref{fig:standard-handle-body}.

\begin{figure}[tb]
\begin{equation*}
\left(H_g,T_f,0\right) =~ 
\raisebox{-.45\height}{\begin{tikzpicture}
        \filldraw[fill=lightgray, draw=black, thick] (-2,0.5) ellipse (1.7 and 2.5);
        \filldraw[fill=lightgray, draw=black, thick] (3,0.5) ellipse (1.7 and 2.5);
        \filldraw[lightgray] (-2,-2) rectangle (3,3);
        \draw[thick] (-2.01,-2) -- (3.01,-2);
        \draw[thick] (-2.01,3) -- (3.01,3);
        \filldraw[string-blue] (1.5,-1.4) circle (0.5mm);
        \draw[string-blue,thick,->] (1.5,-1.4) -- (1.15,-1.4);
       \filldraw[string-blue] (3,-1.4) circle (0.5mm);
       \draw[string-blue,thick,->] (3,-1.4) -- (2.65,-1.4);
        \filldraw[string-blue] (1.5,2.4) circle (0.5mm);
        \draw[string-blue,thick,->] (1.5,2.4) -- (1.85,2.4);
       \filldraw[string-blue] (3,2.4) circle (0.5mm);
       \draw[string-blue,thick,->] (3,2.4) -- (3.35,2.4);
       \draw[string-blue,thick] (1.5,-1.4) -- (1.5,0.2);
       \draw[string-blue,thick,->] (1.5,-1.4) -- (1.5,-0.6);
       \draw[string-blue,thick] (3,-1.4) -- (3,0.2);
       \draw[string-blue,thick,->] (3,-1.4) -- (3,-0.6);
       \draw[string-blue,thick] (1.5,0.8) -- (1.5,2.4);
       \draw[string-blue,thick,->] (1.5,0.8) -- (1.5,1.6);
       \draw[string-blue,thick] (3,0.8) -- (3,2.4);
       \draw[string-blue,thick,->] (3,0.8) -- (3,1.6);
       \draw[string-blue,thick] (0,-0.5) -- (0,0.2);
       \draw[string-red,thick] (-0.4,-0.5) -- (0.4,-0.5);
       \draw[string-red,thick] (-0.3,-0.5) -- (-0.3,-0.7);
       \draw[string-red,thick] (0.3,-0.5) -- (0.3,-0.7);
       \draw[string-red,thick] (-0.6,-1) arc (180:360:0.6 and 0.4);
       \draw[string-red,thick,->] (-0.6,-1) arc (180:275:0.6 and 0.4);
       \draw[string-red,thick] (-0.3,-0.7) .. controls (-0.3,-0.8) and (-0.6,-0.9) .. (-0.6,-1);
       \draw[string-red,thick] (0.3,-0.7) .. controls (0.3,-0.8) and (0.6,-0.9) .. (0.6,-1);
       \draw[string-blue,thick] (-1.9,-0.5) -- (-1.9,0.2);
       \draw[string-red,thick] (-2.3,-0.5) -- (-1.5,-0.5);
       \draw[string-red,thick] (-2.2,-0.5) -- (-2.2,-0.7);
       \draw[string-red,thick] (-1.6,-0.5) -- (-1.6,-0.7);
       \draw[string-red,thick] (-2.5,-1) arc (180:360:0.6 and 0.4);
       \draw[string-red,thick,->] (-2.5,-1) arc (180:275:0.6 and 0.4);
       \draw[string-red,thick] (-2.2,-0.7) .. controls (-2.2,-0.8) and (-2.5,-0.9) .. (-2.5,-1);
       \draw[string-red,thick] (-1.6,-0.7) .. controls (-1.6,-0.8) and (-1.3,-0.9) .. (-1.3,-1);
       \filldraw[fill=white, draw=string-blue, thick] (-2.4,0.2) rectangle (3.4,0.8);
       \filldraw[fill=white, draw=black, thick] (-1.5,-1) arc (0:-180:0.4 and 0.2);
       \fill[lightgray] (-2.3,-0.9) rectangle (-1.5,-1.09);
       \filldraw[fill=white, draw=black, thick] (-1.55,-1.1) arc (0:180:0.35 and 0.15);
       \draw[black, thick] (-1.5,-1) arc (0:-180:0.4 and 0.2);
       \filldraw[fill=white, draw=black, thick] (0.4,-1) arc (0:-180:0.4 and 0.2);
       \fill[lightgray] (-0.4,-0.9) rectangle (0.4,-1.09);
       \filldraw[fill=white, draw=black, thick] (0.35,-1.1) arc (0:180:0.35 and 0.15);
       \draw[black, thick] (0.4,-1) arc (0:-180:0.4 and 0.2);
       \node at (0.5,0.5) [string-blue] {$f$};
       \node at (1.55,-0.65) [string-blue,right] {\scriptsize $X_1$};
       \node at (3.05,-0.65) [string-blue,right] {\scriptsize $X_p$};
       \node at (1.55,1.55) [string-blue,right] {\scriptsize $Y_1$};
       \node at (3.05,1.55) [string-blue,right] {\scriptsize $Y_q$};
       \node at (0,-0.1) [string-blue,right] {\scriptsize $\coend$};
       \node at (-1.9,-0.1) [string-blue,right] {\scriptsize $\coend$};
       \node at (-0.9,-1.1) [thick] {$\dots$};
       \node at (2.4,-0.7) [string-blue] {$\dots$};
       \node at (2.4,1.5) [string-blue] {$\dots$};
   \end{tikzpicture} }
\end{equation*}
\caption{The bordism used to identify Hom-spaces in $\calc$ and state spaces of $\widehat{\rmV}_\calc$.}
\label{fig:standard-handle-body}
\end{figure}
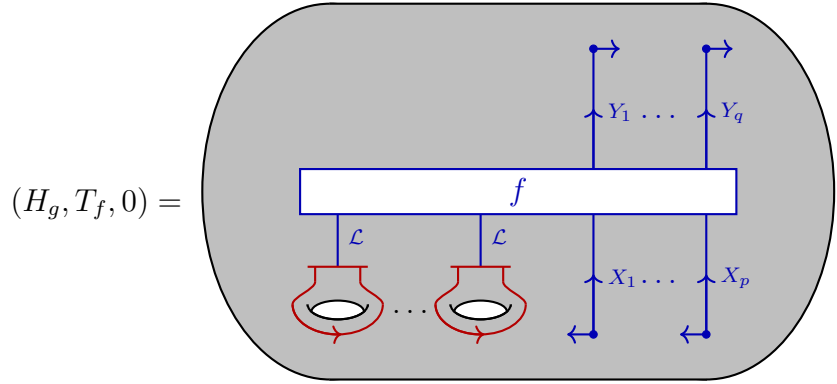

\begin{theorem}[{\autocite[Prop.\,4.17]{DGGPR19}}]
The map
\begin{equation}
\begin{aligned}
    \Phi_{\Sigma}^{(\underline{X};\underline{Y})}  \colon \Hom_\calc(\coend^{\otimes g} \otimes X_1 \otimes \ldots \otimes X_p ,Y_1 \otimes \ldots \otimes Y_q) &\to \widehat{\rmV}_\calc\left(\underline{\Sigma}_g^{(\underline{X};\underline{Y})} \right) \\
    f &\mapsto \widehat{\rmV}_\calc\left(\left[\left(H_g,T_f,0\right)\right] \right)
    (1_\kk)
\end{aligned}
\end{equation}
is a linear isomorphism.
\end{theorem}
\begin{remark}
    \begin{enumerate}[(1)]
    \item In \autocite[Sec.\,4]{DGGPR19} only positively orientated marked points are considered. Our setting is obtained from theirs by first declaring negatively oriented points to be labelled with the dual object and then using the isomorphism (\ref{eq:dualstatespace}) together with the natural isomorphism 
    \[
    \Hom_{\calc}(\coend^{\otimes g} \otimes X_1 \otimes \dots \otimes X_p, Y_1 \otimes \dots \otimes Y_q ) \cong \Hom_{\calc}(\coend^{\otimes g} \otimes X_1 \otimes \dots \otimes X_p \otimes Y_q^* \otimes \dots \otimes Y_1^*,\unit)
    \]
    induced by the rigidity and braiding of $\calc$. 
    \item A different, but related model of the state spaces can be given in terms skein modules \autocite[Sec.\,2.3.3]{DGGPR20}.
\end{enumerate}

\end{remark}

\subsection{Orientation reversal and Deligne products}\label{ssec:tftorideligne}
Finally, we will study how the behaviour of the $3$-manifold invariants under orientation reversal and the Deligne product -- discussed at the end of 
\Cref{ssec:3mf-inv} -- 
gives corresponding statements about the TFTs. For this we will need the following technical lemma, based on \autocite[Lem.\! 17.2]{TV}.\footnote{That lemma cannot be used in our setting directly because the admissible bordism categories are not rigid.} 
\begin{lemma}\label{lem:tftiso}
    Let $\mathcal{A}$ be a category with a fixed object $\unit$. Let $F,G \colon \mathcal{A} \to \vect$ be functors such that:
    \begin{enumerate}[(1)]
        \item for $H=F,G$ one has
        \begin{enumerate}[(a)]
            \item there exists a linear isomorphism $H^0 \colon \kk \to H(\unit)$;
            \item for all $A \in \mathcal{A}$ one has
            \begin{equation}           
            H(A) = \mathrm{span}_{\kk}\big\{\,H(f)(H^0(1_\kk))\,\big|\,f \in \Hom_{\mathcal{A}}(\unit,A)\big\};
            \end{equation}
            \item for all $A \in \mathcal{A}$ one has
            \begin{equation}
            H(A)^* = \mathrm{span}_{\kk}\big\{\, (H^0)^{-1} \circ H(g)\,\big|\,g \in \Hom_{\mathcal{A}}(A,\unit)\big\};
            \end{equation}
        \end{enumerate}
        \item for all $u \in \Hom_{\mathcal{A}}(\unit,\unit)$ the following diagram commutes,
\begin{align}
    \begin{tikzcd}[ampersand replacement=\&]
	\& {F(\unit)} \& {F(\unit)} \\
	\kk \&\&\& \kk \\
	\& {G(\unit)} \& {G(\unit)}
	\arrow["{F(u)}", from=1-2, to=1-3]
	\arrow["{G(u)}"', from=3-2, to=3-3]
	\arrow["{F^0}", from=2-1, to=1-2]
	\arrow["{G^0}"', from=2-1, to=3-2]
	\arrow["{(F^0)^{-1}}", from=1-3, to=2-4]
	\arrow["{(G^0)^{-1}}"', from=3-3, to=2-4]
\end{tikzcd}.
\end{align}
    \end{enumerate}
    Then for all $A \in \calA$ there exists a unique linear map $\phi_A \colon F(A) \to G(A)$ such that for all $f \in \Hom_\calA(\unit,A)$ the following diagram commutes,
\begin{align}\label{eq:defphi}
\begin{tikzcd}[ampersand replacement=\&]
	\& {F(\unit)} \& {F(A)} \\
	\kk \\
	\& {G(\unit)} \& {G(A)}
	\arrow["{F(f)}", from=1-2, to=1-3]
	\arrow["{G(f)}"', from=3-2, to=3-3]
	\arrow["{F^0}", from=2-1, to=1-2]
	\arrow["{G^0}"', from=2-1, to=3-2]
	\arrow["{\phi_A}", dashed, from=1-3, to=3-3]
\end{tikzcd}.
\end{align}
Furthermore, the collection $(\phi_A)_{A \in \calA}$ is a natural isomorphism $F \Rightarrow G$.
\end{lemma}
\begin{proof} 
    If $\phi_A$ exists, then it is an isomorphism by (1b). The proof of existence of $\phi_A$ will be very similar to the one in \autocite[Lem.\! 17.2]{TV}. Namely,
let $\kk\langle\Hom_\calA(\unit, A)\rangle$ be the vector space freely generated by $\Hom_\calA(\unit, A)$ and let    \begin{equation}
        \pi_A^H \colon \kk\langle\Hom_\calA(\unit, A)\rangle \to H(A)
    \end{equation}
    be the linear extension of the map
    \begin{equation}
            \Hom_\calA(\unit, A) \to H(A) 
            ,\quad
            f \mapsto H(f)(H^0(1_\kk)).
    \end{equation}

The maps $\pi_A^F$ and $\pi_A^G$ are surjective by (1b). We want to define $\phi_A$ such that $\phi_A \circ \pi_A^F = \pi_A^G$, which is well defined if $\pi_A^G|_{\mathrm{ker}(\pi_A^F)} = \{0\}$. To check this, let $v = \sum_{i=1}^n v_i f_i \in \mathrm{ker}(\pi_A^F)$ with $v_i \in \kk$ and $f_i \in \Hom_\calA(\unit,A)$. By (1c) it is enough to check that for all $g \in \Hom_\calA(A,\unit)$ we have that
\begin{equation}
    \begin{aligned}
        \sum_{i=1}^n v_i \, \big((G^0)^{-1} \circ G(g) \circ G(f_i)\big) (G^0(1_\kk)) = 0.
    \end{aligned}
\end{equation}
By (2) and functoriality of $F$ and $G$ the left hand side is equal to
\begin{equation}
    \begin{aligned}
        \sum_{i=1}^n v_i \, \big((F^0)^{-1} \circ F(g) \circ F(f_i) \big)(F^0(1_\kk)),
    \end{aligned}
\end{equation}
but this is zero, as already 
$\sum_{i=1}^n v_i \, F(f_i) (F^0(1_\kk)) = \pi_A^F(v) =0$. 
Hence $\phi_A$ is well defined and satisfies \eqref{eq:defphi}. 
The proof of naturality is now the same as in \autocite[Lem.\! 17.2]{TV}.
\end{proof}

For the actual application to TFT functors we will use the following corollary, which includes a slightly stronger uniqueness statement as commutativity of \eqref{eq:defphi} will be implied by monoidality.

\begin{corollary}\label{cor:tftiso}
If $\calA$ is monoidal, $\unit \in \calA$ is the monoidal unit, and $F,G$ are monoidal functors satisfying the conditions in Lemma~\ref{lem:tftiso}, then there exists a unique monoidal natural isomorphism $F \Rightarrow G$.
\end{corollary}
\begin{proof}
The natural isomorphism is the one constructed in Lemma~\ref{lem:tftiso} with $F^0$ and $G^0$ coming from the monoidal structure of $F$ and $G$, respectively. The proof of monoidality is the same as in \autocite[Lem.\! 17.2]{TV} as that part of the proof only uses \eqref{eq:defphi}. Uniqueness follows as in the proof of \cite[Lem.\,4.2]{CMRSS21RTorbi}.
\end{proof}

Let us now consider the effect of orientation reversal. First note that
$\widehat{\bord}_{3,2}^{\chi}(\overline{\calc}) = \widehat{\bord}_{3,2}^{\chi}(\calc)$ as $\calc$ and $\overline{\calc}$ have the same underlying pivotal monoidal category.
Consider the following symmetric monoidal functor
\begin{equation}
\begin{aligned}
    \overline{(-)} \colon \widehat{\bord}_{3,2}^{\chi}(\calc) &\to \widehat{\bord}_{3,2}^{\chi}(\calc) \\
    (\Sigma,P,\lambda) &\mapsto (-\Sigma,P,\lambda)
    \\
    [(M,T,n)] &\mapsto [(-M,T,n)]
\end{aligned}.
\end{equation}
 Note that this functor is not the same as the duality functor $(-)^*$, in particular it does not change the direction of composition.
 By combining \Cref{rem:statespacesspanned} with \Cref{lem:invori}, we can use \Cref{cor:tftiso} to get:
\begin{corollary}
    There is a unique natural
    monoidal isomorphism\footnote{Recall that there is no extra condition for a monoidal natural transformation to preserve the braiding as this is automatic.}
    \begin{equation}
        \widehat{\rmV}_{\overline{\calc}} \cong \widehat{\rmV}_{{\calc}} \circ \overline{(-)}.
    \end{equation}
\end{corollary}
Next we turn to the effect of taking Deligne products.
Let $\calc$ and $\calD$ be modular tensor categories.
We define $\widehat{\bord}_{3,2}^{\chi,\,\mathrm{fact}}(\calc\boxtimes\calD)$ as the following symmetric monoidal subcategory of $\widehat{\bord}_{3,2}^{\chi}(\calc\boxtimes\calD)$:
\begin{itemize}
    \item objects: $(\Sigma,P,\lambda)$ such that the labels of $P$ are of the form $X\boxtimes Y$ for some $X\in \calc$ and $Y\in \calD$;
    \item morphisms: $[(M,T,n)]$ such that all coupons of $T$ are pure tensors under the isomorphism
\begin{align}
        \Hom_{\calc \boxtimes \calD} (-\boxtimes -, -\boxtimes -)\cong \Hom_\calc(-,-) \otimes_\kk \Hom_{\calD}(-,-);
\end{align}
\end{itemize}
By forgetting either the label in $\calc$ or in $\calD$, one obtains symmetric monoidal functors $\widehat{\bord}_{3,2}^{\chi,\,\mathrm{fact}}(\calc\boxtimes\calD) \to \widehat{\bord}_{3,2}^{\chi}(\calc)$ and $\widehat{\bord}_{3,2}^{\chi,\,\mathrm{fact}}(\calc\boxtimes\calD) \to \widehat{\bord}_{3,2}^{\chi}(\calD)$. 
Pulling back along these functors, it is straightforward to define 
\begin{equation}
    \widehat{\rmV}_{\calc} \otimes_{\kk} \widehat{\rmV}_{\calD} \colon \widehat{\bord}_{3,2}^{\chi,\,\mathrm{fact}}(\calc\boxtimes\calD) \to \vect.
\end{equation}
Combining \Cref{cor:tftiso} with \Cref{lem:invdeligne} we conclude:
\begin{corollary}
    Let $\widehat{\rmV}_{\calc\boxtimes\calD}^{\mathrm{fact}}$ be the restriction of $\widehat{\rmV}_{\calc\boxtimes\calD}$ to $\widehat{\bord}_{3,2}^{\chi,\,\mathrm{fact}}(\calc\boxtimes\calD)$. Then there is a unique monoidal natural isomorphism
    \begin{equation}
        \widehat{\rmV}_{\calc\boxtimes\calD}^{\mathrm{fact}} \cong \widehat{\rmV}_{\calc} \otimes_{\kk} \widehat{\rmV}_{\calD}.
    \end{equation}
\end{corollary}

\section{Chiral modular functors}\label{sec:modfunc}
In this section we will describe how to use the $3$d TFTs from the previous section to construct a two dimensional \emph{chiral modular functor}. We are first going to discuss the definition of chiral modular functors as symmetric monoidal $2$-functors. Afterwards we will use the $3$d TFTs reviewed in the previous section to construct such a chiral modular functor. To this end, we will first define assignments on objects, $1$-morphisms, and $2$-morphisms, and then show that these assignments can be consistently turned into a $2$-functor by studying the image of certain $3$-bordisms under the TFT. 

Before coming to the definition of chiral modular functors we note that several variants of modular functors are used in the literature. We will use a $2$-categorical variant due to \autocite{FSY2023RCFTstring1} which is well suited for our applications. In this article a $2$-category will always mean a weak $2$-category otherwise known as a bicategory. Analogously a $2$-functor will always mean a weak $2$-functor with coherence isomorphisms often called pseudofunctor. See for example \autocite{JY20twodcategories}. For more background on symmetric monoidal $2$-categories and $2$-functors we refer to \autocite[Ch.\,2]{SchommerPries2011thesis}, a concise review is given in \autocite[App.\,D]{DeRenzi2022nonssext1}.

\subsection{Chiral bordisms and profunctors}\label{ssec:chiral-bord-prof}
Let us first describe the source and target $2$-categories.
The symmetric monoidal $2$-category $\bord_{2+\epsilon,2,1}^{\chi}$ of two dimensional \emph{chiral bordisms} consists of:
\begin{itemize}
    \item objects: closed one-dimensional manifolds, i.e.\ finite disjoint unions of the unit circle $S^1$;
    \item $1$-morphisms: for objects $\Gamma$ and $\Gamma'$, a $1$-morphism from $\Gamma$ to $\Gamma'$, denoted as $\Gamma \to \Gamma'$,
    is a tuple $(\Sigma,\lambda)$ where $\Sigma$ is two-dimensional bordism\footnote{Note that we have an explicit bordism here, not its diffeomorphism class.} $\Sigma \colon \Gamma \to \Gamma'$ and $\lambda \subset H_1(\overline{\Sigma};\R)$ 
    is a Lagrangian subspace with respect to the intersection pairing for $\overline{\Sigma}$ the closed two-dimensional manifold obtained from $\Sigma$ by gluing in all the disks bounding $\Gamma$ and $\Gamma'$; 
    \item $2$-morphisms: for $1$-morphisms $(\Sigma,\lambda) \colon \Gamma \to \Gamma'$    
    and $(\Sigma',\lambda') \colon \Gamma \to \Gamma'$, a $2$-morphism $(\Sigma,\lambda) \Rightarrow (\Sigma',\lambda')$ is a tuple $([f],n)$ where $n \in \Z$, and $[f]$ is the isotopy class of a diffeomorphism $f \colon \Sigma \to \Sigma'$ which is compatible with the boundary parametrisations;  
    \item horizontal composition: for $1$-morphisms $(\Sigma,\lambda) \colon \Gamma \to \Gamma'$ and $(\Sigma',\lambda') \colon \Gamma' \to \Gamma''$ the horizontal composition $(\Sigma',\lambda') \diamond (\Sigma,\lambda)\colon \Gamma \to \Gamma''$ is given by 
    \begin{equation}
    (\Sigma' \sqcup_{\Gamma'} \Sigma,\lambda' + \lambda),
    \end{equation}
    where we view $\lambda$ and $\lambda'$ as subspaces of $H_1(\overline{\Sigma' \sqcup_{\Gamma'} \Sigma};\R)$ via the inclusions $H_1(\overline{\Sigma};\R)\hookrightarrow H_1(\overline{\Sigma' \sqcup_{\Gamma'} \Sigma};\R)$ and $H_1(\overline{\Sigma'};\R)\hookrightarrow H_1(\overline{\Sigma' \sqcup_{\Gamma'} \Sigma};\R)$, respectively;
    \item identity $1$-morphism: for an object $\Gamma$ the identity $1$-morphism $\id_\Gamma \colon \Gamma \to \Gamma$ is 
    \begin{equation}
    (\Gamma\times I, \{0\})
    \end{equation} 
    as $\overline{\Gamma\times I}$ is a disjoint union of 2-spheres and so $H_1(\overline{\Gamma\times I};\R)=\{0\}$;    \item vertical composition: 
    for $2$-morphisms $([f],n) \colon (\Sigma,\lambda) \Rightarrow (\Sigma',\lambda')$ and $([g],m) \colon$ $(\Sigma',\lambda')$ $\Rightarrow$ $(\Sigma'',\lambda'')$,     
    the vertical composition is given by the $2$-morphism 
    \begin{equation}
    ([g\circ f],n+m-\mu((C_f)_*\lambda,\lambda',(C_g)^*\lambda''));\footnote{\textrm{Here $(C_f)$ and $(C_g)$ denote the diffeomorphism class of the three dimensional mapping cylinder bordisms coming from $[f]$ and $[g]$, and the new signature defect is computed as in \eqref{eq:anomaly-morphism-composition}.}}
    \end{equation}
    \item identity $2$-morphism: for a $1$-morphism $(\Sigma,\lambda) \colon \Gamma \to \Gamma'$ the identity $2$-morphism $\id_{(\Sigma,\lambda)} \colon (\Sigma,\lambda) \Rightarrow (\Sigma,\lambda)$ is the equivalence class of 
    \begin{equation}
    ([\id], 0);
    \end{equation} 
    \item disjoint union as symmetric monoidal structure;
\end{itemize}
There is more coherence data which needs to be specified such as the associativity $2$-morphisms for horizontal composition, however we will skip these details, see for example \autocite[App.\,B \& C]{DeRenzi2022nonssext1} as well as \autocite[Rem.\,2.2]{FSY2023RCFTstring1} for more details. Alternatively one could also define $\bord_{2+\epsilon,2,1}^{\chi}$ as the subcategory of the bordism category defined in \autocite[Sec.\,2.2]{DeRenzi2022nonssext1} with trivial colourings and only mapping cylinders as $2$-morphisms.

The endomorphisms of a surface $\Sigma$ form a central extension of the pure mapping class group of $\Sigma$ by $\Z$, where the extension corresponds to the signature defects, see also \autocite[Sec.\,3.1]{DGGPR20}.
Note here that we could define the $2$-morphisms as diffeomorphism classes of mapping cylinders for diffeomorphisms instead. This is because elements in the isotopy class of $f$ are in one to one correspondence to elements in the diffeomorphism class of the mapping cylinder $C_f$ of $f$.

In the following we will often suppress the Lagrangian subspaces and signature defects when they are not directly relevant. 

Before coming to the target $2$-category, recall that the Deligne tensor product $\calA \boxtimes \calB$ of two finite linear categories $\calA$ and $\calB$, see \Cref{ssec:3mf-inv}, also satisfies an analogous universal property for left exact functors, by the equivalence of left exact and right exact functors \autocite[Thm.\,3.2]{FSS19eilenbergwatts}, see also the discussion in \autocite[Sec.\,2.4]{BW22modfunc}.

The symmetric monoidal $2$-category $\cat{P}\mathrm{rof}_{\mathbbm{k}}^{\coend \mathrm{ex}}$ of \emph{linear, left exact profunctors} consists of
\begin{itemize}
    \item objects; finite linear categories;
    \item $1$-morphisms: for objects $\cat{A}$ and $\cat{B}$ a $1$-morphism from $\cat{A}$ to $\cat{B}$ is a left exact linear profunctor from $\cat{A} \slashedrightarrow \cat{B}$, i.e.\ a left exact linear functor $\cat{A}^{\mathrm{op}} \boxtimes \cat{B} \to \vect$;
    \item $2$-morphisms: natural transformations;
    \item horizontal composition: for $1$-morphisms $F \colon \cat{A} \slashedrightarrow \cat{B}$ and $G\colon \cat{B} \slashedrightarrow \calc$ the horizontal composition is the left exact coend
    \begin{align}
        G \diamond F (-,\sim) := \oint^{B \in \cat{B}}\! G(B,\sim) \otimes_{\mathbbm{k}} F(-,B)  \colon \cat{A} \slashedrightarrow \calc
    \end{align}
    where $-$ stands for an argument from the source category (in this case $\calA$) and $\sim$ for an argument from the target category (in this case $\calc$);
    \item identity $1$-morphism: for an object $\calA$ the identity $1$-morphism $\id_\calA \colon \calA \slashedrightarrow \calA$ is given by unique functor induced from the universal property of the Deligne product applied to the Hom functor $\Hom_\calA(-,\sim)$;
    \item vertical composition: standard vertical composition of natural transformations; 
    \item identity $2$-morphism: identity natural transformations;
    \item Deligne's tensor product of finite linear categories as the symmetric monoidal structure;
\end{itemize} 

For more details on profunctors see for example \autocite[Ch.\,5]{Loregian2021coendcalc}. 

Note that there is a $2$-functor from the $2$-category $\coend\mathrm{ex}_\kk$ of finite linear categories, left exact functors, and natural transformations to $\cat{P}\mathrm{rof}_{\mathbbm{k}}^{\coend \mathrm{ex}}$ which is the identity on objects, and sends a functor $F \colon \calA \to \calB$ to the profunctor 
$\Hom_{\calB}(F(-),\sim) \colon \calA \slashedrightarrow \calB$. It should be mentioned here that this is a contravariant $2$-functor, i.e.\ it reverses the direction of $2$-morphisms. Under our finiteness assumptions, this $2$-functor is actually an equivalence by Eilenberg-Watts, see \autocite[Lem.\,3.2]{Shimizu2016unimod}. 
The reason we use $\cat{P}\mathrm{rof}_{\mathbbm{k}}^{\coend \mathrm{ex}}$ as a target rather than $\coend\mathrm{ex}_\kk$ is that the former appears more naturally when we define the modular functor on $1$- and $2$-morphisms. 

In the following we will often use $-$ for both the source and target argument as it will be clear from the context which one is which, e.g.\ we will write $\Hom_\calA(-,-)$ instead of $\Hom_\calA(-,\sim)$.

\begin{definition} \label{def:modfuncchiral}
    A \emph{chiral modular functor} is a symmetric monoidal $2$-functor 
\begin{align}
    \mathrm{Bl}^{\chi}\colon\bord^{\chi}_{2+\epsilon,2,1} \to \cat{P}\mathrm{rof}_{\mathbbm{k}}^{\coend \mathrm{ex}}
\end{align}
from the symmetric monoidal $2$-category of two-dimensional chiral bordisms to the symmetric monoidal $2$-category of left exact profunctors.
\end{definition}

Throughout the rest of this section let $\calc$ be a fixed modular tensor category and denote with
\begin{align}
\widehat{\rmV}_{\calc} \colon \widehat{\bord}_{3,2}^{\chi}(\calc) \to \vect
\end{align}
the version of the TFT constructed from $\calc$ with admissibility condition on bordism components disjoint from the outgoing boundary. 

We will now construct a chiral modular functor using the TFT $\widehat{\rmV}_{\calc}$,
which we will call the \emph{chiral block functor} of $\calc$,
\begin{equation}
    \Blc\colon\bord^{\chi}_{2+\epsilon,2,1} \to \cat{P}\mathrm{rof}_{\mathbbm{k}}^{\coend \mathrm{ex}}.
\end{equation}
To do this we will first define the action of $\Blc$ on objects, $1$-morphisms, and $2$-morphisms, in \Cref{ssec:modfunc1mor} and \Cref{ssec:modfunc2mor}. By studying the gluing of surfaces in \Cref{ssec:modfunc3mor} we will then show that these assignments assemble into a symmetric monoidal $2$-functor in \Cref{thm:chiralmf}. 

We want to emphasise here that the definition of chiral modular functors and the following construction are $2$-categorical versions of the ones given in \autocite[Ch.\,5]{BK2001modular} if we restrict to modular fusion categories. The main technical difference lies in the treatment of gluing of surfaces which is more subtle here due to the appearance of coends.

\subsection{Block functors}\label{ssec:modfunc1mor}

\begin{definition}[Chiral block functor on objects]
For an object $\Gamma \in \bord_{2+\epsilon,2,1}^{\chi}$ we define $\Blc\left(\Gamma \right)$ to be the Deligne product over its connected components
\begin{equation}
    \Blc\left(\Gamma \right):= \calc^{\boxtimes \pi_0(\Gamma)} ,
\end{equation}
where we index the factors directly by the connected components $\pi_0(\Gamma)$ of $\Gamma$.
\end{definition}

In particular $\Blc(S^1) = \calc$, and by convention the empty product is $\vect$, so that $\Blc(\varnothing) := \vect$. 

Next we construct the profunctors describing the value on $1$-morphisms, called \emph{conformal block functors} in \autocite{FSY2023RCFTstring1}. 

Let $(\Sigma,\lambda) \colon \Gamma \to \Gamma'$ be a $1$-morphism in $\bord_{2+\epsilon,2,1}^{\chi}$. According to the above prescription on objects we want to construct a left exact functor 
\begin{equation}
\Blc\big((\Sigma,\lambda)\big) \colon (\calc^{\mathrm{op}})^{\boxtimes \pi_0(\Gamma)} \boxtimes \calc^{\boxtimes \pi_0(\Gamma')} \to \vect.
\end{equation}
First let us fix lists of objects $\underline{X} \in \calc^{\times \pi_0(\Gamma)}$ and $\underline{Y} \in \calc^{\times \pi_0(\Gamma')}$, where we again index the factors directly by the connected components.
We turn $(\Sigma,\lambda)$ into an object of $\widehat{\bord}_{3,2}^{\chi}(\calc)$ by gluing $|\pi_0(\Gamma)|+|\pi_0(\Gamma')|$ standard unit disks $D^2 \subset \mathbb{C}$ with specific marked points at the origin into the boundary of $\Sigma$. For any incoming boundary component $\gamma \in \pi_0(\Gamma)$, we mark the origin of the disk with a \emph{negatively oriented} point labelled with the corresponding element $X_{\gamma}$ in $\underline{X}$. For any outgoing boundary component $\gamma' \in \pi_0(\Gamma')$, we mark the origin of the disk with a \emph{positively oriented} point labelled with $Y_{\gamma'}$ in $\underline{Y}$. 
As tangent vectors at these points we use the unit vector along the positive real axis for positively oriented points, and the one along the negative real axis for negatively oriented points.
From this procedure we obtain an object
\begin{align}
    \underline{\Sigma}^{(\underline{X};\underline{Y})} := (\overline{\Sigma},   
    P_{(\underline{X},\underline{Y})},\lambda) \in \widehat{\bord}_{3,2}^{\chi}(\calc) \ ,
\end{align} 
see Figure\,\ref{fig:genus2} for an illustration.
\begin{figure}[t]
    \centering
    \begin{subfigure}{0.42\textwidth}
    \centering
        \begin{tikzpicture}
        \filldraw[fill=lightgray!70!white, draw=black, thick] (-1,-1) ellipse (2.5 and 1.5);
        \filldraw[white,opacity=0.7] (-1,-0.3) circle (0.25);
        \draw[thick,dashed] (-1,-0.3) circle (0.25);
        \filldraw[white,opacity=0.7] (-1.9,-1.7) circle (0.25);
        \draw[thick,dashed] (-1.9,-1.7) circle (0.25);
        \filldraw[white,opacity=0.7] (0,-1.7) circle (0.25);
        \draw[thick,dashed] (0,-1.7) circle (0.25);
       \filldraw[fill=white, draw=black, thick] (-1.5,-1) arc (0:-180:0.4 and 0.2);
       \fill[lightgray!70!white] (-2.3,-0.9) rectangle (-1.5,-1.09);
       \filldraw[fill=white, draw=black, thick] (-1.55,-1.1) arc (0:180:0.35 and 0.15);
       \draw[black, thick] (-1.5,-1) arc (0:-180:0.4 and 0.2);
       \filldraw[fill=white, draw=black, thick] (0.4,-1) arc (0:-180:0.4 and 0.2);
       \fill[lightgray!70!white] (-0.4,-0.9) rectangle (0.4,-1.09);
       \filldraw[fill=white, draw=black, thick] (0.35,-1.1) arc (0:180:0.35 and 0.15);
       \draw[black, thick] (0.4,-1) arc (0:-180:0.4 and 0.2);
   \end{tikzpicture} 
    \end{subfigure}
    \hspace{-0.5cm}
    \begin{subfigure}{0.15\textwidth}
        \centering
        \begin{tikzpicture}[baseline=0.3cm]
        \draw[thick,->](0,2) -- (1.5,2);
        \end{tikzpicture}
    \end{subfigure}
    \hspace{-0.5cm}
    \begin{subfigure}{0.42\textwidth}
    \centering
        \begin{tikzpicture}
        \filldraw[fill=lightgray!70!white, draw=black, thick] (-1,-1) ellipse (2.5 and 1.5);
        \filldraw[string-blue] (-1,-0.3) circle (0.05);
        \draw[string-blue,->] (-1,-0.3) -- (-0.75,-0.3);
        \filldraw[string-blue] (-1.9,-1.7) circle (0.05);
        \draw[string-blue,->] (-1.9,-1.7) -- (-2.15,-1.7);
        \filldraw[string-blue] (0,-1.7) circle (0.05);
        \draw[string-blue,->] (0,-1.7) -- (-0.25,-1.7);
       \filldraw[fill=white, draw=black, thick] (-1.5,-1) arc (0:-180:0.4 and 0.2);
       \fill[lightgray!70!white] (-2.3,-0.9) rectangle (-1.5,-1.09);
       \filldraw[fill=white, draw=black, thick] (-1.55,-1.1) arc (0:180:0.35 and 0.15);
       \draw[black, thick] (-1.5,-1) arc (0:-180:0.4 and 0.2);
       \filldraw[fill=white, draw=black, thick] (0.4,-1) arc (0:-180:0.4 and 0.2);
       \fill[lightgray!70!white] (-0.4,-0.9) rectangle (0.4,-1.09);
       \filldraw[fill=white, draw=black, thick] (0.35,-1.1) arc (0:180:0.35 and 0.15);
       \draw[black, thick] (0.4,-1) arc (0:-180:0.4 and 0.2);
       \node at (-1,-0.3) [string-blue,below] {\scriptsize $(Y,+)$};
       \node at (-1.9,-1.7) [string-blue,below] {\scriptsize $(X_1,-)$};
       \node at (0,-1.7) [string-blue,below] {\scriptsize $(X_2,-)$};
   \end{tikzpicture} 
    \end{subfigure}
    \caption{Turning a genus two surface with two incoming and one outgoing boundary circles $\Sigma_2^{(2,1)}$ into the object $\underline{\Sigma}_2^{(X_1,X_2;Y)}$.}
    \label{fig:genus2}
\end{figure}
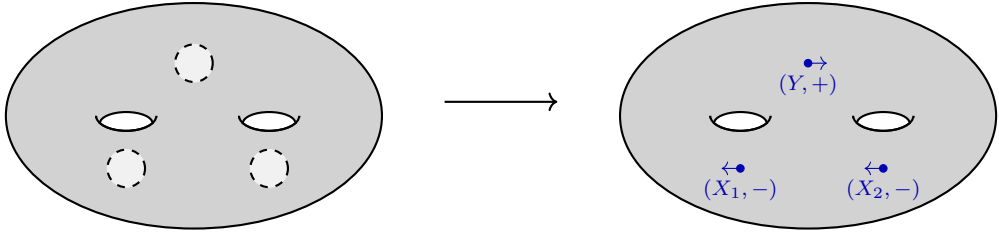

Now for lists of objects $\underline{X}' \in \calc^{\times \pi_0(\Gamma)}$ and $\underline{Y}' \in \calc^{\times \pi_0(\Gamma')}$ and morphisms $f_\gamma \colon X_\gamma' \to X_\gamma$ and $g_{\gamma'} \colon Y_{\gamma'} \to Y_{\gamma'}'$ with $\gamma \in \pi_0(\Gamma)$ and $\gamma' \in \pi_0(\Gamma')$ we define a bordism $\underline{C}^{(\underline{f};\underline{g})}$ in $\widehat{\bord}_{3,2}(\calc)$ as the class of $(\overline{\Sigma} \times [0,1], T_{(\underline{f};\underline{g})},0)$ with embedded ribbon graph $T_{(\underline{f};\underline{g})}$ given by locally embedding $f_\gamma$ (respectively $g_{\gamma'}$) in a cylinder $D_2 \times I$ over the boundary component $\gamma$ (respectively $\gamma'$), see \Cref{fig:mfmorph} for a local illustration.
\begin{figure}[t]
\centering
    \begin{subfigure}{0.45\linewidth}
        \centering
        \begin{tikzpicture}
        \filldraw[fill=lightgray!40!white, opacity=0.7, draw=black, dotted] (0,0) ellipse (2 and 1);
        \draw[string-blue,thick] (0,0) -- (0,4);
        \begin{scope}[very thick,decoration={
                      markings,
                      mark=at position 0.52 with {\arrow{>}}}
                     ] 
            \draw[string-blue,thick,postaction={decorate}]  (0,2) -- (0,0);
            \draw[string-blue,thick,postaction={decorate}]  (0,4) -- (0,2);
        \end{scope}
        \filldraw[fill=white,draw=string-blue] (-0.3,1.8) -- (-0.3,2.2) -- (0.3,2.2) -- (0.3,1.8) -- cycle;
        \draw[dashed] (0,0) ellipse (0.6 and 0.3)
        (0,4) ellipse (0.6 and 0.3)
        (-0.6,0) -- (-0.6,4)
        (0.6,0) -- (0.6,4);
        \filldraw[fill=lightgray!40!white, opacity=0.7, draw=black, dotted]
        (0,4) ellipse (2 and 1);
        \draw[dotted] (-2,0) -- (-2,4)
        (2,0) -- (2,4);
        \draw[dashed] (0,0) ellipse (0.6 and 0.3)
        (0,4) ellipse (0.6 and 0.3);
        \filldraw[string-blue] (0,0) circle (0.05);
        \draw[string-blue,->] (0,0) -- (-0.25,0);
        \filldraw[string-blue] (0,4) circle (0.05);
        \draw[string-blue,->] (0,4) -- (-0.25,4);
        \node at (0,0) [string-blue,right] {\scriptsize $X_{\gamma}$};
        \node at (0,4) [string-blue,right] {\scriptsize $X_{\gamma}'$};
        \node at (0,2) [string-blue] {\scriptsize $f_{\gamma}$};
        \node at (-1.3,0) {\scriptsize $\underline{\Sigma}^{\underline{X};\underline{Y}}$};
        \node at (-1.3,4) {\scriptsize $\underline{\Sigma}^{\underline{X}';\underline{Y}'}$};
        \node at (0.3,-0.45) {\small $\gamma$};
        \node at (0.3,3.55) {\small $\gamma$};
   \end{tikzpicture} 
        \caption{}
    \end{subfigure}
    \hspace{0.5cm}
    \begin{subfigure}{0.45\textwidth}
        \centering
        \begin{tikzpicture}
    \begin{scope}[xshift=1cm,yshift=1cm]
        \filldraw[fill=lightgray!40!white, opacity=0.7, draw=black, dotted] (-1,-1) ellipse (2 and 1);
        \draw[string-blue,thick] (-1,-1) -- (-1,3);
        \begin{scope}[very thick,decoration={
                      markings,
                      mark=at position 0.52 with {\arrow{>}}}
                     ] 
            \draw[string-blue,thick,postaction={decorate}]  (-1,-1) -- (-1,1);
            \draw[string-blue,thick,postaction={decorate}]  (-1,1) -- (-1,3);
        \end{scope}
        \filldraw[fill=white,draw=string-blue] (-1.3,0.8) -- (-1.3,1.2) -- (-0.7,1.2) -- (-0.7,0.8) -- cycle;
        \draw[dashed] (-1,-1) ellipse (0.6 and 0.3)
        (-1,3) ellipse (0.6 and 0.3)
        (-1.6,-1) -- (-1.6,3)
        (-0.4,-1) -- (-0.4,3);
        \filldraw[fill=lightgray!40!white, opacity=0.7, draw=black, dotted]
        (-1,3) ellipse (2 and 1);
        \draw[dotted] (-3,-1) -- (-3,3)
        (1,-1) -- (1,3);
        \draw[dashed] (-1,-1) ellipse (0.6 and 0.3)
        (-1,3) ellipse (0.6 and 0.3);
        \filldraw[string-blue] (-1,-1) circle (0.05);
        \draw[string-blue,->] (-1,-1) -- (-0.75,-1);
        \filldraw[string-blue] (-1,3) circle (0.05);
        \draw[string-blue,->] (-1,3) -- (-0.75,3);
        \node at (-0.9,-1) [string-blue,left] {\scriptsize $Y_{\gamma'}$};
        \node at (-0.9,3) [string-blue,left] {\scriptsize $Y_{\gamma'}'$};
        \node at (-0.95,1) [string-blue] {\scriptsize $g_{\gamma'}$};
        \node at (-2.3,-1) {\scriptsize $\underline{\Sigma}^{\underline{X};\underline{Y}}$};
        \node at (-2.3,3) {\scriptsize $\underline{\Sigma}^{\underline{X}';\underline{Y}'}$};
        \node at (-0.7,-1.45) {\small $\gamma'$};
        \node at (-0.7,2.55) {\small $\gamma'$};
        \end{scope}
   \end{tikzpicture} 
        \caption{}
    \end{subfigure}
    \caption{Local ribbon graphs for (a) incoming boundary $\gamma$ and (b) outgoing boundary components $\gamma'$.}\label{fig:mfmorph}
\end{figure}
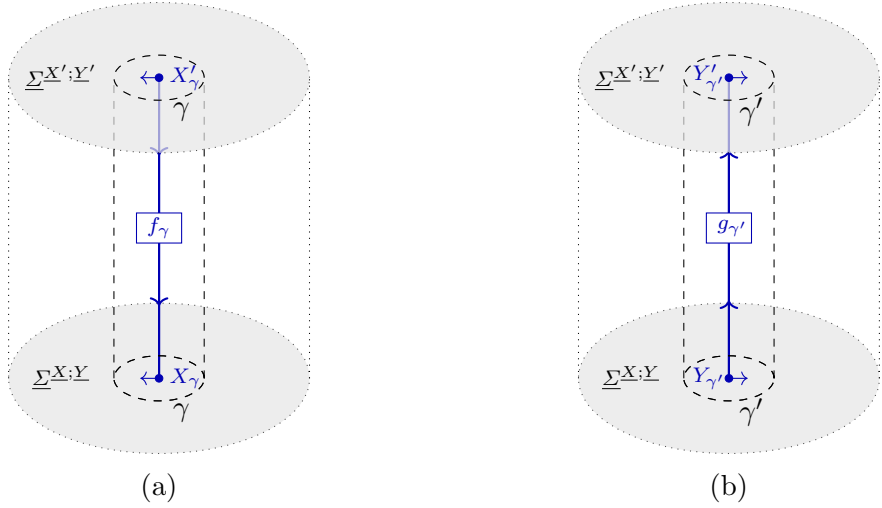
Note that $\underline{C}^{(\underline{f};\underline{g})}$ is indeed admissible.
\begin{lemma}\label{lem:bllex}
    The assignment
\begin{equation}\label{eq:bllex-def}
    \begin{aligned}
     \mathrm{bl}^{\chi}_{\calc} \big((\Sigma,\lambda)\big)
     \colon (\calc^{\mathrm{op}})^{\times\pi_0(\Gamma)} \times \calc^{\times \pi_0(\Gamma')} &\to \vect \\    
     (\underline{X};\underline{Y}) &\mapsto \widehat{\rmV}_\calc(\underline{\Sigma}^{(\underline{X};\underline{Y})}), \\
    (\underline{f};\underline{g}) &\mapsto \widehat{\rmV}_\calc(\underline{C}^{(\underline{f};\underline{g})}).
    \end{aligned}
\end{equation}
defines a functor left exact in each argument.
\end{lemma}
\begin{proof}
    Functoriality follows from the functoriality of $\widehat{\rmV}_\calc$, note that contravariance in the incoming components comes from the orientation of the ribbons.
    
    For left exactness let us first consider the case where $\Sigma$ is connected. Let $p = |\pi_0(\Gamma)|$ and $q = |\pi_0(\Gamma')|$ and let us fix an explicit choice of ordering of boundary components. Next recall the family of isomorphisms 
    \begin{equation}
    \begin{aligned}\label{eq:isoblockhom}
        \Phi_{\Sigma}^{ (\underline{X};\underline{Y})} \colon \Hom_{\calc}(\coend^{\otimes g} \otimes X_1 \otimes \dots \otimes X_p\,, Y_1 \otimes \dots \otimes Y_q ) \to \widehat{\rmV}_\calc  
        \left(\underline{\Sigma}_g^{(\underline{X};\underline{Y})}\right),
    \end{aligned}
    \end{equation}
    from Section\,\ref{sec:algstatespace}.
    A straightforward computation shows that this family defines a natural isomorphism 
    \begin{equation}
    \Hom_{\calc}(\coend^{\otimes g} \otimes - \otimes \dots \otimes -, - \otimes \dots \otimes - ) \Longrightarrow \mathrm{bl}^{\chi}_{\calc}\big((\Sigma,\lambda)\big)(-,\dots,-;-,\dots,-).
    \end{equation}
    Thus $\mathrm{bl}^{\chi}_{\calc}\big((\Sigma,\lambda)\big)$ is representable and therefore in particular left exact in each argument. 
    For the general case it suffices to consider two connected components, i.e.\ $\Sigma = \Sigma_1 \disjun \Sigma_2$ with $\Sigma_1$ and $\Sigma_2$ connected. From the monoidality of $\widehat{\rmV}_\calc$ we get 
    \begin{equation}
        \mathrm{bl}^{\chi}_{\calc}(\Sigma) = \mathrm{bl}^{\chi}_{\calc}(\Sigma_1 \disjun \Sigma_2) \cong \mathrm{bl}^{\chi}_{\calc}(\Sigma_1) \otimes_{\mathbbm{k}} \mathrm{bl}^{\chi}_{\calc}(\Sigma_2).
    \end{equation}
    As $\mathrm{bl}^{\chi}_{\calc}(\Sigma_1)$ and $\mathrm{bl}^{\chi}_{\calc}(\Sigma_2)$ are left exact by the above argument also $\mathrm{bl}^{\chi}_{\calc}(\Sigma)$ is left exact.
    \end{proof}
\begin{definition}[Chiral block functor on 1-morphisms]
For $(\Sigma,\lambda) \colon \Gamma \to \Gamma'$ we define 
    \begin{align}
        \Blc\big((\Sigma,\lambda)\big) \colon (\calc^{\mathrm{op}})^{\boxtimes \pi_0(\Gamma)} \boxtimes  \calc^{\boxtimes  \pi_0(\Gamma')} \to \vect
    \end{align}
to be the functor induced by $\mathrm{bl}^{\chi}_{\calc}\big((\Sigma,\lambda)\big)$ in \eqref{eq:bllex-def} on the Deligne product.
\end{definition}

\begin{corollary}
    For any pair of $1$-morphisms $\Sigma$ and $\Sigma'$ the chiral block functor satisfies
\begin{equation}
    \mathrm{Bl}^{\chi}_{\calc}(\Sigma \sqcup \Sigma') \cong \mathrm{Bl}^{\chi}_{\calc}(\Sigma) \otimes_{\mathbbm{k}} \mathrm{Bl}^{\chi}_{\calc}(\Sigma').
\end{equation}
\end{corollary}

\subsection{Mapping class group actions}\label{ssec:modfunc2mor}

Let us now turn to the $2$-morphism level. Let $(\Sigma,\lambda)$ and $(\Sigma',\lambda')$ be $1$-morphisms from $\Gamma$ to $\Gamma'$ and let $([f],n)$ be a $2$-morphism from $(\Sigma,\lambda)$ to $(\Sigma',\lambda')$. We construct decorated surfaces $\underline{\Sigma}^{(\underline{X};\underline{Y})}$ and $\underline{\Sigma'}^{(\underline{X};\underline{Y})}$ as above for $\underline{X} \in \calc^{\times \pi_0(\Gamma)}$ and $\underline{Y} \in \calc^{\times \pi_0(\Gamma')}$. The diffeomorphism $f$ underlying $([f],n)$ gives rise to a $3$-manifold with corners, its mapping cylinder $C_f$. This only depends on the isotopy class of $[f]$ if we are considering the diffeomorphism class of the $3$-manifold $C_f$. Analogously to the construction of $\mathrm{Bl}^{\chi}_{\calc}$ on surfaces we can glue in solid cylinders with embedded ribbon graph labelled by the $\underline{X} \in \calc^{\times \pi_0(\Gamma)}$ and $\underline{Y} \in \calc^{\times \pi_0(\Gamma')}$ to obtain a family of morphisms $(\underline{C}_f^{(\underline{X};\underline{Y})},n) \colon \underline{\Sigma}^{(\underline{X};\underline{Y})} \to \underline{\Sigma'}^{(\underline{X};\underline{Y})}$ in $\widehat{\bord}_{3,2}^{\chi}(\calc)$. Applying the TFT $\widehat{\rmV}_\calc$ gives us a family of linear maps 
\begin{align}
    \widehat{\rmV}_\calc((\underline{C}_f^{(\underline{X};\underline{Y})}),n) \colon \mathrm{bl}^{\chi}_{\calc}(\Sigma)(\underline{X};\underline{Y}) \to \mathrm{bl}^{\chi}_{\calc}(\Sigma')(\underline{X};\underline{Y}),
\end{align}
which is natural in the labels by construction. We now set
\begin{equation}
    \mathrm{bl}^{\chi}_{\calc}\big(([f],n)\big) := 
    \left(\widehat{\rmV}_\calc((\underline{C}_f^{(\underline{X};\underline{Y})}),n)\right)_{\underline{X} \in \calc^{\times \pi_0(\Gamma)},\underline{Y} \in \calc^{\times \pi_0(\Gamma')}}.
\end{equation}

\begin{definition}[Chiral block functor on 2-morphisms]
For a 2-morphism $([f],n) \colon (\Sigma,\lambda) \Rightarrow (\Sigma',\lambda')$, 
\begin{equation}
    \Blc\big(([f],n)\big) \colon \Blc\big((\Sigma,\lambda)\big) \Rightarrow \Blc\big((\Sigma',\lambda')\big)
\end{equation} 
is defined as the natural isomorphism induced by $\mathrm{bl}^{\chi}_{\calc}\big(([f],n)\big)$ on the Deligne product. 
\end{definition}

\begin{remark}\label{rem:gluing}
\begin{enumerate}[(1)]
\item Let us briefly discuss how this is related to \autocite{DGGPR20}. First note that any $2$-endomorphism $[f]$ of $\Sigma$ gives an element in (a central extension of) the mapping class group of the decorated surface $\underline{\Sigma}^{(\underline{X};\underline{Y})}$ discussed in \autocite[Sec.\,3.1]{DGGPR20} because the underlying diffeomorphism $f$ needs to be compatible with the boundary parametrisation, see also \autocite[Prop.\,5.1.8]{BK2001modular}. Thus we can associate it to a mapping cylinder endomorphism of $\underline{\Sigma}^{(\underline{X};\underline{Y})}$ defined in \autocite[Sec.\,3.2]{DGGPR20}, however this needs to be in the same diffeomorphism class as $(\underline{C}_f^{(\underline{X};\underline{Y})},n)$ by \autocite[Prop.\,5.1.8]{BK2001modular}. Thus our construction and the one of \autocite{DGGPR20} are compatible with each other. 
\item Recall from \Cref{sec:tqftcon} that the signature defect $n$ enters in the construction of $\widehat{\rmV}_\calc$ via the coefficient $\delta$, defined in equation \ref{eq:defdelta}. If $\delta = 1$, i.e.\ if $\calc$ is anomaly free, then $\widehat{\rmV}_\calc$, and thus $\mathrm{Bl}_{\calc}^{\chi}$, will give genuine linear representations of the unextended mapping class group, see also \autocite[Sec.\,3.2]{DGGPR20}.
\end{enumerate}
    
\end{remark}

\subsection{Gluing of surfaces}\label{ssec:modfunc3mor}

We now turn to the compatibility of $\mathrm{Bl}^{\chi}_{\calc}$ with horizontal composition, which means for $\Sigma_1 \colon \Gamma_1 \to \Gamma$ and $\Sigma_2 \colon \Gamma \to \Gamma_2$ composable $1$-morphisms we want to show that the profunctor $\mathrm{Bl}^{\chi}_{\calc}(\Sigma_1 \sqcup_{\Gamma} \Sigma_2)$ is naturally isomorphic to the composition $\Blc(\Sigma_2)\diamond \Blc(\Sigma_1)$.
Note that to include situations where we do not want to glue along the whole outgoing boundary of $\Sigma_1$ or the whole incoming boundary of $\Sigma_2$ we can always modify the $\Sigma_i$ by taking the disjoint union with sufficiently many copies of the cylinder $S^1\times I$. 

Let $\underline{W} \in \calc^{\times \pi_0(\Gamma_1)}$, $\underline{X} \in \calc^{\times \pi_0(\Gamma)}$, and $\underline{Y} \in \calc^{\times \pi_0(\Gamma_2)}$ . We will employ the shorthand notation $\underline{\Sigma_1} \equiv \underline{\Sigma_1}^{(\underline{W};\underline{X})}$, $\underline{\Sigma_2} \equiv \underline{\Sigma_2}^{(\underline{X};\underline{Y})}$, and $\underline{\Sigma_1 \sqcup_{\Gamma} \Sigma_2} \equiv \underline{\Sigma_1 \sqcup_{\Gamma} \Sigma_2}^{(\underline{W};\underline{Y})}$. 

We define a family of bordisms 
\begin{equation}
    M_{\underline{X}} \colon \underline{\Sigma_1} \sqcup \underline{\Sigma_2} \to \underline{\Sigma_1 \sqcup_{\Gamma} \Sigma_2} 
\end{equation}
in $\widehat{\bord}_{3,2}^{\chi}(\calc)$ with underlying $3$-manifold
\begin{equation}
    M_{\underline{X}} = (\underline{\Sigma_1 \disjun \Sigma_2} \times I)/\sim
\end{equation}
where $\sim$ identifies on $\underline{\Sigma_1} \disjun \underline{\Sigma_2} \times \{1\}$ the disks glued into the components of $\Gamma$ on $\Sigma_1$ and $\Sigma_2$.

On handle bodies the action of the bordism $M_{\underline{X}}$ is the same as attaching a handle with an embedded $X_\gamma$-labelled ribbon to the neighbourhood of the marked points labelled with $X_\gamma$ for any $\gamma \in \pi_0(\Gamma)$. See Figure \ref{fig:gluingbordism} for a local visualisation.
\begin{figure}[t]
        \centering
        \begin{tikzpicture}[scale=0.8]
       \draw[thick] (-0.5,0) arc (90:450: 0.2 and 0.5);
       \draw[thick] (-0.5,0) .. controls (-2.5,0.3) .. (-3.5,1);
       \draw[thick,dotted] (-3.5,1) -- (-4,1.4);
       \draw[thick] (-0.5,-1) .. controls (-2.5,-1.3) .. (-3.5,-2);
       \draw[thick,dotted] (-3.5,-2) -- (-4,-2.4);
       \draw[thick] (1,0) arc (90:270: 0.2 and 0.5);
       \draw[thick,dotted] (1,-1) arc (270:450: 0.2 and 0.5);
       \draw[thick] (1,0) .. controls (3,0.3) .. (4,1);
       \draw[thick,dotted] (4,1) -- (4.5,1.4);
       \draw[thick] (1,-1) .. controls (3,-1.3) .. (4,-2);
       \draw[thick,dotted] (4,-2) -- (4.5,-2.4);
       \node at (-2.5,-0.5) {$\Sigma_1$};
       \node at (3,-0.5) {$\Sigma_2$};
       \node at (-0.5,-1.2) {$\Gamma$};
       \node at (1,-1.2) {$\Gamma$};
       \node at (0.25,1) {$\bord^{\chi}_{2+\epsilon,2,1}$};
       \draw[thick] (11,0.2) .. controls (9,0.2) .. (8,1);
       \draw[thick,dotted] (8,1) -- (7.5,1.4);
       \draw[thick] (11,-1.2) .. controls (9,-1.2) .. (8,-2);
       \draw[thick,dotted] (8,-2) -- (7.5,-2.4);
       \draw[thick] (11,0.2) .. controls (13,0.2) .. (14,1);
       \draw[thick,dotted] (14,1) -- (14.5,1.4);
       \draw[thick] (11,-1.2) .. controls (13,-1.2) .. (14,-2);
       \draw[thick,dotted] (14,-2) -- (14.5,-2.4);
       \node at (11.25,-0.5) {$\Sigma_1 \sqcup_{\Gamma} \Sigma_2$};
       \draw[thick] (-0.5,-7) arc (90:-90: 0.4 and 0.5);
       \draw[thick] (-0.5,-7) .. controls (-2.5,-6.7) .. (-3.5,-6);
       \draw[thick,dotted] (-3.5,-6) -- (-4,-5.6);
       \draw[thick] (-0.5,-8) .. controls (-2.5,-8.3) .. (-3.5,-9);
       \draw[thick,dotted] (-3.5,-9) -- (-4,-9.4);
       \filldraw[string-blue] (-0.1,-7.5) circle (0.05);
       \draw[thick] (1,-7) arc (90:270: 0.4 and 0.5);
       \draw[thick] (1,-7) .. controls (3,-6.7) .. (4,-6);
       \draw[thick,dotted] (4,-6) -- (4.5,-5.6);
       \draw[thick] (1,-8) .. controls (3,-8.3) .. (4,-9);
       \draw[thick,dotted] (4,-9) -- (4.5,-9.4);
       \filldraw[string-blue] (0.6,-7.5) circle (0.05);
       \node at (-2.5,-7.5) {$\underline{\Sigma_1}$};
       \node at (3,-7.5) {$\underline{\Sigma_2}$};
       \node at (0.25,-9) {$\widehat{\bord}^{\chi}_{3,2}(\calc)$};
       \node at (-0.1,-7.5) [string-blue,left] {\scriptsize $(X,+)$};
       \node at (0.6,-7.5) [string-blue,right] {\scriptsize $(X,-)$};
       \begin{scope}[xshift=0.5cm]
       \fill[lightgray] (7.5,-10) rectangle (14.5,-5);
       \fill[white] (7.4,-4.98) -- (11,-6.02) -- (14.6,-4.98) -- cycle;
       \fill[white] (11,-6) .. controls (9,-6) .. (7.5,-5)
        (11,-6) .. controls (13,-6) .. (14.5,-5)
        (14.5,-5) -- (7.5,-5);
       \draw[thick] (11,-6) .. controls (9,-6) .. (7.5,-5)
        (11,-6) .. controls (13,-6) .. (14.5,-5);
       \fill[white] (7.4,-10.02) -- (11,-8.98) -- (14.6,-10.02) -- cycle;
       \fill[white] (11,-9) .. controls (9,-9) .. (7.5,-10)
        (11,-9) .. controls (13,-9) .. (14.5,-10)
        (14.5,-5) -- (7.5,-5);
       \draw[thick] (11,-9) .. controls (9,-9) .. (7.5,-10)
        (11,-9) .. controls (13,-9) .. (14.5,-10);
       \fill[lightgray!40] (10.02,-7.01) -- (10.02,-7.99) -- (7.5,-8.48) -- (7.5,-6.52) --cycle;
       \filldraw[fill=lightgray!40,thick] (10,-7) arc (90:-90:0.4 and 0.5);
       \filldraw[fill=lightgray,thick]
        (10,-7) .. controls (8.5,-7) .. (7.5,-6.5)
        (10,-8) .. controls (8.5,-8) .. (7.5,-8.5);
       \fill[lightgray!40] (11.98,-7.01) -- (11.98,-7.99) -- (14.5,-8.48) -- (14.5,-6.52) --cycle;
        \filldraw[fill=lightgray!40,thick] (12,-7) arc (90:270:0.4 and 0.5);
       \filldraw[fill=lightgray,thick]
        (12,-7) .. controls (13.5,-7) .. (14.5,-6.5) 
        (12,-8) .. controls (13.5,-8) .. (14.5,-8.5);
        \filldraw[string-blue] (10.4,-7.5) circle (0.05)
        (11.6,-7.5) circle (0.05);
        \begin{scope}[very thick,decoration={
                      markings,
                      mark=at position 0.52 with {\arrow{>}}}
                     ] 
            \draw[string-blue,thick,postaction={decorate}] (11.6,-7.5) .. controls (11,-7.3) .. (10.4,-7.5);
        \end{scope}
       \node at (10.4,-7.5) [string-blue,left] {\scriptsize $X$};
       \node at (11.6,-7.5) [string-blue,right] {\scriptsize $X$};
       \node at (11.25,-9.5) {$M_{\underline{X}}$};
       \end{scope}
       \draw[thick,dashed,<->] (0.25,-6) -- (0.25,-2);
       \draw[thick,->] (4.5,-0.5) -- (7.5,-0.5);
       \draw[thick,->] (4.5,-7) -- (10,-2);
       \draw[thick,dashed,<->] (8,-4) .. controls (9,-4.5) and (10.5,-4) .. (11,-5.5);
       \node at (6,-0.5) [above] {glue along $\Gamma$};
   \end{tikzpicture} 
        \caption{Schematic visualisation of the gluing procedure to obtain a family of morphisms in $\widehat{\bord}_{3,2}^{\chi}(\calc)$.}\label{fig:gluingbordism}
\end{figure}
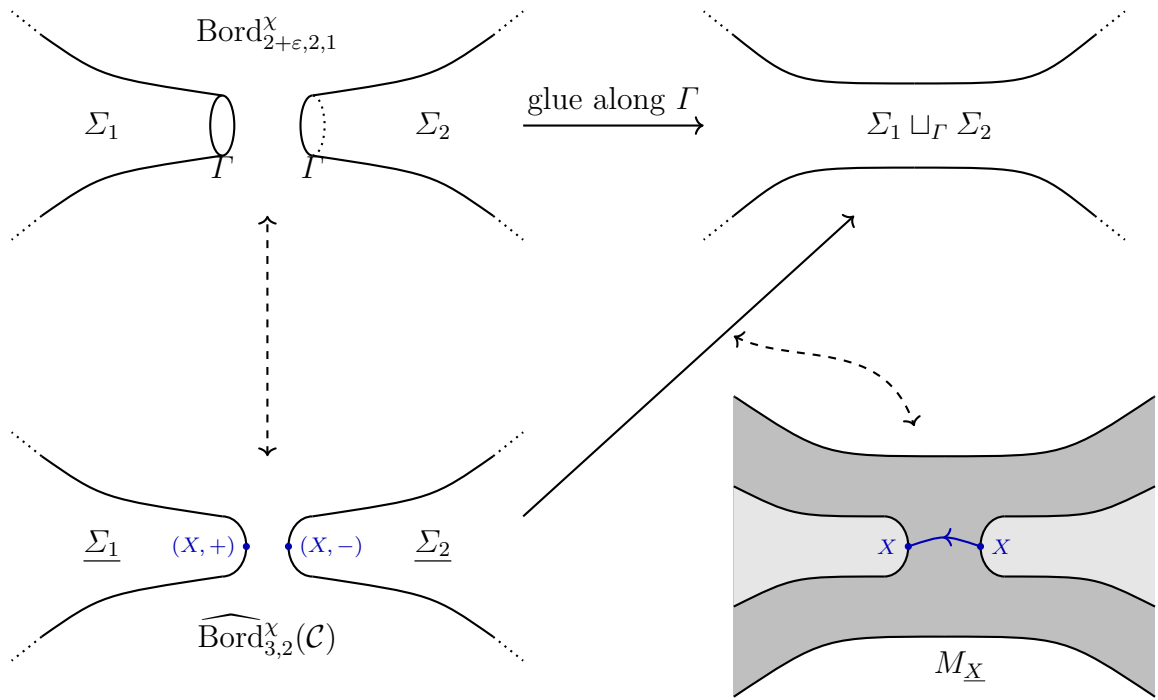

This family of bordisms is by definition dinatural in the labels $\underline{X}$ as well as natural in the $\underline{W}$ and $\underline{Y}$. In the rest of this section we will study the image of the so obtained family under the TFT $\widehat{V}_{\calc}$. 

Since gluing is a local procedure we can obtain $\Sigma_1 \sqcup_\Gamma \Sigma_2$ by consecutively gluing along each connected component of $\Gamma$ one at a time. Moreover the order in which these gluings are performed is irrelevant. Formally this amounts to the functoriality and associativity of horizontal composition in $\bord_{2+\epsilon,2,1}^{\chi}$. We can thus restrict our attention to one gluing. Let us fix an order of gluings and let $\Sigma$ denote the surface obtained after gluing $\Sigma_1$ and $\Sigma_2$ along $(|\pi_0(\Gamma)|-1)$ components of $\Gamma$. We now need to distinguish the following two scenarios:
\begin{enumerate}
    \item We glue boundary components on two different components of $\Sigma$.
    \item We glue boundary components on a connected component of $\Sigma$.
\end{enumerate}
Both of these scenarios are different from a global perspective, as can be seen by the example of gluing two disjoint cylinders to a torus: 
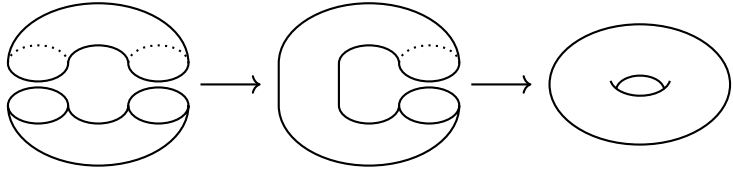
\begin{align*}
    \begin{tikzpicture}[scale=0.8]
       \draw[thick] (0,0) arc (0:360:0.5 and 0.3) 
       (2,0) arc (0:360:0.5 and 0.3)
       (0,0) arc (180:360:0.5 and 0.3)
       (-1,0) arc (180:360:1.5 and 1)
       (0,0.7) arc (0:-180:0.5 and 0.3)
       (2,0.7) arc (0:-180:0.5 and 0.3)
       (0,0.7) arc (180:0:0.5 and 0.3)
       (-1,0.7) arc (180:0:1.5 and 1);
       \draw[thick,dotted] (0,0.7) arc (0:180:0.5 and 0.3)
       (2,0.7) arc (0:180:0.5 and 0.3);
       \begin{scope}[xshift=4.5cm]
           \draw[thick] (2,0) arc (0:360:0.5 and 0.3)
       (0,0) arc (180:360:0.5 and 0.3)
       (-1,0) arc (180:360:1.5 and 1)
       (2,0.7) arc (0:-180:0.5 and 0.3)
       (0,0.7) arc (180:0:0.5 and 0.3)
       (-1,0.7) arc (180:0:1.5 and 1)
       (0,0) -- (0,0.7)
       (-1,0) -- (-1,0.7);
       \draw[thick,dotted] (2,0.7) arc (0:180:0.5 and 0.3);
       \end{scope}
       \begin{scope}[xshift=9.5cm,yshift=0.35cm]
           \draw[thick] (0,0) ellipse (1.5 and 1)
           (0.5,0.06) arc (-10:-170:0.5 and 0.3)
           (0.4,-0.06) arc (10:170:0.4 and 0.25);
       \end{scope}
       \draw[thick,->] (2.2,0.35) -- (3.2,0.35);
       \draw[thick,->] (6.7,0.35) -- (7.7,0.35);
   \end{tikzpicture} 
\end{align*}
\noindent
Gluing along the left pair of incoming and outgoing boundary circles results in a cylinder so that the other two boundary components now lie on the same connected component. The second gluing then identifies the two ends of the cylinder to arrive at the torus.

We can now formulate the main result of this section.
\begin{proposition}\label{prop:gluingfunc}
    Let $\Sigma$ be a surface with at least one incoming and one outgoing boundary, and let $\Sigma_{\mathrm{gl}}$ be the surface obtained from gluing these boundaries. Then the dinatural family
    \begin{align}
        \eta_X \colon \mathrm{Bl}^{\chi}_{\calc}(\Sigma)(X,X) \to \mathrm{Bl}^{\chi}_{\calc}(\Sigma_{\mathrm{gl}})
    \end{align}
    obtained from the gluing bordisms $M_X$ is universal in the category of left exact functors, in other words
    \begin{align}
        \mathrm{Bl}^{\chi}_{\calc}(\Sigma_{\mathrm{gl}}) \cong \oint^{X\in \calc }\! \mathrm{Bl}^{\chi}_{\calc}(\Sigma)(X,X).
    \end{align}
\end{proposition}

\begin{corollary}The chiral block functors are compatible with horizontal composition
    \begin{equation}
        \mathrm{Bl}^{\chi}_{\calc}(\Sigma_1 \sqcup_\Gamma \Sigma_2) \cong \Blc(\Sigma_2)\diamond \Blc(\Sigma_1).
    \end{equation}
\end{corollary}

To prove \Cref{prop:gluingfunc} we will now study the two scenarios discussed above separately. 
\subsubsection*{Proof of \Cref{prop:gluingfunc}: disconnected case}
\begin{figure}
\centering
        \begin{subfigure}{0.4\textwidth}
            \begin{tikzpicture}[scale=0.7]
        \filldraw[fill=lightgray, draw=black, thick] (-2,0) ellipse (1.7 and 2);
        \filldraw[fill=lightgray, draw=black, thick] (1.7,0) ellipse (1.7 and 2);
        \filldraw[lightgray] (-2,-2) rectangle (1.7,2);
        \draw[thick] (-2.01,-2) -- (1.71,-2);
        \draw[thick] (-2.01,2) -- (1.71,2);
        \filldraw[string-blue] (1,1.6) circle (0.5mm)
        (-1.5,1.6) circle (0.5mm);
        \draw[string-blue,thick] (1,1.6) -- (1,0.5);
        \draw[string-blue,thick,->] (1,0.5) -- (1,2.25);
        \draw[string-blue,thick] (-1.5,1.6) -- (-1.5,0.5);
        \draw[string-blue,thick,->] (-1.5,0.5) -- (-1.5,1.1);
        \begin{scope}[yshift=-0.1cm]
        \filldraw[string-blue] (1.5,-1.2) circle (0.5mm);
       \draw[string-blue,thick] (1.5,-1.2) -- (1.5,0);
       \draw[string-blue,thick,->] (1.5,-1.2) -- (1.5,-0.5);
       \draw[string-blue,thick] (-1.9,-0.5) -- (-1.9,0);
       \draw[string-red,thick] (-2.3,-0.5) -- (-1.5,-0.5);
       \draw[string-red,thick] (-2.2,-0.5) -- (-2.2,-0.7);
       \draw[string-red,thick] (-1.6,-0.5) -- (-1.6,-0.7);
       \draw[string-red,thick] (-2.5,-1) arc (180:360:0.6 and 0.4);
       \draw[string-red,thick,->] (-2.5,-1) arc (180:275:0.6 and 0.4);
       \draw[string-red,thick] (-2.2,-0.7) .. controls (-2.2,-0.8) and (-2.5,-0.9) .. (-2.5,-1);
       \draw[string-red,thick] (-1.6,-0.7) .. controls (-1.6,-0.8) and (-1.3,-0.9) .. (-1.3,-1);
       \draw[string-blue,thick] (0,-0.5) -- (0,0);
       \draw[string-red,thick] (-0.4,-0.5) -- (0.4,-0.5);
       \draw[string-red,thick] (-0.3,-0.5) -- (-0.3,-0.7);
       \draw[string-red,thick] (0.3,-0.5) -- (0.3,-0.7);
       \draw[string-red,thick] (-0.6,-1) arc (180:360:0.6 and 0.4);
       \draw[string-red,thick,->] (-0.6,-1) arc (180:275:0.6 and 0.4);
       \draw[string-red,thick] (-0.3,-0.7) .. controls (-0.3,-0.8) and (-0.6,-0.9) .. (-0.6,-1);
       \draw[string-red,thick] (0.3,-0.7) .. controls (0.3,-0.8) and (0.6,-0.9) .. (0.6,-1);
       \filldraw[fill=white, draw=string-blue, thick] (-2.4,0) rectangle (1.9,0.6);
       \filldraw[fill=white, draw=black, thick] (-1.5,-1) arc (0:-180:0.4 and 0.2);
       \fill[lightgray] (-2.3,-0.9) rectangle (-1.5,-1.09);
       \filldraw[fill=white, draw=black, thick] (-1.55,-1.1) arc (0:180:0.35 and 0.15);
       \draw[black, thick] (-1.5,-1) arc (0:-180:0.4 and 0.2);
       \filldraw[fill=white, draw=black, thick] (0.4,-1) arc (0:-180:0.4 and 0.2);
       \fill[lightgray] (-0.4,-0.9) rectangle (0.4,-1.09);
       \filldraw[fill=white, draw=black, thick] (0.35,-1.1) arc (0:180:0.35 and 0.15);
       \draw[black, thick] (0.4,-1) arc (0:-180:0.4 and 0.2);
        \end{scope}
    \begin{scope}[yshift=5cm]
        \filldraw[fill=lightgray, draw=black, thick] (-2,0) ellipse (1.7 and 2);
        \filldraw[fill=lightgray, draw=black, thick] (1.7,0) ellipse (1.7 and 2);
        \filldraw[lightgray] (-2,-2) rectangle (1.7,2);
        \draw[thick] (-2.01,-2) -- (1.71,-2);
        \draw[thick] (-2.01,2) -- (1.71,2);
        \filldraw[string-blue] (1,1.6) circle (0.5mm)
        (-0.25,1.6) circle (0.5mm)
        (-1.5,1.6) circle (0.5mm);
        \draw[string-blue,thick] (1,1.6) -- (1,0.5);
        \draw[string-blue,thick,->] (1,0.5) -- (1,1.1);
        \draw[string-blue,thick] (-1.5,1.6) -- (-1.5,0.5);
        \draw[string-blue,thick,->] (-1.5,0.5) -- (-1.5,1.1);
        \draw[string-blue,thick] (-0.25,1.6) -- (-0.25,0.5);
        \draw[string-blue,thick,->] (-0.25,0.5) -- (-0.25,1.1);
        \begin{scope}[yshift=-0.1cm]
        \filldraw[string-blue] (1,-1.2) circle (0.5mm);
        \draw[string-blue,thick] (1,-1.2) -- (1,0);
        \draw[string-blue,thick] (-1.9,-0.5) -- (-1.9,0);
        \draw[string-red,thick] (-2.3,-0.5) -- (-1.5,-0.5);
        \draw[string-red,thick] (-2.2,-0.5) -- (-2.2,-0.7);
        \draw[string-red,thick] (-1.6,-0.5) -- (-1.6,-0.7);
        \draw[string-red,thick] (-2.5,-1) arc (180:360:0.6 and 0.4);
        \draw[string-red,thick,->] (-2.5,-1) arc (180:275:0.6 and 0.4);
        \draw[string-red,thick] (-2.2,-0.7) .. controls (-2.2,-0.8) and (-2.5,-0.9) .. (-2.5,-1);
        \draw[string-red,thick] (-1.6,-0.7) .. controls (-1.6,-0.8) and (-1.3,-0.9) .. (-1.3,-1);
        \filldraw[fill=white, draw=string-blue, thick] (-2.4,0) rectangle (1.9,0.6);
        \filldraw[fill=white, draw=black, thick] (-1.5,-1) arc (0:-180:0.4 and 0.2);
        \fill[lightgray] (-2.3,-0.9) rectangle (-1.5,-1.09);
        \filldraw[fill=white, draw=black, thick] (-1.55,-1.1) arc (0:180:0.35 and 0.15);
        \draw[black, thick] (-1.5,-1) arc (0:-180:0.4 and 0.2);
    \end{scope}        
    \end{scope}
    \fill[lightgray] (0.7,1.5) -- (0.7,3.9) -- (1.3,3.9) --
    (1.3,1.5) -- cycle;
    \draw[thick] (0.7,1.7) -- (0.7,3.7)
    (1.3,3.7) -- (1.3,1.7)
    (0.4,1.3) .. controls (0.7,1.3) and (0.7,1.7) .. (0.7,1.7)
    (1.6,1.3) .. controls (1.3,1.3) and (1.3,1.7) .. (1.3,1.7)
    (0.4,4.1) .. controls (0.7,4.1) and (0.7,3.7) .. (0.7,3.7)
    (1.6,4.1) .. controls (1.3,4.1) and (1.3,3.7) .. (1.3,3.7); 
    \begin{scope}[very thick,decoration={
                      markings,
                      mark=at position 0.52 with {\arrow{>}}}
                     ] 
        \draw[string-blue,thick,postaction={decorate}] (1,1.3) -- (1,3.9);
        \end{scope}
    \node at (1.3,1) [string-blue] {$X$};
   \end{tikzpicture} 
            \caption{Action of $M_X$ on handle bodies bounding $\Sigma_2^{(1,2)}$ and $\Sigma_1^{(1,3)}$, respectively.}\label{fig:handlebodyglue1a}
        \end{subfigure}
        \hspace{0.5cm}
    \begin{subfigure}{0.4\textwidth}
       \begin{tikzpicture}[scale=0.7]
        \begin{scope}[yscale=1.9,xscale=1.4]
            \filldraw[fill=lightgray, draw=black, thick] (-2,0) ellipse (2 and 2);
            \filldraw[fill=lightgray, draw=black, thick] (1.3,0) ellipse (2 and 2);
            \filldraw[lightgray] (-2,-2) rectangle (1.3,2);
            \draw[thick] (-2.01,-2) -- (1.31,-2);
            \draw[thick] (-2.01,2) -- (1.31,2);
        \end{scope}
        \begin{scope}[yshift=-1.8cm,xshift=0.25cm]
            \draw[string-blue,thick] (0.75,3) -- (0.75,0.5);
            \draw[string-blue,thick,->] (0.75,0.5) -- (0.75,1.1);
            \draw[string-blue,thick] (-1.5,1.2) -- (-1.5,0.5);
            \draw[string-blue,thick,->] (-1.5,0.5) -- (-1.5,1.1);
        \begin{scope}
            \filldraw[string-blue] (1.5,-1.2) circle (0.5mm);
           \draw[string-blue,thick] (1.5,-1.2) -- (1.5,0);
           \draw[string-blue,thick,->] (1.5,-1.2) -- (1.5,-0.5);
           \draw[string-blue,thick] (-1.9,-0.5) -- (-1.9,0);
           \draw[string-red,thick] (-2.3,-0.5) -- (-1.5,-0.5);
           \draw[string-red,thick] (-2.2,-0.5) -- (-2.2,-0.7);
           \draw[string-red,thick] (-1.6,-0.5) -- (-1.6,-0.7);
           \draw[string-red,thick] (-2.5,-1) arc (180:360:0.6 and 0.4);
           \draw[string-red,thick,->] (-2.5,-1) arc (180:275:0.6 and 0.4);
           \draw[string-red,thick] (-2.2,-0.7) .. controls (-2.2,-0.8) and (-2.5,-0.9) .. (-2.5,-1);
           \draw[string-red,thick] (-1.6,-0.7) .. controls (-1.6,-0.8) and (-1.3,-0.9) .. (-1.3,-1);
           \draw[string-blue,thick] (0,-0.5) -- (0,0);
           \draw[string-red,thick] (-0.4,-0.5) -- (0.4,-0.5);
           \draw[string-red,thick] (-0.3,-0.5) -- (-0.3,-0.7);
           \draw[string-red,thick] (0.3,-0.5) -- (0.3,-0.7);
           \draw[string-red,thick] (-0.6,-1) arc (180:360:0.6 and 0.4);
           \draw[string-red,thick,->] (-0.6,-1) arc (180:275:0.6 and 0.4);
           \draw[string-red,thick] (-0.3,-0.7) .. controls (-0.3,-0.8) and (-0.6,-0.9) .. (-0.6,-1);
           \draw[string-red,thick] (0.3,-0.7) .. controls (0.3,-0.8) and (0.6,-0.9) .. (0.6,-1);
           \filldraw[fill=white, draw=string-blue, thick] (-2.4,0) rectangle (1.9,0.6);
           \filldraw[fill=white, draw=black, thick] (-1.5,-1) arc (0:-180:0.4 and 0.2);
           \fill[lightgray] (-2.3,-0.9) rectangle (-1.5,-1.09);
           \filldraw[fill=white, draw=black, thick] (-1.55,-1.1) arc (0:180:0.35 and 0.15);
           \draw[black, thick] (-1.5,-1) arc (0:-180:0.4 and 0.2);
           \filldraw[fill=white, draw=black, thick] (0.4,-1) arc (0:-180:0.4 and 0.2);
           \fill[lightgray] (-0.4,-0.9) rectangle (0.4,-1.09);
           \filldraw[fill=white, draw=black, thick] (0.35,-1.1) arc (0:180:0.35 and 0.15);
           \draw[black, thick] (0.4,-1) arc (0:-180:0.4 and 0.2);
        \end{scope}
        \end{scope}
        \begin{scope}[yshift=1.8cm,xshift=-0.25cm]
            \filldraw[string-blue] (1,1.6) circle (0.5mm)
            (-0.25,1.6) circle (0.5mm)
            (-1.5,1.6) circle (0.5mm);
            \draw[string-blue,thick] (1,1.6) -- (1,0.5);
            \draw[string-blue,thick,->] (1,0.5) -- (1,1.1);
            \draw[string-blue,thick] (-1.5,1.6) -- (-1.5,0.5);
            \draw[string-blue,thick,->] (-1.5,0.5) -- (-1.5,1.1);
            \draw[string-blue,thick] (-0.25,1.6) -- (-0.25,0.5);
            \draw[string-blue,thick,->] (-0.25,0.5) -- (-0.25,1.1);
            \begin{scope}[yshift=-0.1cm]
                \draw[string-blue,thick] (1.25,-1.2) -- (1.25,0);
                \draw[string-blue,thick,->] (1.25,-1.2) -- (1.25,-0.5);
                \filldraw[fill=white, draw=string-blue, thick] (-2.4,0) rectangle (1.9,0.6);
                \draw[string-blue,thick] (-1.9,-0.5) -- (-1.9,0);
        \end{scope}        
    \end{scope}
        \begin{scope}[yshift=-1.8cm,xshift=-1.5cm]
                \draw[string-blue,thick] (-1.9,-0.5) -- (-1.9,0);
                \draw[string-red,thick] (-2.3,-0.5) -- (-1.5,-0.5);
                \draw[string-red,thick] (-2.2,-0.5) -- (-2.2,-0.7);
                \draw[string-red,thick] (-1.6,-0.5) -- (-1.6,-0.7);
                \draw[string-red,thick] (-2.5,-1) arc (180:360:0.6 and 0.4);
                \draw[string-red,thick,->] (-2.5,-1) arc (180:275:0.6 and 0.4);
                \draw[string-red,thick] (-2.2,-0.7) .. controls (-2.2,-0.8) and (-2.5,-0.9) .. (-2.5,-1);
                \draw[string-red,thick] (-1.6,-0.7) .. controls (-1.6,-0.8) and (-1.3,-0.9) .. (-1.3,-1);
                \filldraw[fill=white, draw=black, thick] (-1.5,-1) arc (0:-180:0.4 and 0.2);
                \fill[lightgray] (-2.3,-0.9) rectangle (-1.5,-1.09);
                \filldraw[fill=white, draw=black, thick] (-1.55,-1.1) arc (0:180:0.35 and 0.15);
                \draw[black, thick] (-1.5,-1) arc (0:-180:0.4 and 0.2);
        \end{scope}
        \filldraw[string-blue] (2,3.4) circle (0.5mm);
        \draw[string-blue,thick] (2,3.4) -- (2,1.9);
        \draw[string-blue,thick,->] (2,2.4) -- (2,2.9);
        \begin{scope}[very thick,decoration={
                      markings,
                      mark=at position 0.52 with {\arrow{>}}}
                     ] 
        \draw[string-blue,thick,postaction={decorate}] (-1.25,-0.6) .. controls (-1.25,-0.1) and (2,0.2) .. (2,1.9);
        \draw[string-blue,thick,postaction={decorate}]
        (-3.4,-1.8) .. controls (-3.35,-1.1) and (-2.15,0.5) .. (-2.15,1.2);
        \end{scope}
        \fill[lightgray] (0.85,0) rectangle (1.15,1);
        \draw[string-blue,thick] (1,0) -- (1,1);
        \node at (1.3,-0.4) [string-blue] {$X$};
   \end{tikzpicture} 
       \vspace{0.2cm}
        \caption{Isotoped handle body with embedded ribbon graph as in \Cref{cor:gluingalg1}.}\label{fig:handlebodyglue1b}
    \end{subfigure}
    \caption{Illustration of the handle body obtained after applying $M_X$ for gluing $\Sigma_2^{(1,2)}$ and $\Sigma_1^{(1,3)}$ to $\Sigma_3^{(1,4)}$.}\label{fig:handlebodyglue1}
\end{figure}
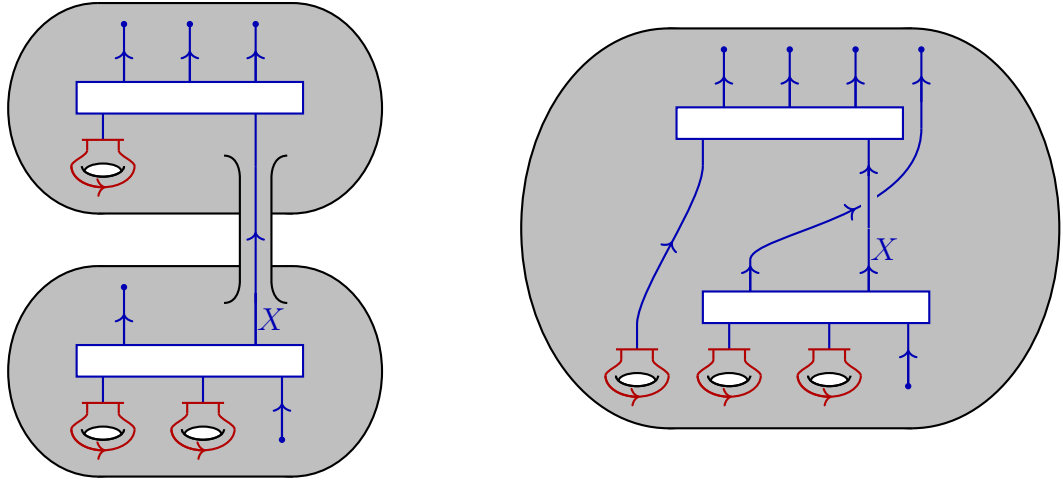

For the first case we can assume without loss of generality that $\Sigma$ has only two connected components, i.e.\ $\Sigma \cong \Sigma_{g_1}^{({p_1},{q_1})} \sqcup \Sigma_{g_2}^{({p_2},{q_2})}$ where $\Sigma_{g}^{({p},{q})}$ denotes a connected surfaces of genus $g$ with $p$ incoming and $q$ outgoing boundary components. After gluing along one boundary component we obtain a connected surface of the form $\Sigma_{\mathrm{gl}} \cong \Sigma_{g_1+g_2}^{({p_1+p_2-1},{q_1+q_2-1})}$. Using the identification of the state space with the morphism space in $\calc$ from \Cref{sec:algstatespace} it can easily be verified that the composition of $M_X$ with the standard handle bodies is given by a handle body with embedded ribbon graph induced from composition and the braiding. This is illustrated in Figure \ref{fig:handlebodyglue1} for the case of gluing $\Sigma_2^{(1,2)}$ and $\Sigma_1^{(1,3)}$ to $\Sigma_3^{(1,4)}$.

Using the isomorphisms identifying state spaces of the TFT with Hom spaces of $\calc$ it immediately follows that the dinatural family $\eta_X$ corresponds to the universal family from \Cref{cor:gluingalg1}.

\subsubsection*{Proof of \Cref{prop:gluingfunc}: connected case}
For the second case let us assume for simplicity that $\Sigma$ is a genus zero surface with two incoming and two outgoing boundary components, i.e.\ $\Sigma \cong \Sigma_{0}^{2,2}$, the general case works completely analogous. After gluing two of the boundary components we obtain a surface of the form $\Sigma_{\mathrm{gl}} \cong \Sigma_{1}^{(1,1)}$. Following the same line of arguments as above we obtain a handle body with embedded ribbon graph as depicted in \Cref{fig:handlebodyglue2a}.

\begin{figure}[t]
\centering
        \begin{subfigure}{0.4\textwidth}
            \begin{tikzpicture}
        \filldraw[fill=lightgray, draw=black, thick] (-1.7,0) ellipse (1.7 and 2);
        \filldraw[fill=lightgray, draw=black, thick] (0.3,0) ellipse (1.7 and 2);
        \filldraw[lightgray] (-1.7,-2) rectangle (0.3,2);
        \draw[thick] (-1.71,-2) -- (0.31,-2);
        \draw[thick] (-1.71,2) -- (0.31,2);
        \filldraw[fill=white,draw=string-blue,thick] (-0.8,-0.2) rectangle (0.9,0.4);
        \filldraw[string-blue] (0.5,1.4) circle (0.5mm)
        (0.5,-1.2) circle (0.5mm);
        \draw[string-blue,thick] (-0.4,-0.2) -- (-0.4,-1)
        (-0.4,0.4) -- (-0.4,1)
        (-2.2,-1) -- (-2.2,1);
        \begin{scope}[very thick,decoration={
                      markings,
                      mark=at position 0.5 with {\arrow{>}}}
                     ] 
        \draw[string-blue,thick,postaction={decorate}] (0.5,-1.2) -- (0.5,-0.2);
        \draw[string-blue,thick,postaction={decorate}] (0.5,0.4) -- (0.5,1.4);
        \draw[string-blue,thick,postaction={decorate}] (-0.4,1) arc (0:180: 0.9 and 0.6);
        \draw[string-blue,thick,postaction={decorate}] (-2.2,-1) arc (180:360: 0.9 and 0.6);
        \end{scope}
        \begin{scope}[scale=1.5,xshift=1cm,yshift=0.3cm]
            \filldraw[fill=white, draw=black, thick] (-1.5,-1) arc (0:-180:0.4 and 0.2);
       \fill[lightgray] (-2.3,-0.9) rectangle (-1.5,-1.09);
       \filldraw[fill=white, draw=black, thick] (-1.55,-1.1) arc (0:180:0.35 and 0.15);
       \draw[black, thick] (-1.5,-1) arc (0:-180:0.4 and 0.2);
        \end{scope}
       \node at (-0.4,-0.5) [string-blue,right] {$X$};
   \end{tikzpicture}
            \caption{Action of $M_X$ on handle body bounding $\Sigma_0^{(2,2)}$.}\label{fig:handlebodyglue2a}
        \end{subfigure}
        \hspace{1cm}
    \begin{subfigure}{0.4\textwidth}
       \begin{tikzpicture}
        \filldraw[fill=lightgray, draw=black, thick] (-1.7,0) ellipse (1.7 and 2);
        \filldraw[fill=lightgray, draw=black, thick] (0.3,0) ellipse (1.7 and 2);
        \filldraw[lightgray] (-1.7,-2) rectangle (0.3,2);
        \draw[thick] (-1.71,-2) -- (0.31,-2);
        \draw[thick] (-1.71,2) -- (0.31,2);
        \filldraw[fill=white,draw=string-blue,thick] (-0.8,0.1) rectangle (0.9,0.7);
        \draw[string-red,thick] (-1.35,-0.73) arc (-90:-270: 0.2 and 0.4);
        \filldraw[string-blue] (0.5,1.4) circle (0.5mm)
        (0.5,-1.2) circle (0.5mm);
        \draw[string-blue,thick](-0.4,0.7) -- (-0.4,0.8)
        (-0.4,0.05) -- (-0.4,0.1)
        (-1.6,0.05) -- (-1.6,0.8);
        \begin{scope}[decoration={
                      markings,
                      mark=at position 0.55 with {\arrow{>}}}
                     ] 
            \fill[lightgray] (-1.6,-0.2) rectangle (-1.45,-0.05)
            (-1.6,-0.58) rectangle (-1.45,-0.45);
            \draw[string-blue,thick,postaction={decorate}] (0.5,-1.2) -- (0.5,0.1);
            \draw[string-blue,thick,postaction={decorate}] (0.5,0.7) -- (0.5,1.4);
            \draw[string-blue,thick,postaction={decorate}] (-0.4,0.8) arc (0:180: 0.6 and 0.3);
            \draw[string-blue,thick,postaction={decorate}] (-1.6,0.05) arc (-180:0: 0.6 and 0.3);
            \draw[string-red,thick,postaction={decorate}] (-0.45,-1.1) arc (0:180: 0.9 and 0.6);
            \draw[string-red,thick,postaction={decorate}] (-2.25,-1.1) arc (-180:0: 0.9 and 0.6);
            \fill[lightgray] (-1.1,-0.3) rectangle (-1.25,-0.17)
            (-1.1,-0.58) rectangle (-1.25,-0.45);
            \draw[string-red,thick,postaction={decorate}] (-1.35,0.07) arc (90:-90: 0.2 and 0.4);
        \end{scope}
        \begin{scope}[scale=1.5,xshift=1cm,yshift=0.3cm]
            \filldraw[fill=white, draw=black, thick] (-1.5,-1) arc (0:-180:0.4 and 0.2);
            \fill[lightgray] (-2.3,-0.9) rectangle (-1.5,-1.09);
            \filldraw[fill=white, draw=black, thick] (-1.55,-1.1) arc (0:180:0.35 and 0.15);
            \draw[black, thick] (-1.5,-1) arc (0:-180:0.4 and 0.2);
        \end{scope}
       \node at (-0.3,-0.2) [string-blue] {$X$};
   \end{tikzpicture} 
        \caption{Handle body after application of the ``slide trick".}\label{fig:handlebodyglue2b}
    \end{subfigure}
    \caption{Handle body obtained after applying $M_X$ to a handle body bounding $\Sigma_0^{(2,2)}$.}\label{fig:handlebodyglue2}
\end{figure}
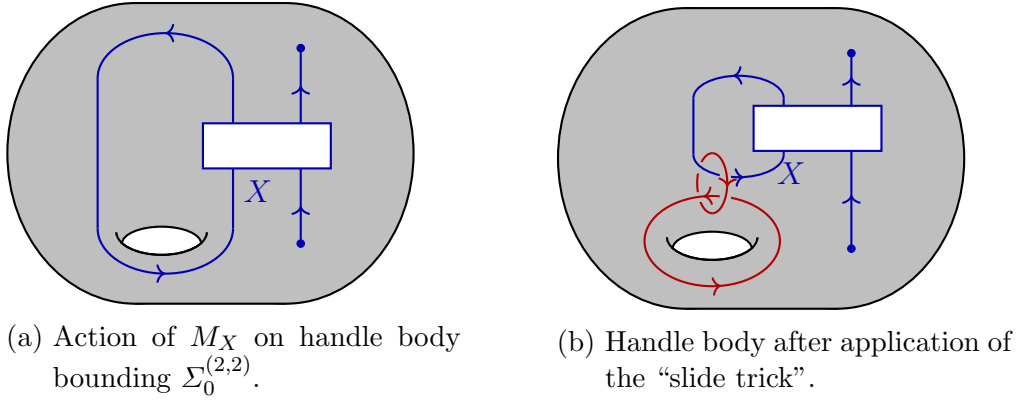

To translate the action of $\eta_X$ to Hom-spaces we need to relate the graph in  \Cref{fig:handlebodyglue2a}
to a standard handle body with bichrome graph as in \Cref{fig:standard-handle-body}. To this end, we first apply the TFT $\widehat{\rmV}_\calc$ and then use the "slide trick" employed in \autocite[Lem.\,4.14]{DGGPR19} to modify the bichrome graph. The resulting handle body with bichrome graph is depicted in \Cref{fig:handlebodyglue2b}. We have 
{
\allowdisplaybreaks
\begin{equation}
    \widehat{\rmV}_\calc \big( \text{Fig.\,\ref{fig:handlebodyglue2}a}
    \big) 
    =
    \zeta^{-1}\, \widehat{\rmV}_\calc 
    \big( \text{Fig.\,\ref{fig:handlebodyglue2}b}
    \big) 
    =
    \widehat{\rmV}_\calc \left( 
    \begin{tikzpicture}[scale=0.85,baseline=0cm]
        \filldraw[fill=lightgray, draw=black, thick] (-1.7,0) ellipse (1.7 and 2);
        \filldraw[fill=lightgray, draw=black, thick] (0.3,0) ellipse (1.7 and 2);
        \filldraw[lightgray] (-1.7,-2) rectangle (0.3,2);
        \draw[thick] (-1.71,-2) -- (0.31,-2);
        \draw[thick] (-1.71,2) -- (0.31,2);
        \filldraw[fill=white,draw=string-blue,thick] (-0.8,-0.25) rectangle (0.9,0.45)
        (-2.2,1) rectangle (-0.4,1.6);
        \filldraw[string-blue] (0.5,1.4) circle (0.5mm)
        (0.5,-1.2) circle (0.5mm);
        \draw[string-blue,thick](-0.4,0.45) -- (-0.4,0.7)
        (-0.4,-0.3) -- (-0.4,-0.25)
        (-1.2,-0.3) -- (-1.2,0.7)
        (-0.3,0.7) -- (-1.3,0.7)
        (-0.8,0.7) -- (-0.8,1);
        \begin{scope}[decoration={
                      markings,
                      mark=at position 0.52 with {\arrow{>}}}
                     ] 
            \draw[string-blue,thick,postaction={decorate}] (0.5,-1.2) -- (0.5,-0.25);
            \draw[string-blue,thick,postaction={decorate}] (0.5,0.45) -- (0.5,1.4);
            \draw[string-blue,thick,postaction={decorate}] (-1.2,-0.3) arc (-180:0: 0.4 and 0.3);
            \draw[string-red,thick,postaction={decorate}] (-2.7,-1) arc (-180:0: 0.9 and 0.6);
        \end{scope}
        \begin{scope}[scale=1.5,xshift=0.7cm,yshift=0.3cm]
            \filldraw[fill=white, draw=black, thick] (-1.5,-1) arc (0:-180:0.4 and 0.2);
            \fill[lightgray] (-2.3,-0.9) rectangle (-1.5,-1.09);
            \filldraw[fill=white, draw=black, thick] (-1.55,-1.1) arc (0:180:0.35 and 0.15);
            \draw[black, thick] (-1.5,-1) arc (0:-180:0.4 and 0.2);
        \end{scope}
        \draw[string-blue,thick] (-1.8,0.1) -- (-1.8,1);
        \draw[string-red,thick] (-2.2,0.1) -- (-1.4,0.1);
        \draw[string-red,thick] (-2.1,-0.3) -- (-2.1,0.1);
        \draw[string-red,thick] (-1.5,-0.3) -- (-1.5,0.1);
        \draw[string-red,thick] (-2.1,-0.3) .. controls (-2.1,-0.5) and (-2.7,-0.8) .. (-2.7,-1);
        \draw[string-red,thick] (-1.5,-0.3) .. controls (-1.5,-0.5) and (-0.9,-0.8) .. (-0.9,-1);
        \node at (-0.25,-0.5) [string-blue] {$X$};
        \node at (-0.05,0.65) [string-blue] {$\iota_X$};
        \node at (-1.3,1.3) [string-blue] {$\kappa$};
   \end{tikzpicture} 
    \right).
\end{equation}}
The second equality follows from a computation in bichrome graphs:
\begin{equation}
    \begin{aligned}
        \begin{tikzpicture}[scale=0.85,baseline=0cm]
        \filldraw[fill=white,draw=string-blue,thick] (-0.8,0.1) rectangle (0.9,0.7);
        \draw[string-red,thick] (-1.35,-0.73) arc (-90:-270: 0.2 and 0.4);
        \draw[string-blue,thick](-0.4,0.7) -- (-0.4,0.8)
        (-0.4,0.05) -- (-0.4,0.1)
        (-1.6,0.05) -- (-1.6,0.8);
        \begin{scope}[decoration={
                      markings,
                      mark=at position 0.58 with {\arrow{>}}}
                     ] 
            \fill[white] (-1.6,-0.2) rectangle (-1.45,-0.05)
            (-1.6,-0.58) rectangle (-1.45,-0.45);
            \draw[string-blue,thick,postaction={decorate}] (0.5,-1.2) -- (0.5,0.1);
            \draw[string-blue,thick,postaction={decorate}] (0.5,0.7) -- (0.5,1.4);
            \draw[string-blue,thick,postaction={decorate}] (-0.4,0.8) arc (0:180: 0.6 and 0.3);
            \draw[string-blue,thick,postaction={decorate}] (-1.6,0.05) arc (-180:0: 0.6 and 0.3);
            \draw[string-red,thick,postaction={decorate}] (-0.85,-1) arc (0:180: 0.5 and 0.5);
            \fill[white] (-1.1,-0.3) rectangle (-1.25,-0.17)
            (-1.1,-0.58) rectangle (-1.25,-0.45);
            \draw[string-red,thick,postaction={decorate}] (-1.35,0.07) arc (90:-90: 0.2 and 0.4);
            \draw[string-red,thick] (-0.85,-1) -- (-0.85,-1.2)
            (-1.85,-1) -- (-1.85,-1.2);
        \end{scope}
       \node at (-0.3,-0.2) [string-blue] {\scriptsize $X$};
   \end{tikzpicture} 
   &\overset{(1)}=
    \begin{tikzpicture}[scale=0.85,baseline=0cm]
        \filldraw[fill=white,draw=string-blue,thick] (-1,0.1) rectangle (0.7,0.7);
        \draw[string-blue,thick](-0.6,0.7) -- (-0.6,0.8)
        (-0.6,0.05) -- (-0.6,0.1)
        (-1.4,0.05) -- (-1.4,0.8);
        \begin{scope}[decoration={
                      markings,
                      mark=at position 0.52 with {\arrow{>}}}
                     ] 
            \draw[string-red,thick,postaction={decorate}] (-1.2,0.7) arc (0:180: 0.35 and 0.35);  
            \draw[string-red,thick,postaction={decorate}] (-2.5,0.7) arc (0:180: 0.35 and 0.35);  
            \draw[string-red,thick,postaction={decorate}] (-1.9,0.7) arc (0:-180: 0.3 and 0.35);  
            \fill[white] (-1.37,1.1) rectangle (-1.28,0.9)
            (-3.13,1.1) rectangle (-3.0,0.85);
            \draw[string-blue,thick,postaction={decorate}] (0.3,-1.2) -- (0.3,0.1);
            \draw[string-blue,thick,postaction={decorate}] (0.3,0.7) -- (0.3,1.4);
            \draw[string-blue,thick,postaction={decorate}] (-0.6,0.8) arc (0:180: 0.4 and 0.3);
            \draw[string-blue,thick,postaction={decorate}] (-1.4,0.05) arc (-180:0: 0.4 and 0.3);
            \draw[string-red,thick,postaction={decorate}] (-3,0.8) arc (0:180: 0.4 and 0.3);
            \draw[string-red,thick] (-3,0.8) -- (-3,-1.2)
            (-3.8,0.8) -- (-3.8,-1.2)
            (-1.2,0.5) -- (-1.2,0.7)
            (-3.2,0.5) -- (-3.2,0.7);
            \fill[white] (-3.1,0.3) rectangle (-2.9,0.15)
            (-1.5,0.3) rectangle (-1.3,0.15);
            \draw[string-red,thick,postaction={decorate}] (-2.55,-0.2) arc (180:360: 0.35 and 0.35);
            \draw[string-red,thick] (-3.2,0.5) .. controls (-3.2,0.2) and (-2.55,0.1) .. (-2.55,-0.2)
            (-1.2,0.5) .. controls (-1.2,0.2) and (-1.85,0.1) .. (-1.85,-0.2);
        \end{scope}
       \node at (-0.5,-0.2) [string-blue] {\scriptsize $X$};
   \end{tikzpicture} 
    \overset{(2)}=
    \begin{tikzpicture}[scale=0.85,baseline=0cm]
        \filldraw[fill=white,draw=string-blue,thick] (-1,-0.1) rectangle (0.7,0.5)
        (-1.8,1.3) rectangle (-0.8,1.7)
        (-3.4,1.3) rectangle (-2.4,1.7);
        \draw[string-blue,thick](-0.6,0.5) -- (-0.6,0.8)
        (-0.6,-0.25) -- (-0.6,-0.1)
        (-1.4,-0.25) -- (-1.4,0.8)
        (-1,1.3) -- (-1,0.8)
        (-1.6,1.3) -- (-1.6,1)
        (-3.2,1.3) -- (-3.2,0.8)
        (-2.6,1.3) -- (-2.6,1);
        \draw[string-red,thick] (-3.7,0.8) -- (-2.7,0.8);
        \begin{scope}[decoration={
                      markings,
                      mark=at position 0.52 with {\arrow{>}}}
                     ] 
            \draw[string-blue,thick,postaction={decorate}] (0.3,-0.8) -- (0.3,-0.1);
            \draw[string-blue,thick,postaction={decorate}] (0.3,0.5) -- (0.3,1.8);
            \draw[string-blue,thick] (-1.5,0.8) -- (-0.5,0.8);
            \draw[string-blue,thick,postaction={decorate}] (-1.4,-0.25) arc (-180:0: 0.4 and 0.3);
            \draw[string-red,thick,postaction={decorate}] (-2.8,-0.8) -- (-2.8,0.8);
            \draw[string-red,thick,postaction={decorate}] (-3.6,0.8) -- (-3.6,-0.8);        
            \draw[string-blue,thick] (-1.6,1) arc (0:-180:0.5 and 0.5)
            (-2.1,0.5) -- (-2.1,0);
            \filldraw[string-blue,thick] (-2.1,0) circle (0.05cm);
        \end{scope}
       \node at (-0.4,-0.4) [string-blue] {\scriptsize $X$};
       \node at (-1.95,-0.1) [string-blue] {\scriptsize $\Lambda$};
       \node at (-0.3,0.7) [string-blue] {\scriptsize $\iota_X$};
       \node at (-1.3,1.5) [string-blue] {\scriptsize $\omega$};
       \node at (-2.9,1.53) [string-blue] {\scriptsize $\overline{\omega}$};
   \end{tikzpicture} 
   \\
   &\overset{(3)}=
    \begin{tikzpicture}[scale=0.85,baseline=0.3cm]
        \filldraw[fill=white,draw=string-blue,thick] (-1,-0.1) rectangle (0.7,0.5)
        (-1.8,1.3) rectangle (-0.8,1.7)
        (-3.4,1.3) rectangle (-2.4,1.7);
        \draw[string-blue,thick](-0.6,0.5) -- (-0.6,0.8)
        (-0.6,-0.25) -- (-0.6,-0.1)
        (-1.4,-0.25) -- (-1.4,0.8)
        (-1,1.3) -- (-1,0.8)
        (-1.6,1.3) -- (-1.6,1)
        (-3.2,1.3) -- (-3.2,0.3)
        (-2.6,1.3) -- (-2.6,1);
        \begin{scope}[decoration={
                      markings,
                      mark=at position 0.52 with {\arrow{>}}}
                     ] 
            \draw[string-blue,thick,postaction={decorate}] (0.3,-0.8) -- (0.3,-0.1);
            \draw[string-blue,thick,postaction={decorate}] (0.3,0.5) -- (0.3,1.8);
            \draw[string-blue,thick] (-1.5,0.8) -- (-0.5,0.8);
            \draw[string-blue,thick,postaction={decorate}] (-1.4,-0.25) arc (-180:0: 0.4 and 0.3);
            \draw[string-red,thick,postaction={decorate}] (-2.8,-0.8) -- (-2.8,0.3);
            \draw[string-red,thick,postaction={decorate}] (-3.6,0.3) -- (-3.6,-0.8);        
            \draw[string-red,thick] (-3.7,0.3) -- (-2.7,0.3);
            \draw[string-blue,thick] (-1.6,1) arc (0:-180:0.5 and 0.5)
            (-2.1,0.5) -- (-2.1,0);
            \filldraw[string-blue,thick] (-2.1,0) circle (0.05cm);
            \draw[string-blue,thick] (-3.2,0.8) circle (0.13cm)
            (-3.33,0.8) -- (-3.07,0.8);
        \end{scope}
       \node at (-0.4,-0.4) [string-blue] {\scriptsize $X$};
       \node at (-1.95,-0.1) [string-blue] {\scriptsize $\Lambda$};
       \node at (-0.3,0.7) [string-blue] {\scriptsize $\iota_X$};
       \node at (-1.3,1.5) [string-blue] {\scriptsize $\omega$};
       \node at (-2.9,1.5) [string-blue] {\scriptsize $\omega$};
   \end{tikzpicture}
   \overset{(4)}=
    \begin{tikzpicture}[scale=0.85,baseline=0.3cm]
        \filldraw[fill=white,draw=string-blue,thick] (-1,-0.1) rectangle (0.7,0.5)
        (-2.2,1.3) rectangle (-0.8,1.7);
        \draw[string-blue,thick](-0.6,0.5) -- (-0.6,0.8)
        (-0.6,-0.25) -- (-0.6,-0.1)
        (-1.4,-0.25) -- (-1.4,0.8)
        (-1,1.3) -- (-1,0.8)
        (-2,1.3) -- (-2,0.3);
        \begin{scope}[decoration={
                      markings,
                      mark=at position 0.52 with {\arrow{>}}}
                     ] 
            \draw[string-blue,thick,postaction={decorate}] (0.3,-0.8) -- (0.3,-0.1);
            \draw[string-blue,thick,postaction={decorate}] (0.3,0.5) -- (0.3,1.8);
            \draw[string-blue,thick] (-1.5,0.8) -- (-0.5,0.8);
            \draw[string-blue,thick,postaction={decorate}] (-1.4,-0.25) arc (-180:0: 0.4 and 0.3);
            \draw[string-red,thick,postaction={decorate}] (-1.7,-0.8) -- (-1.7,0.3);
            \draw[string-red,thick,postaction={decorate}] (-2.3,0.3) -- (-2.3,-0.8);        
            \draw[string-red,thick] (-2.4,0.3) -- (-1.6,0.3);
            \draw[string-blue,thick] (-2,0.55) circle (0.13cm)
            (-2.13,0.55) -- (-1.87,0.55);
        \end{scope}
        \filldraw[fill=white,draw=string-blue,thick] (-2.15,0.8) rectangle (-1.85,1.15);
        \node at (-2,1) [string-blue] {\scriptsize $\calS$};
       \node at (-0.4,-0.4) [string-blue] {\scriptsize $X$};
       \node at (-0.3,0.7) [string-blue] {\scriptsize $\iota_X$};
       \node at (-1.5,1.5) [string-blue] {\scriptsize $\omega$};
   \end{tikzpicture} 
   \\
   &\overset{(5)}=\zeta \,
    \begin{tikzpicture}[scale=0.85,baseline=0.3cm]
        \filldraw[fill=white,draw=string-blue,thick] (-1,-0.1) rectangle (0.7,0.5)
        (-2.2,1.3) rectangle (-0.8,1.7);
        \draw[string-blue,thick](-0.6,0.5) -- (-0.6,0.8)
        (-0.6,-0.25) -- (-0.6,-0.1)
        (-1.4,-0.25) -- (-1.4,0.8)
        (-1,1.3) -- (-1,0.8)
        (-2,1.3) -- (-2,0.3);
        \begin{scope}[decoration={
                      markings,
                      mark=at position 0.52 with {\arrow{>}}}
                     ] 
            \draw[string-blue,thick,postaction={decorate}] (0.3,-0.8) -- (0.3,-0.1);
            \draw[string-blue,thick,postaction={decorate}] (0.3,0.5) -- (0.3,1.8);
            \draw[string-blue,thick] (-1.5,0.8) -- (-0.5,0.8);
            \draw[string-blue,thick,postaction={decorate}] (-1.4,-0.25) arc (-180:0: 0.4 and 0.3);
            \draw[string-red,thick,postaction={decorate}] (-1.7,-0.8) -- (-1.7,0.3);
            \draw[string-red,thick,postaction={decorate}] (-2.3,0.3) -- (-2.3,-0.8);        
            \draw[string-red,thick] (-2.4,0.3) -- (-1.6,0.3);
        \end{scope}
       \node at (-0.4,-0.4) [string-blue] {\scriptsize $X$};
       \node at (-0.3,0.7) [string-blue] {\scriptsize $\iota_X$};
       \node at (-1.5,1.5) [string-blue] {\scriptsize $\kappa$};
   \end{tikzpicture} 
    \end{aligned}
\end{equation}

Step 1 is isotopy invariance; step 2 amounts to the defining identities for the structure morphisms of the coend $\coend$ from \Cref{ssec:ribboncats} together with expressing the integral $\Lambda$ as a red cup; step 3 is \eqref{eq:hopfpairingmirror}; in step 4 we insert the definition of the $S$-transformation $\calS$ in \eqref{eq:def-S-endo-of-L}; step 5 is \eqref{eq:S^2}.

In this form it is evident that $\eta_X$ is the universal dinatural transformation $i_X$ from \Cref{prop:lexcoendhom}, after identifying the state spaces of the TFT with the Hom spaces of $\calc$. This finishes the proof of \Cref{prop:gluingfunc}. \qed

Altogether we now have:
\begin{theorem}\label{thm:chiralmf}
    For every modular tensor category $\calc$, the 3d TFT
    \begin{align}
    \widehat{\rmV}_{\calc} \colon \widehat{\bord}_{3,2}^{\chi}(\calc) \to \vect
    \end{align}
    of \autocite{DGGPR19} induces a chiral modular functor
    \begin{align}
    \mathrm{Bl}_\calc^{\chi}\colon\bord_{2+\epsilon,2,1}^{\mathrm{\chi}} \to \cat{P}\mathrm{rof}_{\mathbbm{k}}^{\coend \mathrm{ex}}.
    \end{align}
\end{theorem}
Moreover going through the construction carefully we immediately find:
\begin{corollary}\label{cor:modfunciso}
    Let $\calc$ and $\calD$ be modular tensor categories which are equivalent as finite ribbon categories. Then the ribbon equivalence $\calc \simeq \calD$ induces an isomorphism $\mathrm{Bl}_\calc^{\chi} \cong \mathrm{Bl}_\calD^{\chi}$ of symmetric monoidal $2$-functors.\footnote{For the definition of a $2$-natural transformation see \autocite[Sec.\,4.2]{JY20twodcategories}.}
\end{corollary}
\begin{remark}\label{rem:mainthm}
\begin{enumerate}[(1)]
    \item Our construction is related to the one of \autocite{DeRenzi2021nonssext2} as both use the $3$d TFTs of \autocite{DGGPR19} to construct symmetric monoidal $2$-functors. However both the source and target $2$-categories differ. Moreover, in \autocite{DeRenzi2021nonssext2} functoriality is satisfied automatically and the non-trivial step in the construction is proving monoidality. In our case it is exactly the other way around with monoidality being built in and functoriality being non-trivial.
    \item In \autocite{Lyu95invproj} Lyubashenko constructed a modular functor, albeit in a different formulation, from a modular tensor category $\calc$ directly using generators and relations. Nevertheless we can still compare his projective MCG representations and gluing maps to the ones constructed above. The projective MCG representations of Lyubashenko and of the $3d$ TFT $\widehat{\rmV}_\calc$ are isomorphic by the main result of \autocite{DGGPR20}.
    Moreover the gluing morphisms $\eta_X$ from \Cref{prop:gluingfunc} are precisely the gluing morphisms from \autocite[Sec.\,9.2]{Lyu1996ribbon} under the above isomorphism.
    \item An operadic formulation and classification of modular functors is given in \autocite{BW22modfunc}. There it is also shown that Lyubashenko's modular functor, suitably interpreted in their framework, is the essentially unique chiral modular functor that can be constructed from a modular tensor category \autocite[Cor.\,8.3]{BW22modfunc}. 
\end{enumerate}
\end{remark}

\section{Modular functors on surfaces with involution}\label{sec:anomfreemod}

In this concluding section, we will extend the chiral modular functors from the previous section to a larger $2$-category of bordisms with orientation-reversing involutions. To this end, we will discuss two functors, one in each direction, between the chiral bordism $2$-category and the one with orientation reversing involutions. We will then show that pulling back the chiral modular functor associated with a modular tensor category $\calc$ along their composite results in a chiral modular functor which is naturally isomorphic to the one coming from the Drinfeld centre $\mathcal{Z}(\mathcal{C})$. Finally, we will discuss the connection to the $2$-category of open-closed bordisms and the corresponding modular functors of \autocite[Sec.\,2.1]{FSY2023RCFTstring1}. 

\subsection[The 2-category of surfaces with involution]{The \texorpdfstring{$\boldsymbol{2}$}{2}-category of surfaces with involution}\label{ssec:bordinv}
The symmetric monoidal $2$-category
$\bord_{2+\epsilon,2,1}^{\circlearrowleft}$ of two-dimensional \emph{bordisms with involution} 
consists of:
\begin{itemize}
    \item objects: are tuples $(\Gamma ,\tau)$, where $\Gamma$ is a closed one-dimensional manifold, and $\tau$ is an orientation reversing involution of $\Gamma$;
    \item $1$-morphisms: for objects $(\Gamma ,\tau)$ and $(\Gamma ',\tau')$, a $1$-morphism from $(\Gamma ,\tau)$ to $(\Gamma',\tau')$, denoted as $(\Gamma ,\tau) \to (\Gamma',\tau')$,
    is a tuple $(\Sigma,\sigma)$, where $\Sigma$ is a two-dimensional bordism $\Sigma \colon \Gamma  \to \Gamma'$, and $\sigma$ is an orientation reversing involution of $\Sigma$ which restricts to $\tau$ ($\tau'$) on $\Gamma$ (resp. $\Gamma'$);
    \item $2$-morphisms: for $1$-morphisms $(\Sigma,\sigma) \colon (\Gamma,\tau) \to (\Gamma',\tau')$ and 
    $(\Sigma,\sigma) \colon (\Gamma,\tau) \to (\Gamma',\tau')$ a $2$-morphism $[f] \colon (\Sigma,\lambda) \Rightarrow (\Sigma',\lambda')$ is the isotopy class $[f]$ of a diffeomorphism $f \colon \Sigma \to \Sigma'$ which is compatible with the boundary parametrisations and commutes with the involutions;
    \item horizontal composition: for $1$-morphisms $(\Sigma,\sigma) \colon (\Gamma ,\tau) \to (\Gamma',\tau')$ and \\ 
    $(\Sigma',\sigma') \colon (\Gamma',\tau') \to (\Gamma'',\tau'')$, the horizontal composition $(\Sigma',\sigma') \diamond (\Sigma,\sigma) \colon (\Gamma,\tau) \to (\Gamma'',\tau'')$ is given by $(\Sigma' \sqcup_{\Gamma'} \Sigma,\sigma' \sqcup_{\Gamma'} \sigma)$;
    \item identity $1$-morphism: for an object $(\Gamma,\tau)$ the identity $1$-morphism $\id_{(\Gamma,\tau)} \colon (\Gamma,\tau) \to (\Gamma,\tau)$ is $(\Gamma\times I, \tau \times \id_I)$;
\end{itemize}
The rest of the structure is defined analogous to $\bord_{2+\epsilon,2,1}^{\chi}$, see \Cref{ssec:chiral-bord-prof} for details.

We will now discuss a functor from $\bord^{\circlearrowleft}_{2+\epsilon,2,1}$ to $\bord^{\chi}_{2+\epsilon,2,1}$ which forgets most of the data contained in the involutions. 
For this let $(\Sigma,\sigma)$ be a $1$-morphism in $\bord^{\circlearrowleft}_{2+\epsilon,2,1}$ and denote with $\lambda_\sigma$ the eigenspace of the induced map $\sigma_* \colon H_1(\Sigma,\R) \to H_1(\Sigma,\R)$ for the eigenvalue $-1$. By \autocite[Lem.\,3.5]{FFFS2002} $\lambda_\sigma$ is a Lagrangian subspace of $H_1(\Sigma,\R)$. 

\begin{lemma}
    The assignment
\begin{equation}
    \begin{aligned}
        U\colon \bord^{\circlearrowleft}_{2+\epsilon,2,1} &\to \bord^{\chi}_{2+\epsilon,2,1}\\
        (\Gamma ,\tau) &\mapsto \Gamma \\
        (\Sigma,\sigma) &\mapsto (\Sigma,\lambda_\sigma) \\
        [f] &\mapsto ([f],0)
    \end{aligned}
\end{equation}
defines a symmetric monoidal functor.
\end{lemma}
\begin{proof}
     Functoriality for horizontal compositions as well as symmetric monoidality are clear by definition. With this particular choice of Lagrangian subspaces the choice of $0$ as signature defect is preserved under composition of $2$-morphisms by \autocite[Prop.\,3.6]{FFFS2002}, which in turn guarantees the functoriality with respect to  vertical composition.
\end{proof}

There is also a symmetric monoidal $2$-functor in the converse direction called the \emph{orientation double} functor, or just \emph{double} functor constructed as follows: 
\begin{itemize}
    \item 
For an object $\Gamma \in \bord^{\chi}_{2+\epsilon,2,1}$ we define 
\begin{align}
    \widehat{\Gamma } = \Gamma  \sqcup -\Gamma .
\end{align}
This is a closed oriented manifold which naturally carries an orientation reversing involution $\tau_\Gamma \colon \widehat{\Gamma } \to \widehat{\Gamma }$.

\item
For a $1$-morphism 
$(\Sigma,\lambda) \colon \Gamma \to \Gamma'$ we use the surface $\Sigma$ to 
define
\begin{align}
    \widehat{\Sigma} = \Sigma \sqcup -\Sigma.
\end{align}
This is a two dimensional manifold with boundary
$\partial \widehat{\Sigma} = \partial(\Sigma) \sqcup \partial(-\Sigma)$. Thus we can use the boundary parametrisation of $\Sigma \colon \Gamma \to \Gamma'$ to equip $\widehat{\Sigma}$ with the structure of a bordism from $\widehat{\Gamma}$ from $\widehat{\Gamma'}$. Moreover $\widehat{\Sigma}$ again comes with an orientation reversing involution $\sigma_{\Sigma} \colon \widehat{\Sigma} \to \widehat{\Sigma}$. It is easy to check that $\sigma_\Sigma$ is compatible with the boundary parametrisations.

\item For a 2-morphism $F = ([f],n) \colon (\Sigma,\lambda) \to (\Sigma',\lambda')$ we set $\widehat{F} = [f\sqcup -f]$.
\end{itemize}

The following lemma is now clear:
\begin{lemma}
    The assignment
    \begin{equation}
    \begin{aligned}
        \widehat{(-)}\colon \bord^{\chi}_{2+\epsilon,2,1} &\to \bord^{\circlearrowleft}_{2+\epsilon,2,1}
        \\
        \Gamma &\mapsto (\widehat{\Gamma},\tau_\Gamma) \\
        (\Sigma,\lambda) &\mapsto (\widehat{\Sigma},\sigma_\Sigma) \\
        ([f],n) &\mapsto [f\sqcup -f]
    \end{aligned}
    \end{equation}
    defines a symmetric monoidal $2$-functor
\end{lemma}

\subsection{Anomaly free modular functors and the Drinfeld centre }
If we pullback a chiral modular functor along $U$, the corresponding mapping class group representations will no longer be projective, i.e.\ the modular functor will be anomaly free. Moreover if we further pullback along $\widehat{(-)}$ we obtain a new chiral modular functor. For the chiral modular functor $\mathrm{Bl}_\calc^{\chi}$ from \Cref{thm:chiralmf} this gives a geometric relation between $\calc$ and its Drinfeld centre $\calZ(\calc)$. 

First recall that the \emph{Drinfeld centre} of a monoidal category $\calA$ is the braided monoidal category $\calZ(\calA)$ with objects pairs $(X,\gamma_X)$, where $X \in \calA$, and $\gamma_X \colon X \otimes - \Rightarrow - \otimes X$ is a natural isomorphism, called a \emph{half braiding}, satisfying a hexagon type axiom \autocite[Def.\,7.13.1]{EGNO}. 
A key property of the Drinfeld centre is that for suitable $\calA$ it is modular:

\begin{proposition}[{\autocite[Thm.\,5.11]{Shimizu17ribbondrinf}}]\label{prop:drinmod}
Let $\calA$ be finite tensor category satisfying a sphericality condition.\footnote{
The sphericality condition is the one given in \autocite[Def.\,3.5.2]{DSS13tensorcats}, not the equality of left and right trace. If $\calA$ is in addition ribbon and is itself modular, it automatically satisfies this sphericality condition. To see this, combine \autocite[Thm.\,1.3]{SS21modtracenaka} and \autocite[Cor.\,4.7]{gr2020}.}
 Then $\calZ(\calA)$ is a modular tensor category.
\end{proposition}
A generalisation of this result to pivotal tensor categories can be found in \autocite[Cor.\,2.14]{MW22drinGV}.

There is a natural forgetful functor $U \colon \calZ(\calA) \to \calA$ which sends a pair $(X,\gamma_X)$ to $X$. Moreover, for any braided tensor category $\calB$ with braiding $\beta$, there is a natural braided functor $\calB \to \calZ(\calB)$ sending an object $X \in \calB$ to $(X,\beta_{X,-}) \in \calZ(\calB)$, as well as a corresponding braided functor $\overline{\calB}\to \calZ(\calB)$, where $\overline{\calB}$ is the same monoidal category as $\calB$ but with reversed braiding. These functors extend to a braided functor 
\begin{align}
    \calB \boxtimes \overline{\calB} \to \calZ(\calB).
\end{align}
If $\calB$ is ribbon, $\calZ(\calB)$ is ribbon as well and the functor discussed above is a ribbon functor.
\begin{proposition}[{\autocite[Thm.\,1.1]{SHIMIZU2019nondeg}}]\label{prop:modularchar}
    Let $\calB$ be a finite ribbon tensor category. Then $\calB$ is modular if and only if the canonical ribbon functor
    \begin{align}
        \calB \boxtimes \overline{\calB} \to \calZ(\calB)
    \end{align}
    is an equivalence.
\end{proposition}

\begin{proposition}{\label{prop:condi}}
    Let $\calc$ be a modular tensor category. There exists a symmetric monoidal $2$-natural isomorphism filling the following diagram of symmetric monoidal $2$-functors
\begin{align}
\begin{tikzcd}[ampersand replacement=\&]
	{\bord_{2+\epsilon,2,1}^{\mathrm{\chi}}} \& {\bord^{\circlearrowleft}_{2+\epsilon,2,1}} \& {\bord_{2+\epsilon,2,1}^{\mathrm{\chi}}} \\
	\\
	\&\& {\cat{P}\mathrm{rof}_{\mathbbm{k}}^{\coend \mathrm{ex}}}
	\arrow[""{name=0, anchor=center, inner sep=0}, "{\mathrm{Bl}_{\calZ(\calc)}^{\chi}}"', from=1-1, to=3-3]
	\arrow["{\widehat{(-)}}", from=1-1, to=1-2]
	\arrow["U", from=1-2, to=1-3]
	\arrow["{\mathrm{Bl}_\calc^{\chi}}", from=1-3, to=3-3]
	\arrow["\cong"', shorten >=8pt, Rightarrow, from=1-3, to=0]
\end{tikzcd}
\quad .
\end{align}
\end{proposition}
\begin{proof} 
    Let us denote the composition 
    \begin{align}
        \bord^{\chi}_{2+\epsilon,2,1} \stackrel{\widehat{(-)}}{\to} \bord^{\circlearrowleft}_{2+\epsilon,2,1} \stackrel{U}{\to} \bord^{\chi}_{2+\epsilon,2,1} \stackrel{\mathrm{Bl}_\calc^{\chi}}{\to} \cat{P}\mathrm{rof}^{\coend\mathrm{ex}_\kk} 
    \end{align}
    by $\widehat{\mathrm{Bl}}_\calc$. First note that the equivalence $\calc \boxtimes \overline{\calc} \simeq \calZ(\calc)$  induces an isomorphism $\mathrm{Bl}_{\calc \boxtimes \overline{\calc}}^\chi \cong \mathrm{Bl}_{\calZ(\calc)}^\chi$ by \Cref{cor:modfunciso}. It thus suffices to show that 
    \begin{equation}
        \widehat{\mathrm{Bl}}_\calc \cong \mathrm{Bl}_{\calc \boxtimes \overline{\calc}}^\chi. 
    \end{equation} 
    On the level of linear categories we have $\calc = \overline{\calc}$, thus we immediately get $\widehat{\mathrm{Bl}}_\calc(\Gamma) = (\calc \boxtimes \calc)^{\pi_0(\Gamma)} = (\calc \boxtimes \overline{\calc})^{\pi_0(\Gamma)} = \mathrm{Bl}_{\calc \boxtimes \overline{\calc}}^{\chi}(\Gamma)$ for any object $\Gamma \in \bord_{2+\epsilon,2,1}^{\chi}$. On the level of $1$- and $2$-morphisms the result is a direct consequence of the behaviour of the TFT discussed in \Cref{ssec:tftorideligne}.
   \end{proof}   
\begin{remark}
    Note that in a string-net construction the chiral modular functor of a Drinfeld centre is automatically anomaly free because the mapping class group acts geometrically on the string-nets \autocite[Cor.\,8.7]{MSWY23drinmodfunc}.
\end{remark}

\subsection{Open-closed bordisms and full modular functors}
    Finally, let us discuss how the modular functors from above are related to the so-called \emph{full modular functors} of \autocite[Sec.\,2.1]{FSY2023RCFTstring1}. The symmetric monoidal $2$-category $\bord_{2+\epsilon,2,1}^{\oc}$ of two dimensional oriented \emph{open-closed bordisms} is defined as a kind of categorification of the category of open-closed bordisms defined in \autocite[Sec.\,3]{LP2008octqfts}, see \autocite{BCR2006ocbordcat} for a detailed account. The basic idea is to also allow compact one dimensional manifolds with boundaries as objects, consequentially bordisms between them need to have an underlying manifold with corners, more precisely a $\langle2\rangle$-manifold. Recall here that a two dimensional $\langle2\rangle$-manifold $\Sigma$ is a $2$-dimensional compact manifold with corners together with a decomposition $\partial \Sigma = \partial^{\mathrm{g}} \Sigma \sqcup \partial^{\mathrm{f}} \Sigma$ of its boundary into a \emph{gluing} boundary
    $\partial^{\mathrm{g}} \Sigma$ and a \emph{free} boundary $\partial^{\mathrm{f}} \Sigma$ such that $\partial^{\mathrm{g}} \Sigma \cap \partial^{\mathrm{f}} \Sigma$ consists of the corner points of $\Sigma$.\footnote{The gluing and free boundaries are called black and coloured boundaries, respectively, in \autocite[Sec.\,3]{LP2008octqfts}.} We define $\bord_{2+\epsilon,2,1}^{\oc}$ as follows:
    \begin{itemize}
    \item objects: compact one-dimensional manifolds, i.e.\ finite disjoint unions of the standard interval $I = [0,1]$ and the unit circle $S^1$;
    \item $1$-morphisms: for objects $\Upsilon$ and $\Upsilon'$, a $1$-morphism from $\Upsilon$ to $\Upsilon'$, denoted as $\Upsilon \to \Upsilon'$,
    is a two-dimensional open-closed bordism $\Sigma \colon \Upsilon \to \Upsilon'$, i.e.\ a $\langle2\rangle$-manifold together with a parametrisation of its gluing boundary $\partial^{\mathrm{gl}}\Sigma \cong -\Upsilon \sqcup \Upsilon'$;
    \item $2$-morphisms: for $1$-morphisms $\Sigma \colon \Upsilon \to \Upsilon'$ and $\Sigma' \colon \Upsilon \to \Upsilon'$ a $2$-morphism $[f] \colon \Sigma \Rightarrow \Sigma'$ is the isotopy class of a diffeomorphism $f \colon \Sigma \to \Sigma'$ of $\langle2\rangle$-manifolds which is compatible with the boundary parametrisations.
\end{itemize}
The rest of the structure is defined analogously to $\bord_{2+\epsilon,2,1}^{\chi}$, see \Cref{ssec:chiral-bord-prof} for details.

An example of a $1$-morphism $W$ from $S^1$ to $I$ is given by the so-called whistle bordism 
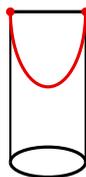
\[
    \begin{tikzpicture}[baseline=1cm,rotate=180]
        \draw[bordcolor-C,very thick] (1,0) -- (0,0);
        \filldraw[bordcolor-A] (0,0) circle (0.5mm);
        \filldraw[bordcolor-A] (1,0) circle (0.5mm);
        \draw[bordcolor-C,very thick] (1,2) arc (0:180:0.5 and 0.2);
        \draw[bordcolor-C,thick] (0,0) -- (0,2);
        \draw[bordcolor-C,thick] (1,0) -- (1,2);
        \draw[bordcolor-C,very thick] (1,2) arc (0:-180:0.5 and 0.2);
        \draw[bordcolor-A,very thick] (1,0) arc (0:180:0.5 and 1);
    \end{tikzpicture}
\]
where the red line indicates the free boundary, while the thick black ones correspond to the gluing boundary. 

Let us now discuss the connection between $\bord_{2+\epsilon,2,1}^{\circlearrowleft}$ and $\bord_{2+\epsilon,2,1}^{\oc}$. 
For this we first extend the orientation double functor to 
$\bord_{2+\epsilon,2,1}^{\oc}$ as follows:
\begin{itemize}
    \item For an object $\Upsilon \in \bord^{\oc}_{2+\epsilon,2,1}$ we define 
\begin{align}
    \widehat{\Upsilon} = \Upsilon \sqcup -\Upsilon / \sim \quad \text{with} \quad (p,+) \sim (p,-) \quad \text{for} \quad p \in \partial \Upsilon.
\end{align} 
This is a closed oriented manifold which naturally carries an orientation reversing involution $\tau_\Upsilon \colon \widehat{\Upsilon} \to \widehat{\Upsilon}$ coming from the orientation reversing involution on $\Upsilon \sqcup -\Upsilon$. 
\item For a $1$-morphism $\Sigma \colon \Upsilon \to \Upsilon'$ we define
\begin{align}
    \widehat{\Sigma} = \Sigma \sqcup -\Sigma / \sim \quad \text{with} \quad (p,+) \sim (p,-) \quad \text{for} \quad p \in \partial^f \Sigma.
\end{align}
This is a two dimensional manifold with boundary, moreover we have $\partial \widehat{\Sigma} = \partial^{\mathrm{g}}(\Sigma) \sqcup \partial^{\mathrm{g}}(-\Sigma) / \sim$. 
Thus we can use the boundary parametrisation of $\Sigma \colon \Upsilon \to \Upsilon'$ to equip $\widehat{\Sigma}$ with the structure of a bordism from $\widehat{\Upsilon}$ to $\widehat{\Upsilon'}$. Moreover $\widehat{\Sigma}$ again comes with an orientation reversing involution $\sigma_\Sigma\colon \widehat{\Sigma} \to \widehat{\Sigma}$. It is easy to check that $\sigma_\Sigma$ is compatible with the boundary parametrisations.

\item For a $2$-morphism $[f] \colon \Sigma \to \Sigma'$ we can extend the underlying diffeomorphism $f$ to a homeomorphism $\widehat{f} \colon \widehat{\Sigma} \to \widehat{\Sigma'}$ by the universal property of the quotient. Moreover $\widehat{f}$ commutes with the induced involutions by definition. Since the smooth and the topological mapping class group are isomorphic \autocite[Sec.\,2.1]{FM2011pMCG} we can set $\widehat{[f]} := [\widehat{f}]$.
\end{itemize}
\begin{lemma}
    The assignment
    \begin{equation}
    \begin{aligned}
        \widehat{(-)}\colon \bord^{\oc}_{2+\epsilon,2,1} &\to \bord^{\circlearrowleft}_{2+\epsilon,2,1}
        \\
        \Upsilon &\mapsto (\widehat{\Upsilon},\tau_\Upsilon) 
        \\
\Sigma &\mapsto (\widehat{\Sigma},\sigma_\Sigma) \\
        [f] &\mapsto [\widehat{f}\,]
    \end{aligned}
    \end{equation}
    defines a symmetric monoidal $2$-functor. 
\end{lemma}
We employ the same notation for the orientation double functor here as in the previous section since there is a commutative diagram
\begin{equation}\label{diag:double}
    \begin{tikzcd}[ampersand replacement=\&]
	{\bord_{2+\epsilon,2,1}^{\chi}} \&\& {\bord_{2+\epsilon,2,1}^{\circlearrowleft}} \\
	\& {\bord_{2+\epsilon,2,1}^{\oc}}
	\arrow["{\widehat{(-)}}", from=1-1, to=1-3]
	\arrow[from=1-1, to=2-2]
	\arrow["{\widehat{(-)}}"', from=2-2, to=1-3]
\end{tikzcd}
\end{equation}
where $\bord_{2+\epsilon,2,1}^{\chi} \to \bord_{2+\epsilon,2,1}^{\oc}$ forgets the Lagrangian subspace and signature defect.
\begin{remark}
Note that $\widehat{(-)} \colon \bord^{\oc}_{2+\epsilon,2,1} \to \bord^{\circlearrowleft}_{2+\epsilon,2,1}$ is not locally essentially surjective as it only gives surfaces whose quotient by the involution is again orientable, e.g.\ $S^2$ with antipodal involution is not in its essential image as the quotient space is the crosscap.
\end{remark}
Pulling back the anomaly free modular functor from the previous section we obtain a \emph{full modular functor} in the sense of \autocite[Sec.\,2.1]{FSY2023RCFTstring1}
\begin{align}
    \mathrm{Bl}_{\calc}^{\mathrm{full}}\colon\bord^{\oc}_{2+\epsilon,2,1} \stackrel{\widehat{(-)}}{\longrightarrow} \bord^{\circlearrowleft}_{2+\epsilon,2,1} \stackrel{U}{\longrightarrow} \bord^{\chi}_{2+\epsilon,2,1} \stackrel{\mathrm{Bl}^{\chi}_{\calc}}{\longrightarrow}  \cat{P}\mathrm{rof}_{\mathbbm{k}}^{\coend \mathrm{ex}}.
\end{align}

As an illustrative application, we will now use $\mathrm{Bl}_{\calc}^{\mathrm{full}}$ to give a topological proof for the following result of \autocite[Sec.\,5.6]{BV12monadcenter} on the monadicity of the Drinfeld centre, in the special case where one already starts from a modular tensor category
(see also \autocite[Ch.\,9]{TV} for a textbook account):
\begin{proposition}\label{prop:freeforgetadj}{\autocite[Cor.\,5.14]{BV12monadcenter}}
    Let $\calc$ be a modular tensor category, then  the forgetful functor $U \colon \calZ(\calc) \to \calc$ has a two sided adjoint $F \colon \calc \to \calZ(\calc)$. Moreover the corresponding monad $U \circ F$ is naturally isomorphic to the central monad $Z$ on $\calc$
    \begin{align}
        (U \circ F)(-) \cong Z (-) := \int^{X \in \calc} X^* \otimes - \otimes X.
    \end{align}
\end{proposition}

\begin{proof}
For the first part of the statement we will study the whistle bordism $W$ from above. For the interval $I$ we directly compute $\mathrm{Bl}_{\calc}^{\mathrm{full}}(I) = \calc.$
While for the unit circle $S^1$ we get $\mathrm{Bl}_{\calc}^{\mathrm{full}}(S^1) \simeq \calZ(\calc)$ by the same argument as in the proof of \Cref{prop:condi}. The double of the whistle bordism $W \colon S^1 \to I$ is given by a pair of pants bordism $\widehat{W} \colon S^1 \sqcup - S^1 \to S^1 $. From this we get $\mathrm{bl}_{\calc}^{\chi}(\widehat{W}) \cong \Hom_{\calc}(-\otimes -,-)$. 
By carefully tracking how the involved functors behave under the equivalence $\calc\boxtimes\overline{\calc} \simeq \calZ(\calc)$
we find that 
\begin{align}
    \mathrm{Bl}^{\mathrm{full}}_{\calc}(W) \cong \Hom_{\calc}(U(-),-) \colon \calZ(\calc) \slashedrightarrow \calc,
\end{align} 
where $U \colon \calZ(\calc) \to \calc$ is the forgetful functor.

Now note that there is also a whistle bordism in the opposite direction $\widetilde{W} \colon I \to S^1$ in $\bord^{\oc}_{2+\epsilon,2,1}$.
For this an analogous computation gives
\begin{align}
    \mathrm{Bl}^{\mathrm{full}}_{\calc}(\widetilde{W}) 
    \cong \Hom_{\calc}(-,U(-)) \colon \calc \slashedrightarrow \calZ(\calc).
\end{align} 
In particular the profunctors $\mathrm{Bl}^{\mathrm{full}}_{\calc}(W)$ and $\mathrm{Bl}^{\mathrm{full}}_{\calc}(\widetilde{W})$ are adjoint $1$-morphisms in $\cat{P}\mathrm{rof}_{\mathbbm{k}}^{\coend \mathrm{ex}}$ \autocite[Rem.\,5.2.1]{Loregian2021coendcalc}. 

Moreover since all profunctors we consider are left exact, they are representable. From this we get a functor $F \colon \calc \to \calZ(\calc)$ such that $\mathrm{Bl}^{\mathrm{full}}_{\calc}(\widetilde{W}) \cong \Hom_{\calZ(\calc)}(F(-),-)$. In particular $F \colon \calc \to \calZ(\calc)$ is a left-adjoint of $U \colon \calZ(\calc) \to \calc$ by definition. By an analogous argument $F$ is also a right-adjoint. This proves the first part of \Cref{prop:freeforgetadj}.

For the second part recall that there is a diffeomorphism between $W \sqcup_{S^1} \widetilde{W}$ and the composition between two flat pairs of pants and the braiding, see e.g.\ \autocite[Sec.\,3.2]{carqueville16defect}:
\begin{equation}
\begin{tikzpicture}[scale=0.5,baseline=1.5cm]
        \draw[bordcolor-C,thick] (1,0) -- (0,0);
        \draw[bordcolor-C,thick] (0,0) -- (0,6);
        \draw[bordcolor-C,thick] (1,0) -- (1,6);
        \draw[bordcolor-C,thick] (0,6) -- (1,6);
        \filldraw[bordcolor-A] (0,0) circle (0.5mm);
        \filldraw[bordcolor-A] (1,0) circle (0.5mm);
        \draw[bordcolor-A,very thick] (1,0) arc (0:180:0.5 and 1);
        \filldraw[bordcolor-A] (0,6) circle (0.5mm);
        \filldraw[bordcolor-A] (1,6) circle (0.5mm);
        \draw[bordcolor-A,very thick] (1,6) arc (0:-180:0.5 and 1);
    \end{tikzpicture}
   \, \cong
    \begin{tikzpicture}[scale=0.5,baseline=1.5cm]
        \draw[bordcolor-C,thick] (1,6) -- (2,6);
        \filldraw[bordcolor-A] (1,6) circle (0.5mm);
        \filldraw[bordcolor-A] (2,6) circle (0.5mm);
        \draw[bordcolor-A,very thick] (2,4) arc (0:180:0.5 and 0.6);
        \draw[bordcolor-A,very thick] (0,4) .. controls (0.1,5.2) and (0.9,4.8) .. (1,6);
        \draw[bordcolor-A,very thick] (3,4) .. controls (2.9,5.2) and (2.1,4.8) .. (2,6);
        \draw[bordcolor-C,thick] (1,0) -- (2,0);
        \filldraw[bordcolor-A] (2,0) circle (0.5mm);
        \filldraw[bordcolor-A] (1,0) circle (0.5mm);
        \draw[bordcolor-A,very thick] (2,2) arc (0:-180:0.5 and 0.6);
        \draw[bordcolor-A,very thick] (0,2) .. controls (0.1,0.8) and (0.9,1.2) .. (1,0);
        \draw[bordcolor-A,very thick] (3,2) .. controls (2.9,0.8) and (2.1,1.2) .. (2,0);
        \draw[bordcolor-A,very thick] (0,4) -- (2,2);
        \draw[bordcolor-A,very thick] (1,4) -- (3,2);
        \fill[white,opacity=0.9] (0,2) -- (1,2) -- (3,4) -- (2,4) -- (0,2);
        \draw[bordcolor-A,very thick] (1,2) -- (3,4);
        \draw[bordcolor-A,very thick] (0,2) -- (2,4);
        \filldraw[bordcolor-A] (0,2) circle (0.15mm);
        \filldraw[bordcolor-A] (1,2) circle (0.15mm);
        \filldraw[bordcolor-A] (2,2) circle (0.15mm);
        \filldraw[bordcolor-A] (3,2) circle (0.15mm);
        \filldraw[bordcolor-A] (0,4) circle (0.15mm);
        \filldraw[bordcolor-A] (1,4) circle (0.15mm);
        \filldraw[bordcolor-A] (2,4) circle (0.15mm);
        \filldraw[bordcolor-A] (3,4) circle (0.15mm);
    \end{tikzpicture}
\end{equation}
Applying $\mathrm{Bl}^{\mathrm{full}}_{\calc}$ to the left hand side we immediately get
\begin{equation}
\mathrm{Bl}^{\mathrm{full}}_{\calc}\left(\,
    \begin{tikzpicture}[scale=0.3,baseline=0.7cm]
        \draw[bordcolor-C,thick] (1,0) -- (0,0);
        \draw[bordcolor-C,thick] (0,0) -- (0,6);
        \draw[bordcolor-C,thick] (1,0) -- (1,6);
        \draw[bordcolor-C,thick] (0,6) -- (1,6);
        \filldraw[bordcolor-A] (0,0) circle (0.5mm);
        \filldraw[bordcolor-A] (1,0) circle (0.5mm);
        \draw[bordcolor-A,very thick] (1,0) arc (0:180:0.5 and 1);
        \filldraw[bordcolor-A] (0,6) circle (0.5mm);
        \filldraw[bordcolor-A] (1,6) circle (0.5mm);
        \draw[bordcolor-A,very thick] (1,6) arc (0:-180:0.5 and 1);
    \end{tikzpicture}\, \right)
    \cong \Hom_{\calc}((U \circ F)(-),-)
\end{equation}
 by functoriality. For the right hand side first note that the double of a flat pair of pants is a closed pair of pants, and the double of the braiding in the open sector is the braiding in the closed sector. Composing the resulting profunctors finally leads to 
{\allowdisplaybreaks
\begin{align}
    \mathrm{Bl}^{\mathrm{full}}_{\calc}\left( \begin{tikzpicture}[scale=0.3,baseline=0.7cm]
        \draw[bordcolor-C,thick] (1,6) -- (2,6);
        \filldraw[bordcolor-A] (1,6) circle (0.5mm);
        \filldraw[bordcolor-A] (2,6) circle (0.5mm);
        \draw[bordcolor-A,very thick] (2,4) arc (0:180:0.5 and 0.6);
        \draw[bordcolor-A,very thick] (0,4) .. controls (0.1,5.2) and (0.9,4.8) .. (1,6);
        \draw[bordcolor-A,very thick] (3,4) .. controls (2.9,5.2) and (2.1,4.8) .. (2,6);
        \draw[bordcolor-C,thick] (1,0) -- (2,0);
        \filldraw[bordcolor-A] (2,0) circle (0.5mm);
        \filldraw[bordcolor-A] (1,0) circle (0.5mm);
        \draw[bordcolor-A,very thick] (2,2) arc (0:-180:0.5 and 0.6);
        \draw[bordcolor-A,very thick] (0,2) .. controls (0.1,0.8) and (0.9,1.2) .. (1,0);
        \draw[bordcolor-A,very thick] (3,2) .. controls (2.9,0.8) and (2.1,1.2) .. (2,0);
        \draw[bordcolor-A,very thick] (0,4) -- (2,2);
        \draw[bordcolor-A,very thick] (1,4) -- (3,2);
        \fill[white,opacity=0.9] (0,2) -- (1,2) -- (3,4) -- (2,4) -- (0,2);
        \draw[bordcolor-A,very thick] (1,2) -- (3,4);
        \draw[bordcolor-A,very thick] (0,2) -- (2,4);
        \filldraw[bordcolor-A] (0,2) circle (0.15mm);
        \filldraw[bordcolor-A] (1,2) circle (0.15mm);
        \filldraw[bordcolor-A] (2,2) circle (0.15mm);
        \filldraw[bordcolor-A] (3,2) circle (0.15mm);
        \filldraw[bordcolor-A] (0,4) circle (0.15mm);
        \filldraw[bordcolor-A] (1,4) circle (0.15mm);
        \filldraw[bordcolor-A] (2,4) circle (0.15mm);
        \filldraw[bordcolor-A] (3,4) circle (0.15mm);
    \end{tikzpicture}\right) 
    &\cong 
    \mathrm{Bl}^{\mathrm{full}}_{\calc}\left( \,
    \begin{tikzpicture}[scale=0.3,baseline=1.3cm]
        \draw[bordcolor-C,thick] (1,6) -- (2,6);
        \filldraw[bordcolor-A] (1,6) circle (0.5mm);
        \filldraw[bordcolor-A] (2,6) circle (0.5mm);
        \draw[bordcolor-A,very thick] (2,4) arc (0:180:0.5 and 0.6);
        \draw[bordcolor-A,very thick] (0,4) .. controls (0.1,5.2) and (0.9,4.8) .. (1,6);
        \draw[bordcolor-A,very thick] (3,4) .. controls (2.9,5.2) and (2.1,4.8) .. (2,6);
        \draw[bordcolor-C,thick] (0,4) -- (1,4);
        \filldraw[bordcolor-A] (0,4) circle (0.5mm);
        \filldraw[bordcolor-A] (1,4) circle (0.5mm);
        \draw[bordcolor-C,thick] (2,4) -- (3,4);
        \filldraw[bordcolor-A] (2,4) circle (0.5mm);
        \filldraw[bordcolor-A] (3,4) circle (0.5mm);
    \end{tikzpicture}\,\right) 
    \diamond \mathrm{Bl}^{\mathrm{full}}_{\calc}\left( \,
    \begin{tikzpicture}[scale=0.3,baseline=0.4cm]
        \draw[bordcolor-C,thick] (1,0) -- (2,0);
        \filldraw[bordcolor-A] (2,0) circle (0.5mm);
        \filldraw[bordcolor-A] (1,0) circle (0.5mm);
        \draw[bordcolor-A,very thick] (2,2) arc (0:-180:0.5 and 0.6);
        \draw[bordcolor-A,very thick] (0,2) .. controls (0.1,0.8) and (0.9,1.2) .. (1,0);
        \draw[bordcolor-A,very thick] (3,2) .. controls (2.9,0.8) and (2.1,1.2) .. (2,0);
        \draw[bordcolor-C,thick] (0,4) -- (1,4);
        \filldraw[bordcolor-A] (0,4) circle (0.5mm);
        \filldraw[bordcolor-A] (1,4) circle (0.5mm);
        \filldraw[bordcolor-A] (2,4) circle (0.5mm);
        \filldraw[bordcolor-A] (3,4) circle (0.5mm);
        \draw[bordcolor-A,very thick] (0,4) -- (2,2);
        \draw[bordcolor-A,very thick] (1,4) -- (3,2);
        \fill[white,opacity=0.9] (0,2) -- (1,2) -- (3,4) -- (2,4) -- (0,2);
        \draw[bordcolor-A,very thick] (1,2) -- (3,4);
        \draw[bordcolor-A,very thick] (0,2) -- (2,4);
        \filldraw[bordcolor-A] (0,2) circle (0.15mm);
        \filldraw[bordcolor-A] (1,2) circle (0.15mm);
        \filldraw[bordcolor-A] (2,2) circle (0.15mm);
        \filldraw[bordcolor-A] (3,2) circle (0.15mm);
        \filldraw[bordcolor-A] (0,4) circle (0.15mm);
        \filldraw[bordcolor-A] (1,4) circle (0.15mm);
        \filldraw[bordcolor-A] (2,4) circle (0.15mm);
        \filldraw[bordcolor-A] (3,4) circle (0.15mm);
        \draw[bordcolor-C,thick] (2,4) -- (3,4);
    \end{tikzpicture}\,\right) 
    \nonumber
    \\
    &\cong \oint^{X\boxtimes Y \in \calc\boxtimes\calc} \Hom_{\calc}\left(X \otimes Y, - \right) \otimes_\mathbb{k} \Hom_{\calc}\left( - ,Y \otimes X\right) 
    \nonumber
    \\
    &\cong \oint^{X\in \calc} \oint^{Y \in \calc} \Hom_{\calc}\left(X \otimes Y, - \right) \otimes_\mathbb{k} \Hom_{\calc}\left( - ,Y \otimes X\right) 
    \nonumber
    \\
    &\cong \oint^{Y\in \calc} \oint^{X \in \calc} \Hom_{\calc}\left(X \otimes Y, - \right) \otimes_\mathbb{k} \Hom_{\calc}\left( Y^*\otimes- ,X\right) 
    \nonumber
    \\
    &\cong \oint^{Y\in \calc} \Hom_{\calc}\left(Y^* \otimes - \otimes Y, - \right)  
    \nonumber
    \\
    &\cong \Hom_{\calc}\left(\int^{Y \in \calc}Y^* \otimes - \otimes Y ,-\right).
    \end{align}}
Here, the first and second isomorphism come from functoriality and definition of $\mathrm{Bl}^{\mathrm{full}}_{\calc}$, respectively. The third step uses \autocite[Lem.\,3.2]{FSS19eilenbergwatts} in the setting were all functors are left exact. The fourth isomorphism uses the Fubini theorem for left exact coends \autocite[Thm.\,B.2]{Lyu1996ribbon} as well as rigidity of $\calc$. The last two isomorphisms are the Yoneda lemma and \Cref{prop:lexcoendhom}, respectively. By functoriality of $\mathrm{Bl}^{\mathrm{full}}_{\calc}$ we thus obtain an isomorphism of profunctors 
\begin{equation}
    \Hom_{\calc}((U \circ F)(-),-) \cong \Hom_{\calc}\left(\int^{Y \in \calc}Y^* \otimes - \otimes Y ,-\right),
\end{equation}
which by the Yoneda lemma implies the second part of \Cref{prop:freeforgetadj}. 
\end{proof}

\addcontentsline{toc}{section}{References}
\renewcommand{\bibfont}{\small}
\raggedright
\printbibliography
\end{document}